\documentclass[final,10pt,a4paper]{elsarticle}

\usepackage{amssymb}
\usepackage{amsmath}

\usepackage{amsthm}
\theoremstyle{plain}
\newtheorem{theorem}{Theorem}[section]
\newtheorem{lemma}{Lemma}[section]

\theoremstyle{definition}
\newtheorem{remark}{Remark}[section]

\usepackage{lineno}
\modulolinenumbers[5]

\usepackage{tikz}
\usetikzlibrary{shapes,arrows}
\usetikzlibrary{patterns}

\usepackage{pgfplots}
\pgfplotsset{compat=1.13}
\usepgfplotslibrary{fillbetween}

\definecolor{PrimalOrange}{RGB}{255,153,85}
\definecolor{DualBlue}{RGB}{0,255,255}
\definecolor{SlabGreen}{RGB}{170,212,0}
\definecolor{HSUred}{RGB}{197,0,66}
\definecolor{mountainmeadow}{rgb}{0.19, 0.73, 0.56}
\definecolor{navyblue}{rgb}{0.0, 0.0, 0.5}

\usepackage{hyperref}
\usepackage{xcolor}

\newcommand{\doi}[1]{\href{https://doi.org/#1}{doi:#1}}

\newcommand{\concentration}{{u}}
\newcommand{\convection}{{\boldsymbol v}}
\newcommand{\pressure}{{p}}
\newcommand{\dualz}{{z}}
\newcommand{\density}{{\rho}}
\newcommand{\viscosity}{{\nu}}
\newcommand{\transportforce}{{g}}
\newcommand{\stokesforce}{{\boldsymbol f}}

\numberwithin{equation}{section}

\usepackage{geometry}
\journal{Computers \& Mathematics with Applications}

\begin{document}

\begin{frontmatter}


\title{Flexible goal-oriented adaptivity for higher-order 
space-time discretizations of transport problems with coupled flow}

\author[hsuhh1]{Markus Bause}
\author[hsuhh1]{Marius Paul Bruchh\"auser$^{*,}$}
\author[hsuhh1]{Uwe K\"ocher}

\address[hsuhh1]{Helmut Schmidt University,
Faculty of Mechanical Engineering,
Holstenhofweg 85, 22043 Hamburg\\
\textnormal{\texttt{bause@hsu-hh.de}},
\textnormal{\texttt{bruchhaeuser@hsu-hh.de}}
($^*\!\!$ corresponding author),\\
\textnormal{\texttt{koecher@hsu-hh.de}}.}

\begin{abstract}
In this work, a flexible higher-order space-time adaptive finite element approximation
of convection-dominated transport with coupled fluid flow is developed
and studied.
Convection-dominated transport is a challenging subproblem in poromechanics
in which coupled transport with flow, chemical reaction and
mechanical response in porous media is considered.
Key ingredients are
the arbitrary degree discontinuous Galerkin time discretization
of the primal and dual problems for the Dual Weighted Residual (DWR) approach,
an a posteriori error estimation for the transport problem coupled with flow 
and its implementation in an advanced software architecture.
The error estimate allows the separation of the temporal and spatial discretization
error contributions which facilitates the simultaneous adjustment of the time
and space mesh.
The performance of the approach and its software implementation is studied
by numerical convergence examples as well as an example of physical interest for
convection-dominated cases.
\end{abstract}


\end{frontmatter}



\section{Introduction}
\label{sec:1:introduction}

Coupling convection-dominated transport with flow, chemical reaction and
mechanical response in porous media with or without fracture development
has attracted researcher's interest for many years and
receives increasing interest currently; cf.
\cite{BruchhaeuserLM07,BruchhaeuserB41,BruchhaeuserA89,BruchhaeuserS00,
BruchhaeuserAKDSW16,BruchhaeuserGM17,BruchhaeuserW16,BruchhaeuserOKL19}.
The strong multi-physics character of the models of coupled transport, flow,
deformation and fracture propagation yields, for instance, different time scales
such that iterative coupling methods and multi-rate time discretizations
are broadly used and studied; cf.
\cite{BruchhaeuserAKDSW16,BruchhaeuserGM18,BruchhaeuserGWH16,BruchhaeuserMW13}.
The governing system of a single-phase convection-diffusion transport reads as
\begin{equation}
\label{eq:1:transport}
\partial_t( \phi\, \density(\pressure)\, \concentration )
+ \nabla \cdot ( \density(\pressure)\, \concentration\, \convection
- \phi\, \density(\pressure)\, \boldsymbol D(\convection)\, \nabla \concentration)
= \transportforce\,,
\end{equation}
with the unknown concentration variable $\concentration$ which is transported
by a convec\-tion-diffusion process with the convection tensor $\convection$ and
the diffusion tensor $\boldsymbol D$. The other variables in Eq.~\eqref{eq:1:transport}
are the porosity $\phi \in (0,1]$, the mass density $\density > 0$,
the fluid pressure $\pressure$ and the source term $\transportforce$ of the
transport equation. Here, we consider a linearized tensor of the type
$\boldsymbol D = \varepsilon\, \boldsymbol I$ with a scalar diffusion parameter
$\varepsilon$.
Dispersion effects are neglectable for convection-dominated transport.
The Navier-Stokes equations of a flow model read as
\begin{displaymath}
\partial_t \convection
- 2\, \widetilde{\viscosity} (\nabla \cdot \boldsymbol \epsilon(\convection) )
+ (\convection \cdot \nabla) \convection + \nabla \pressure = \boldsymbol f\,,
\quad
\nabla \cdot \convection = 0\,,
\end{displaymath}
for the fluid convection $\convection$ and pressure $\pressure$ variables
and the viscosity $\widetilde{\viscosity}$.
This flow model is simplified here to the quasi-static Stokes equations
\begin{equation}
\label{eq:2:stokes}
- \viscosity\, \Delta\, \convection + \nabla \pressure = \stokesforce
\end{equation}
of an incompressible fluid $\nabla \cdot \convection = 0$
due to the assumption of a slow moving viscous fluid. The latter assumption
comes from our focus on a highly time-dynamic transport in a porous media
with very small diffusion relative to the convection.
Thus, our linear transport problem yields a singulary-disturbed linear transport
problem with the typical issue of numerical oscillations.
In poromechanics the quasi-static or dynamic mechanical response model
is coupled very often in an iterative way to transport and flow models;
cf. e.g. \cite{BruchhaeuserMW13} and references therein. For completeness,
we give here the extension of the governing transport model from
Eqs. \eqref{eq:1:transport}-\eqref{eq:2:stokes}
with the quasi-static Biot's equations as
\begin{displaymath}
-\nabla \cdot \boldsymbol \sigma(\boldsymbol u_m)
+ \alpha\, \nabla \pressure = \boldsymbol f_m
\end{displaymath}
for the mechanical deformation $\boldsymbol u_m$ due to fluid and capillary
pressure and time-dependent mechanical forcing changes.
The mechanical stress is denoted by $\boldsymbol \sigma$ and $\alpha$ is the
Biot-Willis model-coupling parameter.
This mechanical equilibrium equation is supported by the additional
Darcy transport system
\begin{displaymath}
\partial_t (c_0\, \pressure + \alpha (\nabla \cdot \boldsymbol u_m))
+ \nabla \cdot \boldsymbol v_D = q\,,\quad
\boldsymbol v_D = \boldsymbol \kappa\,
(-\nabla \pressure + \density_f\, \boldsymbol g)\,,
\end{displaymath}
with a storage coefficient $c_0$, the Darcy velocity $\boldsymbol v_D$,
a source term $q$, a viscosity-permeability tensor $\boldsymbol \kappa$,
the fluid density $\density_f$ and the gravity $\boldsymbol g$.
The latter Darcy transport equation is of the same structure as Eq.
\eqref{eq:1:transport} with $\epsilon_v := \nabla \cdot \boldsymbol u_m$,
$\boldsymbol u_m$ and $\boldsymbol v_D$ depending on the pressure $\pressure$,
and the first equation of the Stokes system \eqref{eq:2:stokes} is similar
to the quasi-static equilibrium equation of the Biot's system.

The numerical approximation of convection-dominated transport problems
\eqref{eq:1:transport} and incompressible flow \eqref{eq:2:stokes}
remains a challenging task, cf. \cite{BruchhaeuserJKN18} and references therein.
In \eqref{eq:1:transport}, convection-dominated transport is comprised
by assuming that $0< \varepsilon \ll |\convection|$.
The solution of these transport problems are typically characterized by the
occurrence of sharp moving fronts and layers. The key challenge for the
numerical approximation exists in the accurate and efficient solution while
avoiding non-physical oscillations or smearing effects.
The application of stabilization techniques is a typical approach to overcome
non-physical effects.
As shown in a comparative study for time-dependent convection-diffusion-reaction
equations in \cite{BruchhaeuserJS08}, stabilization techniques on globally refined
meshes fail to avoid these oscillations even after tuning stabilization parameters.
For a general review of stabilization techniques we refer to
\cite{BruchhaeuserRST08,BruchhaeuserJKN18}.

Furthermore, the non-availability of parameter-robust a posteriori error estimates
for quantities of physical interest and in general situations is complained in
\cite{BruchhaeuserJKN18}.
Moreover, the authors point out that adaptive mesh refinement strategies that are
based on such a posteriori error estimates are desirable and indispensable for
further improvement.

One possible technique for those adaptive strategies is goal-oriented a posteriori
error control. For a general review of a posteriori error estimation we refer to
\cite{BruchhaeuserAO00,BruchhaeuserV96}.
In particular, the Dual Weighted Residual (DWR) method,
introduced by Becker and Rannacher \cite{BruchhaeuserBR96,BruchhaeuserBR98,BruchhaeuserBR01},
allows for goal-oriented error control and adaptive mesh refinement
by weighting the influence of local residuals on the error within a goal
quantity of physical relevance.
The DWR approach relies on a space-time variational formulation of the discrete
problem and uses duality techniques to find a rigorous a posteriori error estimate
through the approximation of an addition dual problem.
Since the pioneering work in duality based error estimation, numerous studies
have been done for the application of the DWR method to several classes of
problems of physics including coupled phenomena and problems of optimal control.
Generalized versions of the DWR method in that not only the discretization error
but also the iteration or a modelling error is addressed, have been developed
and studied; cf. \cite{BruchhaeuserRV13,BruchhaeuserBE03}.
Further generalizations consider multi-objective goal functionals and the
treatment of higher-order corrections of the error estimator that are often
neglected; cf. \cite{BruchhaeuserEW17,BruchhaeuserELW19}.
In \cite{BruchhaeuserEW17}, error localization of the DWR method is performed in
a variational form using a partition-of-unity approach, instead of evaluating
strong operators and face integrals, as it is done in the
classical way of error localization \cite{BruchhaeuserBR03}. Thereby, node-wise
error contributions are obtained and computational costs are reduced.

DWR based error control has been well understood for single-physics problems.
Rigorous studies of DWR techniques for nonstationary multi-physics systems and
their numerical validations are still rare in the literature.
One reason for this might be that the separation of contributions to the
discretization error and the control of the temporal and spatial mesh becomes
more involved since, in addition, the impact of each of the subproblems and
its approximation on the goal quantity has to be understood and balanced within
the localized error representation and mesh refinement process.
In particular, goal-oriented error control was strongly analyzed for the
computation of nonstationary incompressible flow modeled by the Navier--Stokes
equations; cf. \cite{BruchhaeuserBR12,BruchhaeuserMR15} and references therein.
Moreover, the economical simulation of fluid-structure interaction as a
prominent system of multi-physics has attracted the usage of a posteriori error
control mechanisms based on duality techniques; cf.
\cite{BruchhaeuserR17,BruchhaeuserFW18,BruchhaeuserR12}.
A goal-oriented spatial-only adaptivity for a nonstationary transport problem 
coupled with a stationary Darcy flow, which is related to this work,
is studied in \cite{BruchhaeuserLM07}.

Finally, we note that the efficiency of goal-oriented space-time
adaptive methods demands on their efficient software implementation.
This requires the appropriate selection and implementation, respectively,
of data structures and efficient algorithms acting on them.
Recently, a programming model for the DWR approach applied to the nonstationary
diffusion equation with fixed lowest-order time discretizations for the primal
and dual problem was published by the authors in \cite{BruchhaeuserKBB17}.

In this work we combine the DWR approach with streamline upwind Petrov-Galerkin
(SUPG) stabilized approximations of convection-dominated transport problems as
introduced by Eq. \eqref{eq:1:transport}.
The transport problem is coupled via a convection tensor obtained by an auxiliary
flow problem as given by Eq. \eqref{eq:2:stokes} that has to be solved additionally.
%
%
Precisely, this work is characterised by the following features.
\begin{itemize}
\item An arbitrary order discontinuous Galerkin (dG) time discretization is
rigorously applied to the primal and dual problem.
\item The automatic adaptation of the space and time mesh is simplified by
separating the errors of temporal and spatial discretization similar to the
approach given in \cite{BruchhaeuserBR12,BruchhaeuserSV08}.
\item The dual residual is computed explicitly for both error representations,
the error in space and time, similarly to, e.g., \cite{BruchhaeuserBR12,BruchhaeuserSV08}.
\item A new software based on tensor-product space-time slabs was
developed. Further, the Stokes solver is an encapsulated module providing the
convection tensor for the convection-diffusion transport problem working on
a different triangulation. This is an extension of the software used in
\cite{BruchhaeuserKBB17}.
\end{itemize}

This work is organized as follows.
In Sec.~\ref{sec:2:modelproblem} we present the space-time discretization of our
model problem, including its stabilization for convection-dominated transport.
In Sec.~\ref{sec:3:errorestimation} the DWR method is applied and localized
a posteriori error representations, separating the effects of temporal and spatial
discretization, are derived.
In Sec.~\ref{sec:4:practical_aspects} the underlying adaptive algorithm is presented,
some practical aspects for the realization of the adaptivity process are illustrated
and details of the software implementation are given.
In Sec.~\ref{sec:5:examples} the numerical performance properties are studied with
convergence tests and a sophisticated experiment of physical relevance is studied.
Finally, in Sec.~\ref{sec:6:summary} we summarize and give some outlook
for future work.

\section{Model problem and stabilized space-time discretization}
\label{sec:2:modelproblem}

In this section we briefly introduce the space-time finite element
discretization of \eqref{eq:1:transport}, \eqref{eq:2:stokes}
including the SUPG stabilization to capture convec\-tion-dominated transport.

\subsection{Model problem}
\label{sec:2:1}

The time dependent convection-diffusion transport problem is given by
\begin{subequations}
\label{eq:transport_problem_0}
\begin{align}
\density \partial_{t} \concentration
- \nabla \cdot (\varepsilon \nabla \concentration)
+ \convection \cdot \nabla \concentration
+ \alpha \concentration & =  \transportforce \phantom{0 u_0} \text{in} \;\,
Q = \Omega\times I\,,\\
\concentration & =  0 \phantom{\transportforce u_0} \text{on}\;\;
\Sigma_D = \partial\Omega \times I\,,\\
%
%
\concentration & =  \concentration_{0}\phantom{\transportforce 0}
\text{on} \;\;\Sigma_0 = \Omega\times \{0\}\,,
\end{align}
\end{subequations}
with the coupled Stokes flow
\begin{subequations}
\label{eq:stokes_problem_0}
\begin{align}
- \viscosity\, \Delta \convection + \nabla \pressure & =
\stokesforce\phantom{0} \quad \text{in}\;\;  \Omega\,,\\
%
\nabla \cdot \convection & =  0\phantom{\stokesforce} \quad   \text{in}\;\; \Omega\,,
\end{align}
\end{subequations}
equipped with appropriate boundary conditions.
The system \eqref{eq:transport_problem_0}, \eqref{eq:stokes_problem_0} is
studied due to its prototype character for a wide range of applications in
practice, as introduced in Sec.~\ref{sec:1:introduction}.
In particular, its use for poroelasticity models and poromechanics in a more
general sense is straightforward.
The goal-oriented adaptivity approach that is developed in this work
is general enough such that it can be adapted to multi-physics systems.

In \eqref{eq:transport_problem_0}, \eqref{eq:stokes_problem_0}, we denote by
$\Omega \subset \mathbb{R}^{d}$, with $d=2,3$, a polygonal or polyhedral bounded
domain with Lipschitz boundary $\partial\Omega$ and $I=(0,T]$, $0 < T < \infty$,
is a finite time interval.
We assume that $\varepsilon > 0$ is a constant diffusion coefficient,
$\alpha \in L^{\infty}(\Omega)$ is the reaction coefficient,
$\density > 0$ is the constant density coefficient and
$\viscosity > 0$ is the constant viscosity coefficient.
Homogeneous Dirichlet boundary conditions in the transport and flow problem
are prescribed for brevity only.
In our numerical examples in Sec.~\ref{sec:5:examples} we also consider more
general boundary conditions.
For the sake of physical realism, the transport problem is supposed to be
convection-dominated by the at least local condition of
$0 < \varepsilon \ll | \convection |$.
Well-posedness of \eqref{eq:transport_problem_0}, \eqref{eq:stokes_problem_0}
and the existence of a sufficiently regular solution, such that all of the
arguments and terms used below are well-defined, are tacitly assumed without
mentioning explicitly all technical assumptions about the data and coefficients.

\subsection{Weak formulation}
\label{sec:2:2}

Let
$X:=\{v \in L^2(0,T;H^1_0(\Omega)) \mid \partial_t v \in L^2(0,T;H^{-1}(\Omega))\}$
and $Y_1 := H_0^1(\Omega)^d$.
The weak formulation of \eqref{eq:transport_problem_0} reads as follows:
\textit{Find $\concentration \in X$, $\convection \in Y_1$ such that}
\begin{equation}
\label{eq:2:3:A_u_phi_eq_F_phi}
A(\concentration,\convection)(\varphi) = G(\varphi) \quad \forall \varphi \in X\,,
\end{equation}
\textit{where the bilinear form $A: \{X,Y_1\} \times X \rightarrow \mathbb{R}$
and the linear form $G: L^2(0,T;$ $H^{-1}(\Omega)) \rightarrow \mathbb{R}$
are defined by}
\begin{displaymath}
\begin{array}{r@{\,}c@{\,}l}
A(\concentration,\convection)(\varphi) & := &  \displaystyle \int_{I} \big\{
  (\density \partial_{t} \concentration, \varphi)
+ a(\concentration,\convection)(\varphi) \big\}\; \mathrm{d} t
+ (\concentration(0), \varphi(0))\,,\\[1.5ex]
G(\varphi) & := & \displaystyle \int_I
  (\transportforce, \varphi)\; \mathrm{d} t
+ (\concentration_{0},\varphi(0) ) \,,
\end{array}
\end{displaymath}
\textit{with the bilinear form
$a: \{H^1_0(\Omega), H_0^1(\Omega)^d\} \times H^1_0(\Omega) \rightarrow \mathbb{R}$
given by}
\begin{equation}
\label{eq:2:4:Def_a_u_phi}
a(\concentration,\convection)(\varphi) :=
  (\varepsilon \nabla \concentration, \nabla \varphi)
+ (\convection \cdot \nabla \concentration, \varphi)
+ (\alpha \concentration, \varphi)\,.
\end{equation}
Here, $(\cdot, \cdot)$ denotes the inner product of $L^2(\Omega)$ or duality
pairing of $H^{-1}(\Omega)$ with $H^1_0(\Omega)$, respectively. By $\|\cdot \|$
we denote the associated $L^2$-norm.

For the variational formulation of problem \eqref{eq:stokes_problem_0} we define
$Y_2 := L_0^2(\Omega) := \big\{\pressure \in L^2(\Omega) \mid  \int_\Omega \pressure \;
\text{d}\boldsymbol{x}=0\big\}$. Then we get:
\textit{For $\stokesforce \in H^{-1}(\Omega)^d$
find
$\{ \convection, \pressure\} \in Y := Y_1 \times Y_2 =
H_0^1(\Omega)^d \times L_0^2(\Omega)$, such that}
\begin{equation}
\label{eq:2:5:B_vp_psichi_eq_F2_psi}
B(\convection,\pressure)(\boldsymbol{\psi},\chi) = F(\boldsymbol{\psi}) \quad
\forall \{\boldsymbol{\psi}, \chi\} \in Y\,,
\end{equation}
\textit{where the bilinear form $B:Y \times Y \rightarrow \mathbb{R}$
and the linear form $F: H^{-1}(\Omega)^d \rightarrow \mathbb{R}$ are defined by}
\begin{subequations}
\label{eq:2:6:Def_B_and_Def_F2}
\begin{align}
B(\convection,\pressure)(\boldsymbol{\psi},\chi) & :=
  \viscosity (\nabla \convection, \nabla \boldsymbol{\psi})
- (\pressure, \nabla \cdot \boldsymbol{\psi})
+ (\nabla \cdot \convection, \chi)\,,\\
F(\boldsymbol{\psi}) & :=  (\stokesforce, \boldsymbol{\psi})\,.
\end{align}
\end{subequations}

\begin{remark}
The coupling of the transport problem with the stationary Stokes problem is
via the convection variable $\convection$ of the system \eqref{eq:transport_problem_0},
\eqref{eq:stokes_problem_0}. We consider for the error estimation in
Sec.~\ref{sec:3:errorestimation} a coupled system but remark that the coupling
is uni-directional from the Stokes to the transport problem.
\end{remark}

\subsection{Discretization in time}
\label{sec:2:3}

For the discretization in time of the transport problem \eqref{eq:2:3:A_u_phi_eq_F_phi}
we use a discontinuous Galerkin method dG($r$) with an arbitrary polynomial
degree $r\ge 0$.
Let $0=:t_0<t_1<\dots<t_N:=T$ be a partition of the closure of the time domain
$\bar{I}=[0,T]$ into left-open subintervals $I_n:=(t_{n-1},t_n]$, $n=1,\dots,N$,
with time step sizes $\tau_n=t_n-t_{n-1}$ and the global time discretization
parameter $\tau=\max_{n}\,\tau_{n}$.
Therefore, we introduce the time-discrete function space
$X_{\tau}^{\textnormal{dG}(r)}$
for the transport problem.
\begin{equation}
\label{eq:2:7:Def_X_tau_dGr}
 \begin{aligned}
X_{\tau}^{\text{dG}(r)} :=
 \Big\{ & \concentration_{\tau}\in L^{2}(I; H_0^1(\Omega))\big|
 \concentration_{\tau}|_{I_{n}}\in \mathcal{P}_{r}(I_{n}; H_0^1(\Omega)),\\
 & \concentration_{\tau}(0)\in L^2(\Omega), n=1,\dots,N
\Big\}\,,
\end{aligned}
\end{equation}
where $\mathcal{P}_{r}(\bar{I}_{n}; H_0^1(\Omega))$ denotes the space of all
polynomials in time up to degree $r\geq0$ on $I_n$ with values in $H_0^1(\Omega)\,.$

For some discontinuous in time function
$\concentration_{\tau}\in X_{\tau}^{\textnormal{dG}(r)}$ we define
the limits $\concentration_{\tau,n}^{\pm}$ from above and below of
$\concentration_{\tau}$ at $t_n$ as
well as their jump at $t_n$ by
\begin{displaymath}
\begin{array}{lcrclcr}
\concentration_{\tau,n}^{\pm}
& := &
\displaystyle\lim_{t\mapsto t_n\pm0} \concentration_\tau(t) \,,
&
[\concentration_{\tau}]_{n} & := & \concentration_{\tau,n}^{+}
-\concentration_{\tau,n}^{-} \,.
\end{array}
\end{displaymath}

The semidiscretization in time of the the transport problem
\eqref{eq:2:3:A_u_phi_eq_F_phi} then reads as follows:
\textit{Find $\concentration_\tau \in X_{\tau}^{\textnormal{dG}(r)}$,
$\convection \in Y_1$ such that
}
\begin{equation}
\label{eq:2:8:A_tau_u_phi_eq_F_phi}
A_{\tau}(\concentration_\tau,\convection)(\varphi_\tau)
=
G_\tau(\varphi_\tau)
\quad \forall \varphi_\tau \in X_{\tau}^{\text{dG}(r)}\,,
\end{equation}
where the semi-discrete bilinear form and linear form are given by
\begin{subequations}
\label{eq:2:9:Def_A_tau_u_phi_and_Def_F_phi_tau}
\begin{align}
\label{eq:2:9:Def_A_tau_u_phi}
A_{\tau}(\concentration_\tau,\convection)(\varphi_\tau)
& :=
\sum_{n=1}^{N}\int_{I_n}\big\{(\density \partial_{t} \concentration_\tau,\varphi_\tau)
+ a(\concentration_\tau, \convection)(\varphi_\tau)
\big\} \mathrm{d} t\;
\\
\nonumber
& + \sum_{n=2}^{N}(\density[\concentration_\tau]_{n-1},\varphi_{\tau,n-1}^+ )
+ (\density\concentration_{\tau,0}^+,\varphi_{\tau,0}^+)\,,
\\
\label{eq:2:9:Def_F_phi_tau}
G_\tau (\varphi_{\tau}) & :=  \int_I(\transportforce,\varphi_{\tau})\;\mathrm{d}t
+ (\concentration_{0},\varphi_{\tau,0}^+)\,.
\end{align}
\end{subequations}

\subsection{Discretization in space}
\label{sec:2:4}

Next, we describe the Galerkin finite element approximation in space of the
semi-discrete transport problem (\ref{eq:2:8:A_tau_u_phi_eq_F_phi}) and
the flow problem \eqref{eq:2:5:B_vp_psichi_eq_F2_psi}, respectively.
We use Lagrange type finite element spaces of continuous functions that are
piecewise polynomials. For the discretization in space, we consider a
decomposition $\mathcal{T}_{h}$ of the domain $\Omega$ into disjoint elements
$K$, such that $\overline{\Omega}=\cup_{K\in\mathcal{T}_{h}}\overline{K}$.
Here, we choose the elements $K\in\mathcal{T}_{h}$ to be quadrilaterals
for $d=2$ and hexahedrals for $d=3$.
We denote by $h_{K}$ the diameter of the element $K$. The global space
discretization parameter $h$ is given by $h:=\max_{K\in\mathcal{T}_{h}}h_{K}$.
Our mesh adaptation process yields locally refined cells, which is
enabled by using hanging nodes. We point out that
the global conformity of the finite element approach is preserved since the
unknowns at such hanging nodes are eliminated by interpolation between the
neighboring 'regular' nodes; cf.~\cite[Chapter 4.2]{BruchhaeuserBR03} and
\cite{BruchhaeuserCO84} for more details.
On $\mathcal{T}_{h}$ we define the discrete finite element space by
$
V_{h}^{p,n}:=
\big\{v\in C(\overline{\Omega})\mid v_{|K}
\in Q_h^p(K)\,,\forall K\in\mathcal{T}_{h},
\big\}\,,
$
with $n=1,\dots,N$, where $Q_h^p(K)$ is the space defined on the reference element
with maximum degree $p$ in each variable. By replacing $H_0^1(\Omega)$ in the
definition of the semi-discrete function space $X_{\tau}^{\textnormal{dG}(r)}$
in (\ref{eq:2:7:Def_X_tau_dGr}) by $V_h^{p,n}$, we obtain the fully discrete
function space for the transport problem
\begin{equation}
\begin{aligned}
\label{eq:2:10:Def_X_tau_h_dGr_p}
X_{\tau h}^{\text{dG}(r),p} := \Big\{ &
\concentration_{\tau h}\in X_{\tau}^{\text{dG}(r)} \big|
\concentration_{\tau h}|_{I_n} \in \mathcal{P}_r(I_n;H_h^{p_{\concentration},n})
\,,\\
& \concentration_{\tau h}(0) \in H_h^{p_{\concentration},0},
n=1,\dots,N
\Big\}
\subseteq L^{2}(I; H_0^1(\Omega))\,.
\end{aligned}
\end{equation}
The discrete in space function space for the flow problem is given by
\begin{eqnarray}
\label{eq:2:11:Def_Y_h_p}
Y_{h}^{p} & := & (H_h^{p_v})^d \times L_h^{p_p}
\subseteq Y\,,
\end{eqnarray}
\begin{displaymath}
H_h^{p_{\concentration},n}:=V_h^{p_{\concentration},n}\cap H_0^1(\Omega), \quad
H_h^{p_v}:=V_h^{p_v}\cap H_0^1(\Omega), \quad
L_h^{p_p}:=V_h^{p_p}\cap L_0^2(\Omega).
\end{displaymath}
We note that the spatial finite
element space $V_h^{p,n}$ is allowed to be different on
all subintervals $I_n$ which is natural in the context of a
discontinuous Galerkin approximation of the time variable and allows dynamic
mesh changes in time.
Due to the conformity of $H_h^{p_{\concentration},n}$ we get
$X_{\tau h}^{\textnormal{dG}(r),p}\subseteq X_{\tau}^{\textnormal{dG}(r)}$.

The fully discrete discontinuous in time scheme for the transport problem then
reads as follows:
\textit{Find $\concentration_{\tau h} \in X_{\tau h}^{\textnormal{dG}(r),p}$,
$\convection_h \in (H_h^{p_v})^d$ such that}
\begin{equation}
\label{eq:2:12:A_tau_h_u_phi_eq_F_phi}
 A_{\tau}(\concentration_{\tau h}, \convection_h)(\varphi_{\tau h})
=
G_\tau(\varphi_{\tau h})
\quad \forall \varphi_{\tau h} \in X_{\tau h}^{\text{dG}(r),p}\,,
\end{equation}
\textit{with} $A_{\tau}(\cdot,\cdot)(\cdot)$ \textit{and} $G_\tau(\cdot)$
\textit{being defined in \eqref{eq:2:9:Def_A_tau_u_phi_and_Def_F_phi_tau}.}
We note that the bilinear form $a(\cdot,\cdot)(\cdot)$ occurring in
$A_\tau(\cdot,\cdot)(\cdot)$ reads here as
\begin{displaymath}
a(\concentration_{\tau h}, \convection_h)(\varphi_{\tau h}) =
(\varepsilon \nabla \concentration_{\tau h}, \nabla \varphi_{\tau h})
+(\convection_{h} \cdot \nabla \concentration_{\tau h}, \varphi_{\tau h})
+ (\alpha \concentration_{\tau h},\varphi_{\tau h})
\end{displaymath}
for the fully discrete solutions.

The fully discrete scheme for the flow problem reads as follows:
\textit{Find $\{\convection_{h},\pressure_{h}\} \in
Y_{h}^{p}$ such that}
\begin{equation}
\label{eq:2:13:B_v_p_psi_chi_eq_F_psi}
B(\convection_{h},\pressure_{h})(\boldsymbol{\psi}_{h},\chi_{h})
= F(\boldsymbol{\psi}_{h})
\quad \forall \{\boldsymbol{\psi}_{h},\chi_{h}\} \in
Y_{h}^{p}\,,
\end{equation}
\textit{with $B(\cdot,\cdot)(\cdot,\cdot)$ and $F(\cdot)$ being defined in
\eqref{eq:2:6:Def_B_and_Def_F2}.}

\subsection{SUPG stabilization}
\label{sec:2:5}

In this work we consider, for the sake of physical realism,
convection-dominated transport with small diffusion parameter $\varepsilon$
in Eq.~\eqref{eq:transport_problem_0} which, on the hand, poses an additional
challenge to the a posteriori error control but, on the other hand, illustrates
nicely the potential, reliability and efficiency of the DWR-based approach.
For convection-dominated transport, the finite element approximation needs to be
stabilized in order to further reduce spurious and non-physical oscillations of
the discrete solution arising close to sharp fronts or layers. Here, we apply
the streamline upwind Petrov-Galerkin (SUPG) method;
cf. \cite{BruchhaeuserHB79,BruchhaeuserBH81}.
We explicitly note that SUPG stabilization and automatic mesh adaptation
interact strongly; cf. Rem. \ref{rem:2:2}.
Balancing their effects needs particular consideration and has not been
strongly studied so far in the literature;
cf. for instance \cite{BruchhaeuserBSB18}.

The stabilized variant of the fully discrete scheme
\eqref{eq:2:12:A_tau_h_u_phi_eq_F_phi} then reads as follows:
\textit{Find $\concentration_{\tau h} \in X_{\tau h}^{\textnormal{dG}(r),p}$,
$\convection_h \in (H_h^{p_v})^d$ such that}
\begin{equation}
\label{eq:2:14:A_S_u_phi_eq_F_phi}
A_{S}(\concentration_{\tau h}, \convection_h)(\varphi_{\tau h}) =
G_\tau(\varphi_{\tau h})
\quad \forall \varphi_{\tau h} \in X_{\tau h}^{\text{dG}(r),p}\,,
\end{equation}
\textit{with $A_{S}(\concentration_{\tau h}, \convection_h)(\varphi_{\tau h}) :=
A_{\tau}(\concentration_{\tau h},\convection_h)(\varphi_{\tau h})
+S(\concentration_{\tau h}, \convection_h)(\varphi_{\tau h})$
and stabilization term}
\begin{equation}
\label{eq:2:15:Def_S_u_phi}
\begin{aligned}
S(\concentration_{\tau h}, \convection_h)(\varphi_{\tau h}) & :=
\displaystyle\sum_{n=1}^N\int_{I_n}
\sum\limits_{K\in \mathcal{T}_h}\delta_K\big(
r(\concentration_{\tau h}, \convection_h),
\convection_{h} \cdot \nabla \varphi_{\tau h}\big)_K \,\mathrm{d} t
\\
& \quad
+ \displaystyle\sum\limits_{n=2}^{N}\sum\limits_{K\in\mathcal{T}_h}
\delta_K
\big(\density\left[\concentration_{\tau h}\right]_{n-1},
\convection_{h} \cdot \nabla \varphi_{\tau h,n-1}^+\big)_{K}
\\
& \quad
+ \displaystyle\sum\limits_{K\in\mathcal{T}_h}
\delta_K \big(\density \concentration_{\tau h,0}^{+} - \concentration_0,
\convection_{h} \cdot \nabla \varphi_{\tau h,0}^{+} \big)_{K}
\end{aligned}
\end{equation}
\textit{and residual}
\begin{displaymath}
r(\concentration_{\tau h}, \convection_h) :=
\density \partial_{t} \concentration_{\tau h}
- \nabla\cdot\left(\varepsilon\nabla \concentration_{\tau h}\right)
+ \convection_{h} \cdot \nabla \concentration_{\tau h}
+ \alpha \concentration_{\tau h}
- \transportforce \,.
\end{displaymath}

\begin{remark}
\label{rem:2:2}
The proper choice of the stabilization parameter $\delta_K$ is an important
issue in the application of the SUPG approach; cf., e.g.,
\cite{BruchhaeuserJN11,BruchhaeuserJS08,BruchhaeuserJKN18}
and the discussion therein. For time-dependent convection-diffusion-reaction
problems an optimal error estimate for $\delta_K=\mathrm{O}(h)$ is derived
in \cite{BruchhaeuserJN11}.
\end{remark}

\begin{remark}
\label{rem:2:3}
For the error $e = \concentration_{\tau}-\concentration_{\tau h}$
we get by subtracting Eq.~(\ref{eq:2:14:A_S_u_phi_eq_F_phi}) from
Eq.~(\ref{eq:2:8:A_tau_u_phi_eq_F_phi}) the identity
\begin{equation}
\label{eq:2:16:Galerkin_orthogonality_v_v_h}
\begin{aligned}
\sum_{n=1}^{N} & \int_{I_n}
\big\{
(\density \partial_{t} e,\varphi_{\tau h})
+ a(e, \convection_h)(\varphi_{\tau h})
\big\}
\mathrm{d} t\;\\
& \qquad + \sum_{n=2}^{N}(\density[e]_{n-1},\varphi_{\tau h,n-1}^+ )
+ (e_{0}^+,\varphi_{\tau h,0}^+)
\\
&=
S(\concentration_{\tau h}, \convection_h)(\varphi_{\tau h})
- \sum_{n=1}^{N}\int_{I_n}
\big(
(\convection-\convection_{h}) \cdot \nabla \concentration_{\tau},
\varphi_{\tau h}
\big)
\mathrm{d} t\,,
\end{aligned}
\end{equation}
with a non-vanishing right-hand side term depending on the stabilization and the
error in the approximation of the flow field.
Eq.~\eqref{eq:2:16:Galerkin_orthogonality_v_v_h} with the perturbation term on
the right-hand side replaces the standard Galerkin orthogonality of the
space-time finite element approximation.
\end{remark}

\section{A posteriori error estimation}
\label{sec:3:errorestimation}

In this section we derive our DWR error representation for the 
stabilized transport problem \eqref{eq:2:14:A_S_u_phi_eq_F_phi} coupled with the 
flow problem via the convection tensor $\convection_h$ given by 
Eq.~\eqref{eq:2:13:B_v_p_psi_chi_eq_F_psi}.
Here, only goal quantities depending on the unknown $\concentration$ are
studied. For applications of practical interest, physical quantities in terms of
the transport quantity $u$ are typically of higher relevance than quantities in
the unknowns $\convection$ and $\pressure$ of the flow problem. In the sequel,
we denote by $J:X \rightarrow \mathbb{R}$ a user-chosen, physically relevant
target functional represented in the form
\begin{equation}
\label{eq:3:1:Def_J_u}
J(\concentration)=\int_0^T J_1(\concentration(t))\mathrm{d}t + J_2(\concentration(T))\,,
\end{equation}
where $J_1:H^1_0(\Omega)\rightarrow \mathbb R$ or
$J_2:H^1_0(\Omega)\rightarrow \mathbb R$ may be zero. Since we aim at
controlling the respective errors due to the discretization in time as well as
in space, we split the a posteriori error
representation with respect to $J$ into the contributions
\begin{equation}
\label{eq:3:2:J_u_J_u_tauh_eq_J_u_J_u_tau_J_u_tauh}
 J(\concentration)-J(\concentration_{\tau h}) =
 J(\concentration)-J(\concentration_{\tau})
 + J(\concentration_{\tau})-J(\concentration_{\tau h})\,.
\end{equation}

For the respective error representations we define the Lagrangian
functionals
$\mathcal{L}: X\times X \rightarrow \mathbb{R}$,
$\mathcal{L}_\tau: X_{\tau}^{\textnormal{dG}(r)} \times X_{\tau}^{\textnormal{dG}(r)}
\rightarrow \mathbb{R}$, and
$\mathcal{L}_{\tau h}:
X_{\tau h}^{\textnormal{dG}(r),p} \times X_{\tau h}^{\textnormal{dG}(r),p}
\rightarrow \mathbb{R}$ by
\begin{subequations}
\label{eq:3:3:Def_L_u_z_Def_L_tau_u_z_Def_L_tau_h_u_z}
\begin{align}
\label{eq:3:3:Def_L_u_z}
\mathcal{L}(\concentration,\dualz) & :=  J(\concentration)
+ G(\dualz)
- A(\concentration, \convection)(\dualz)\,,
\\
\label{eq:3:3:Def_L_tau_u_z}
\mathcal{L}_{\tau}(\concentration_\tau,\dualz_\tau) & :=
J(\concentration_{\tau}) + G_\tau(\dualz_{\tau})
- A_{\tau}(\concentration_{\tau}, \convection)(\dualz_{\tau})\,,
\\
\label{eq:3:3:Def_L_tau_h_u_z}
\mathcal{L}_{\tau h}(\concentration_{\tau h},\dualz_{\tau h}) & :=
J(\concentration_{\tau h})
+ G_\tau (\dualz_{\tau h})
- A_S(\concentration_{\tau h}, \convection_h)(\dualz_{\tau h})\,.
\end{align}
\end{subequations}
Here, the Lagrange multipliers $\dualz$, $\dualz_\tau,$ and $\dualz_{\tau h}$
are called dual variables in contrast to the primal variables
$\concentration$, $\concentration_\tau,$ and
$\concentration_{\tau h}$; cf. \cite{BruchhaeuserBR12,BruchhaeuserBR01}.
\begin{remark}
For the sake of simplicity, we exclude the convection field $\convection$ and
$\convection_h$ from the primal variables within the Lagrangian functionals 
in \eqref{eq:3:3:Def_L_u_z_Def_L_tau_u_z_Def_L_tau_h_u_z} due to the choice of 
the goal quantity given by Eq.~\eqref{eq:3:1:Def_J_u}.
This can be generalized in a standard fashion by introducing a vector of primal 
unknowns and a respective Lagrangian multiplier as dual variable.
Nevertheless, the coupling is still present within the transport problem and 
results in additional coupling terms in the error representation formula, 
cf. Eq.~\eqref{eq:3:16b:J_u_tau_minus_J_u_tau_h} and 
Rem.~\ref{rem:3:3:AdditionalCouplingTerms}.
\end{remark}
Considering the directional derivatives of the Lagrangian functionals, also
known as
G\^{a}teaux derivatives, with respect to their first argument, i.e.\
\begin{equation}
\label{eq:3:4:Def_Gateaux_derivative}
\mathcal{L}^{\prime}_{\concentration}(\concentration,\dualz)(\varphi) :=
\lim_{t\neq0,t\rightarrow 0}
t^{-1}\big\{\mathcal{L}(\concentration + t\varphi,\dualz)
-\mathcal{L}(\concentration,\dualz)\big\},
\quad \varphi \in X\,,
\end{equation}
leads to the so-called dual problems, cf., e.g., \cite{BruchhaeuserBR12}.
The continuous, semi-discrete, and fully discrete dual solutions
$\dualz~\in~X$, $\dualz_\tau~\in~X_{\tau}^{\textnormal{dG}(r)},$ and
$\dualz_{\tau h} \in X_{\tau h}^{\textnormal{dG}(r),p}$
are determined by the optimality conditions
\begin{subequations}
\label{eq:3:5:Def_L_uprime_u_z_Def_L_tau_uprime_u_z_Def_L_tauh_uprime_u_z}
\begin{align}
\label{eq:3:5:Def_L_uprime_u_z}
\mathcal{L}^{\prime}_{\concentration}(\concentration,\dualz)(\varphi) & =  0
\quad \forall \varphi \in X\,,
\\
\label{eq:3:5:Def_L_tau_uprime_u_z}
\mathcal{L}^{\prime}_{\tau, \concentration}(\concentration_\tau,\dualz_\tau)(\varphi) & =  0
\quad \forall \varphi_{\tau} \in
X_{\tau}^{\text{dG}(r)}\,,
\\
\label{eq:3:5:Def_L_tauh_uprime_u_z}
\mathcal{L}^{\prime}_{\tau h, \concentration}(\concentration_{\tau h},\dualz_{\tau h})(\varphi) & =  0
\quad \forall
\varphi_{\tau h} \in X_{\tau h}^{\text{dG}(r),p}\,.
\end{align}
\end{subequations}
More precisely, the continuous dual solution $\dualz \in X$ is the solution of
\begin{equation}
\label{eq:3:6:A_prime_u_phi_z_eq_J_prime_u_phi}
A^{\prime}(\concentration, \convection)(\varphi,\dualz)
=
J^{\prime}(\concentration)(\varphi)
\quad \forall \varphi \in X\,,
\end{equation}
where
the adjoint bilinear form $A^{\prime}(\cdot, \cdot)(\cdot, \cdot)$ is given by
\begin{equation}
\label{eq:3:7:Def_A_prime_u_phi_z}
A^{\prime}(\concentration, \convection)(\varphi,\dualz)
:=
\int_I\big\{(\varphi,-\density \partial_{t} \dualz)
+a^{\prime}(\concentration, \convection)(\varphi,\dualz)\big\}\mathrm{d}t
+(\varphi(T),\dualz(T))\,.
\end{equation}
We note that for the representation \eqref{eq:3:7:Def_A_prime_u_phi_z} of
$A^{\prime}(\cdot, \cdot)(\cdot, \cdot)$ integration by parts in time is applied, 
which is allowed for weak
solutions $z\in X$; cf., e.g., \cite[Lemma 8.9]{BruchhaeuserR17}.
The derivative $a^{\prime}(\concentration, \convection)(\varphi,\dualz)$ of the 
bilinear form $a(\concentration, \convection)(\dualz)$ in $A^{\prime}$ admits 
the explicit form
\begin{displaymath}
a^{\prime}(\concentration, \convection)(\varphi,\dualz)=
(\varepsilon\nabla \varphi,\nabla \dualz)
+ (\convection \cdot \nabla \varphi,\dualz)
+ (\alpha \varphi,\dualz)\,.
\end{displaymath}
The right-hand side of Eq. \eqref{eq:3:6:A_prime_u_phi_z_eq_J_prime_u_phi} is
given by
\begin{equation}
\label{eq:3:8:Def_J_prime_u_phi}
J^{\prime}(\concentration)(\varphi) :=
\int_I J_1^{\prime}(\concentration)(\varphi)\mathrm{d}t+J_2^{\prime}(\concentration(T))(\varphi(T))\,.
\end{equation}
Further, the semi-discrete dual solution $\dualz_\tau\in X_{\tau}^{\textnormal{dG}(r)}$
and the fully discrete dual solution $\dualz_{\tau h} \in X_{\tau h}^{\textnormal{dG}(r),p}$
satisfy the
equations
\begin{subequations}
\label{eq:3:9:A_tau_prime_u_phi_z_eq_J_prime_u_phi_A_S_prime_u_phi_z_eq_J_prime_u_phi}
\begin{align}
\label{eq:3:9:A_tau_prime_u_phi_z_eq_J_prime_u_phi}
A_{\tau}^{\prime}(\concentration_{\tau}, \convection)(\varphi_{\tau},\dualz_{\tau})
& =
J^{\prime}(\concentration_{\tau})(\varphi_{\tau})
\quad \hspace{0.3cm} \forall \varphi_{\tau}\in X_{\tau}^{\text{dG}(r)}\,,
\\
\label{eq:3:9:A_S_prime_u_phi_z_eq_J_prime_u_phi}
A_{S}^{\prime}(\concentration_{\tau h}, \convection_h)(\varphi_{\tau h},\dualz_{\tau h})
& =
J^{\prime}(\concentration_{\tau h})(\varphi_{\tau h})
\quad \forall \varphi_{\tau h}\in X_{\tau h}^{\text{dG}(r),p}\,,
\end{align}
\end{subequations}
where $A_{\tau}^{\prime}(\cdot, \cdot)(\cdot, \cdot)$ and 
$A_{S}^{\prime}(\cdot, \cdot)(\cdot, \cdot)$ are given by
\begin{displaymath}
\begin{aligned}
A_{\tau}^{\prime}(\concentration_{\tau}, \convection)(\varphi_{\tau},\dualz_{\tau})  :=
& \sum_{n=1}^N\int_{I_n}\big\{
(\varphi_{\tau},-\density \partial_{t} \dualz_{\tau})
+ a^{\prime}(\concentration_{\tau}, \convection)(\varphi_{\tau},\dualz_{\tau})
\big\}\mathrm{d}t
\\
& - \sum_{n=1}^{N-1}(\varphi_{\tau,n}^-,\density[\dualz_{\tau}]_{n})
+(\varphi_{\tau,N}^-,\density \dualz_{\tau,N}^-)\,,
\end{aligned}
\end{displaymath}
and
\begin{displaymath}
\begin{aligned}
A_{S}^{\prime}(\concentration_{\tau h}, \convection_h)&(\varphi_{\tau h},\dualz_{\tau h}) :=
 \sum_{n=1}^N\int_{I_n}\big\{
(\varphi_{\tau h},-\density \partial_{t} \dualz_{\tau h})
+ a_{h}^{\prime}(\concentration_{\tau h}, \convection_h)(\varphi_{\tau h},\dualz_{\tau h})
\big\}\mathrm{d}t
\\
&
+ S^{\prime}(\concentration_{\tau h}, \convection_h)(\varphi_{\tau h},\dualz_{\tau h})
- \sum_{n=1}^{N-1}(\varphi_{\tau h,n}^-,\density[\dualz_{\tau h}]_{n})
+ (\varphi_{\tau h,N}^-,\density \dualz_{\tau h,N}^-)\,.
\end{aligned}
\end{displaymath}

\begin{remark}
We note that the directional derivatives of the Lagrangian functionals with
respect to their second argument leads to the primal problems given by
Eqs.~\eqref{eq:2:3:A_u_phi_eq_F_phi},
\eqref{eq:2:8:A_tau_u_phi_eq_F_phi} and \eqref{eq:2:14:A_S_u_phi_eq_F_phi},
respectively.
\end{remark}

For the derivation of computable representations of the separated error
contributions in Eq.~\eqref{eq:3:2:J_u_J_u_tauh_eq_J_u_J_u_tau_J_u_tauh} we
need the following known result, that is explicitly summarized here in order
to keep this work self-contained.

\begin{lemma}
\label{Bruchhaeuser:Lem:1.1}
Let $\mathcal{Y}$ be a function space and $L$ and $\tilde{L}$ be three times
G\^{a}teaux differentiable functionals on $\mathcal{Y}$.
We seek a stationary point $y_1$ of $L$ on a subspace
$\mathcal{Y}_1\subseteq \mathcal{Y}$:
Find $y_1\in \mathcal{Y}_1$ such that
\begin{equation}
\label{eq:3:10:L_prime_y1_delta_y1}
L^{\prime}(y_1)(\delta y_1) = 0 \quad \forall \delta y_1 \in \mathcal{Y}_1.
\end{equation}
This equation is approximated by a Galerkin method using the functional
$\tilde{L}$ on a subspace $\mathcal{Y}_2\subseteq \mathcal{Y}$.
Hence, the discrete problem seeks $y_2\in \mathcal{Y}_2$ such that
\begin{equation}
\label{eq:3:11:L_tilde_prime_y2_delta_y2}
\tilde{L}^{\prime}(y_2)(\delta y_2) = 0 \quad \forall \delta y_2 \in \mathcal{Y}_2.
\end{equation}
If the continuous solution $y_1$ additionally fulfills
\begin{equation}
\label{eq:3:12:L_prime_y1_y2}
L^{\prime}(y_1)(y_2) = 0\,,
\end{equation}
with the approximated solution $y_2$, we  have the error representation
\begin{eqnarray}
\label{eq:3:13:L_y1_minus_L_tilde_y2}
\begin{aligned}
L(y_1)-\tilde{L}(y_2) =& \frac{1}{2}L^{\prime}(y_2)(y_1-\tilde{y}_2)
\\
&+ \frac{1}{2}(L-\tilde{L})^{\prime}(y_2)(\tilde{y}_2-y_2)
+ (L-\tilde{L})(y_2)
+ \mathcal{R}
\,,
\end{aligned}
\end{eqnarray}
for arbitrary $\tilde{y}_2 \in \mathcal{Y}_2$, where the remainder term
$\mathcal{R}$ is given in terms of $e:=y_1-y_2$ as
\begin{equation}
\label{eq:3:14:Def_R}
\mathcal{R}=\frac{1}{2}
\int_0^1
L^{\prime\prime\prime}(y_2+se)(e,e,e)s(s-1)\mathrm{d}s\,.
\end{equation}
\end{lemma}

\begin{proof}
The proof of Lemma~\ref{Bruchhaeuser:Lem:1.1} can be found in \cite{BruchhaeuserBR12}.
\end{proof}

In the following Thm.~\ref{Thm:1.1} we apply the abstract error representation
formula (\ref{eq:3:13:L_y1_minus_L_tilde_y2}) to the Lagrangian
functionals (\ref{eq:3:3:Def_L_u_z})--(\ref{eq:3:3:Def_L_tau_h_u_z}). This step is a
modification of Thm.~5.2 in \cite{BruchhaeuserBR12} due
to the presence of the additional coupling term, cf.~Rem.~\ref{rem:2:3}.
To proceed with our computations, we still introduce the primal and dual
residuals that are defined by means of
\begin{equation}
\label{eq:3:15:primal_dual_residuals}
\rho_{\mathrm{t}}(\concentration)(\varphi)  :=
\mathcal{L}_{\tau,\dualz}^{\prime}(\concentration,\dualz)(\varphi)\,,
\quad \quad
\rho_{\mathrm{t}}^{\ast}(\concentration,\dualz)(\varphi)
:=
\mathcal{L}_{\tau,\concentration}^{\prime}(\concentration,\dualz)(\varphi)\,.
\end{equation}

By using Lemma~\ref{Bruchhaeuser:Lem:1.1} we now get the following result for the
DWR-based error representation.

\begin{theorem}
\label{Thm:1.1}
Let $\{\concentration,\dualz\}\in X \times X$,
$\{\concentration_{\tau},\dualz_{\tau}\}
\in
X_{\tau}^{\textnormal{dG}(r)} \times X_{\tau}^{\textnormal{dG}(r)}$,
and
$\{\concentration_{\tau h},\dualz_{\tau h}\}
\in X_{\tau h}^{\textnormal{dG}(r),p} \times X_{\tau h}^{\textnormal{dG}(r),p}$
denote the stationary points of
$\mathcal{L}, \mathcal{L}_{\tau}$, and $\mathcal{L}_{\tau h}$
on the different levels of discretization, i.e.,
\begin{displaymath}
\begin{aligned}
\mathcal{L}^{\prime}(\concentration,\dualz)(\delta \concentration, \delta \dualz)
= \mathcal{L}_{\tau}^{\prime}(\concentration,\dualz)(\delta \concentration, \delta \dualz)
& = 0 \quad
\forall \{\delta \concentration,\delta \dualz\}\in X \times X\,,
\\
\mathcal{L}_{\tau}^{\prime}(\concentration_{\tau},\dualz_{\tau})
(\delta \concentration_{\tau}, \delta \dualz_{\tau})
& = 0
\quad \forall \{\delta \concentration_{\tau},\delta \dualz_{\tau}\}
\in X_{\tau}^{\text{dG}(r)} \times X_{\tau}^{\text{dG}(r)}\,,
\\
\mathcal{L}_{\tau h}^{\prime}(\concentration_{\tau h},\dualz_{\tau h})
(\delta \concentration_{\tau h}, \delta \dualz_{\tau h})
& = 0
\quad \forall \{\delta \concentration_{\tau h},\delta \dualz_{\tau h}\}
\in X_{\tau h}^{\text{dG}(r),p} \times X_{\tau h}^{\text{dG}(r),p}\,.
\end{aligned}
\end{displaymath}
Additionally, for the error $e = \concentration_{\tau} - \concentration_{\tau h}$
we have the Eq.~\eqref{eq:2:16:Galerkin_orthogonality_v_v_h} of Galerkin
orthogonality type. Then, for the discretization errors in space and time we get
the representation formulas
\begin{subequations}
\label{eq:3:16}
\begin{align}
\label{eq:3:16a:J_u_minus_J_u_tau}
J(\concentration)-J(\concentration_{\tau}) & =
\frac{1}{2}\rho_{\mathrm{t}}(\concentration_{\tau})(\dualz-\tilde{\dualz}_{\tau})
+ \frac{1}{2}\rho_{\mathrm{t}}^{\ast}(\concentration_{\tau},\dualz_{\tau})
(\concentration-\tilde{\concentration}_{\tau})
+ \mathcal{R}_{\tau}\,,
\\
\label{eq:3:16b:J_u_tau_minus_J_u_tau_h}
J(\concentration_{\tau})-J(\concentration_{\tau h}) & =
\frac{1}{2}\rho_{\mathrm{t}}(\concentration_{\tau h})(\dualz_{\tau}-\tilde{\dualz}_{\tau h})
+ \frac{1}{2}
\rho_{\mathrm{t}}^{\ast}(\concentration_{\tau h},\dualz_{\tau h})
(\concentration_{\tau}-\tilde{\concentration}_{\tau h})
\\
\nonumber
& \qquad
+ \frac{1}{2} \mathcal{D}_{\tau h}^{\prime}(\concentration_{\tau h},\dualz_{\tau h})
(\tilde{\concentration}_{\tau h}-\concentration_{\tau h},\tilde{\dualz}_{\tau h}-\dualz_{\tau h})
\\
\nonumber
& \qquad
+ \mathcal{D}_{\tau h}(\concentration_{\tau h},\dualz_{\tau h}) + \mathcal{R}_{h}\,,
\end{align}
\end{subequations}
where $\mathcal{D}_{\tau h}(\cdot,\cdot)$ is given by
\begin{equation}
\label{eq:3:17:Def_D_tau_h_phi_psi}
\mathcal{D}_{\tau h}(\varphi,\psi)
=
S(\varphi, \convection_h)(\psi)
-\sum_{n=1}^{N}\int_{I_n}\big(
 (\convection-\convection_{h})\cdot\nabla \varphi,\psi
 \big)
 \mathrm{d}t
\,,
\end{equation}
with $S(\cdot,\cdot)(\cdot)$ being defined in \eqref{eq:2:15:Def_S_u_phi}.
Here,
$\{\tilde{\concentration}_{\tau},\tilde{\dualz}_{\tau}\}\in X_{\tau}^{\textnormal{dG}(r)}
\times X_{\tau}^{\textnormal{dG}(r)}$,
and
$\{\tilde{\concentration}_{\tau h},\tilde{\dualz}_{\tau h}\} \in
X_{\tau h}^{\textnormal{dG}(r),p} \times X_{\tau h}^{\textnormal{dG}(r),p}$
can be chosen arbitrarily and the remainder terms $\mathcal{R}_{\tau}$ and
$\mathcal{R}_{h}$ have the same structure as the remainder term
\eqref{eq:3:14:Def_R} in Lemma~\ref{Bruchhaeuser:Lem:1.1}.
\end{theorem}
%
%
\begin{remark}
\label{rem:3:3:AdditionalCouplingTerms}
We note that within the spatial error representation formula 
\eqref{eq:3:16b:J_u_tau_minus_J_u_tau_h} additional terms due to the coupling 
occur besides the terms due to stabilization, 
cf. Eq.~\eqref{eq:3:17:Def_D_tau_h_phi_psi}.
This is an extension of our previous results obtained in \cite{BruchhaeuserBSB18},
where a convection-dominated transport problem was considered for a given
convection tensor not obtained by a Stokes problem that has to be solved
additionally, cf. the algorithm described in Sec.~\ref{sec:4:practical_aspects}.
\end{remark}
%
\begin{proof}
The proof is related to that one of Thm.~5.2 in \cite{BruchhaeuserBR12}.
Evaluating the Lagrangian functionals at the respective primal and dual
solutions, there holds that
\begin{displaymath}
J(\concentration)=\mathcal{L}(\concentration,\dualz)\,, \quad
J(\concentration_{\tau})=\mathcal{L}_{\tau}(\concentration_{\tau},\dualz_{\tau})\,, \quad
J(\concentration_{\tau h})=\mathcal{L}_{\tau h}(\concentration_{\tau h},\dualz_{\tau h})\,.
\end{displaymath}
Since the additional jump terms in $\mathcal{L}_{\tau}$ vanish for a continuous
solution $\concentration \in X$, we get the following representation for the
temporal and spatial error, respectively,
\begin{subequations}
\begin{align}
\label{eq:3:18:J_u_minus_J_u_tau_eq}
J(\concentration) - J(\concentration_{\tau}) & =
\mathcal{L}(\concentration,\dualz) - \mathcal{L}_{\tau}(\concentration_{\tau},\dualz_{\tau})
= \mathcal{L}_{\tau}(\concentration,\dualz) - \mathcal{L}_{\tau}(\concentration_{\tau},\dualz_{\tau})
\,,
\\
\label{eq:3:18:J_u_tau_minus_J_u_tau_h_eq}
J(\concentration_{\tau})-J(\concentration_{\tau h}) & =
\mathcal{L}_{\tau}(\concentration_{\tau},\dualz_{\tau})
- \mathcal{L}_{\tau h}(\concentration_{\tau h},\dualz_{\tau h})\,.
\end{align}
\end{subequations}

To prove the assertion (\ref{eq:3:16a:J_u_minus_J_u_tau}) for the temporal error,
we apply Lemma~\ref{Bruchhaeuser:Lem:1.1} with the identifications
\begin{displaymath}
L = \mathcal{L}_{\tau}\,,\;\;
\tilde{L}=\mathcal{L}_{\tau}\,,\;\;
\mathcal{Y}_1 = X \times X\,,\;\;
\mathcal{Y}_2 = X_{\tau}^{\text{dG}(r)} \times X_{\tau}^{\text{dG}(r)}
\end{displaymath}
to the identity \eqref{eq:3:18:J_u_minus_J_u_tau_eq}. Further, we have to choose
$\mathcal Y:=\mathcal Y_1+ \mathcal Y_2$ since
here {$X_{\tau}^{\textnormal{dG}(r)} \nsubseteq X$}. Thus, we have to verify
condition \eqref{eq:3:12:L_prime_y1_y2}, that now reads as
$\mathcal{L}^{\prime}(\concentration, \dualz)(\concentration_{\tau}, \dualz_{\tau}) = 0$,
or equivalently,
\begin{equation}
\label{eq:3:19:L_u_prime_L_z_prime_zero}
\mathcal{L}_{\concentration}^{\prime}(\concentration, \dualz)(\concentration_{\tau}) = 0
\quad\text{and}\quad
\mathcal{L}_{\dualz}^{\prime}(\concentration, \dualz)(\dualz_{\tau}) = 0\,.
\end{equation}
We only give the proof of the second equation in \eqref{eq:3:19:L_u_prime_L_z_prime_zero}.
The first one can be proved analogously.
To show that $\mathcal{L}_{\dualz}^{\prime}(\concentration, \dualz)(\dualz_{\tau}) = 0$,
we rewrite Eq.~\eqref{eq:3:19:L_u_prime_L_z_prime_zero} as
\begin{displaymath}
\displaystyle \sum_{n=1}^N \int_{I_n} \big\{
(\transportforce-\density \partial_{t} \concentration,\dualz_{\tau})
-a(\concentration, \convection)(\dualz_{\tau})
\big\}\mathrm{d}t = 0\,.
\end{displaymath}
By construction, the continuous solution $\concentration$ satisfies that
\begin{equation}
\label{eq:3:20:int_I_primal}
\displaystyle \int_{I} \big\{
(\density \partial_{t} \concentration,\varphi)+a(\concentration, \convection)(\varphi)
\big\}\mathrm{d}t = \int_{I}(\transportforce ,\varphi)\mathrm{d}t \quad
\forall\varphi \in X\,.
\end{equation}
Since $X$ is dense in $L^2(I;H_0^1(\Omega))$ with respect to the norm of
$L^2(I;H_0^1(\Omega))$ and since no time derivatives of $\varphi$
arise in \eqref{eq:3:20:int_I_primal}, this equation is also satisfied for all
$\varphi \in L^2(I;H_0^1(\Omega))$. The inclusion
$\dualz_{\tau} \in X_{\tau}^{\textnormal{dG}(r)} \subset L^2(I;H_0^1(\Omega))$
then implies that the second equation in \eqref{eq:3:19:L_u_prime_L_z_prime_zero}
is fulfilled.

Now, applying Lemma~\ref{Bruchhaeuser:Lem:1.1} with the above-made
identifications yields that
\begin{align}
\nonumber
J(\concentration)-J(\concentration_{\tau}) & =
\mathcal{L}_{\tau}(\concentration,\dualz)
- \mathcal{L}_{\tau}(\concentration_{\tau},\dualz_{\tau})
\\
\label{eq:3:21:replace_1}
& =
\frac{1}{2}\mathcal{L}_{\tau}^{\prime}(\concentration_{\tau},\dualz_{\tau})
(\concentration - \tilde{\concentration}_{\tau},\dualz - \tilde{\dualz}_{\tau})
+ \mathcal{R}_{\tau}
\,.
\end{align}
With the definition of the primal and dual residuals given
in \eqref{eq:3:15:primal_dual_residuals}, Eq.\eqref{eq:3:21:replace_1} can be
rewritten as
\begin{displaymath}
J(\concentration)-J(\concentration_{\tau}) =
\frac{1}{2}\rho_{\mathrm{t}}(\concentration_{\tau})(\dualz-\tilde{\dualz}_{\tau})
+ \frac{1}{2}\rho_{\mathrm{t}}^{\ast}(\concentration_{\tau},\dualz_{\tau})
(\concentration-\tilde{\concentration}_{\tau})
+ \mathcal{R}_{\tau}
\,,
\end{displaymath}
where the remainder term $\mathcal{R}_{\tau}$ is given by
\begin{displaymath}
\mathcal{R}_{\tau}=\frac{1}{2}\int_0^1 \mathcal{L}_{\tau}^{\prime\prime\prime}
(\concentration_{\tau}+se,\dualz_{\tau}+se^{\ast})(e,e,e,e^{\ast},e^{\ast},e^{\ast})s(s-1)
\mathrm{d}s\,,
\end{displaymath}
with the `primal' and `dual' errors $e:=\concentration-\concentration_{\tau}$ and
$e^{\ast}:=\dualz-\dualz_{\tau}$, respectively.
This proves the assertion \eqref{eq:3:16a:J_u_minus_J_u_tau}.

To prove the spatial error representation (\ref{eq:3:16b:J_u_tau_minus_J_u_tau_h}),
we apply Lemma~\ref{Bruchhaeuser:Lem:1.1} with the identifications
\begin{displaymath}
L = \mathcal{L}_{\tau}\,,\;
\tilde{L} = \mathcal{L}_{\tau h}\,,\;\;
\mathcal{Y}_1 = X_{\tau}^{\text{dG}(r)} \times X_{\tau}^{\text{dG}(r)}\,,\;\;
\mathcal{Y}_2 = X_{\tau h}^{\text{dG}(r),p} \times X_{\tau h}^{\text{dG}(r),p}
\end{displaymath}
to Eq.~\eqref{eq:3:18:J_u_tau_minus_J_u_tau_h_eq}.
In this case, we have $\mathcal Y_2\subseteq \mathcal Y_1$ since
$X_{\tau h}^{\textnormal{dG}(r),p}\subseteq X_{\tau}^{\textnormal{dG}(r)}$.
Hence, we can choose $\mathcal{Y}:=\mathcal{Y}_1$ in Lemma~\ref{Bruchhaeuser:Lem:1.1}
and condition \eqref{eq:3:12:L_prime_y1_y2} is directly satisfied.
Now, applying Lemma~\ref{Bruchhaeuser:Lem:1.1} with these identifications
implies that
\begin{align}
\nonumber
J(\concentration_{\tau})-J(\concentration_{\tau h}) & =
\mathcal{L}_{\tau}(\concentration_{\tau},\dualz_{\tau})
- \mathcal{L}_{\tau h}(\concentration_{\tau h},\dualz_{\tau h})
\\
\label{eq:3:22:replace_2}
& =
\frac{1}{2}\mathcal{L}_{\tau}^{\prime}(\concentration_{\tau h},\dualz_{\tau h})
(\concentration_{\tau} - \tilde{\concentration}_{\tau h},\dualz_{\tau}
- \tilde{\dualz}_{\tau h})
\\
\nonumber
& \qquad
+ \frac{1}{2}
(\mathcal{L}_{\tau}-\mathcal{L}_{\tau h})^{\prime}(\concentration_{\tau h},\dualz_{\tau h})
(\tilde{\concentration}_{\tau h}-\concentration_{\tau h},\tilde{\dualz}_{\tau h}-\dualz_{\tau h})
\\
\nonumber
& \qquad
+ (\mathcal{L}_{\tau}-\mathcal{L}_{\tau h})(\concentration_{\tau h},\dualz_{\tau h}) + \mathcal{R}_{h}
\,.
\end{align}
Again, using the definition in \eqref{eq:3:15:primal_dual_residuals} of the
primal and dual residual as well as the definition of $\mathcal{D}_{\tau h}$
given by Eq.~(\ref{eq:3:17:Def_D_tau_h_phi_psi}),
Eq.~(\ref{eq:3:22:replace_2}) can be
rewritten as
\begin{align*}
\nonumber
J(\concentration_{\tau})-J(\concentration_{\tau h}) & =
\frac{1}{2}\rho_{\mathrm{t}}(\concentration_{\tau h})(\dualz_{\tau}-\tilde{\dualz}_{\tau h})
+ \frac{1}{2}\rho_{\mathrm{t}}^{\ast}(\concentration_{\tau h},\dualz_{\tau h})
(\concentration_{\tau}-\tilde{\concentration}_{\tau h})
\\
\nonumber
& \qquad
+ \frac{1}{2} \mathcal{D}_{\tau h}^{\prime}(\concentration_{\tau h},\dualz_{\tau h})
(\tilde{\concentration}_{\tau h}-\concentration_{\tau h},\tilde{\dualz}_{\tau h}-\dualz_{\tau h})
\\
\nonumber
& \qquad
+ \mathcal{D}_{\tau h}(\concentration_{\tau h},\dualz_{\tau h}) + \mathcal{R}_{h}\,,
\end{align*}
where the remainder term $\mathcal{R}_{h}$ is given by
\begin{displaymath}
\mathcal{R}_{h}=\frac{1}{2}\int_0^1 \mathcal{L}_{\tau}^{\prime\prime\prime}
(\concentration_{\tau h}+se,\dualz_{\tau h}+se^{\ast})(e,e,e,e^{\ast},e^{\ast},e^{\ast})s(s-1)
\mathrm{d}s\,,
\end{displaymath}
with the `primal' and `dual' errors $e:=\concentration_{\tau}-\concentration_{\tau h}$
and
$e^{\ast}:=\dualz_{\tau}-\dualz_{\tau h}$.
This proves the assertion (\ref{eq:3:16b:J_u_tau_minus_J_u_tau_h}).
\end{proof}

\section{Practical aspects}
\label{sec:4:practical_aspects}

Here we present the underlying adaptive algorithm,
illustrate some practical aspects for the realization of the
adaptivity process as well as for the software implementation and
give the definition of the (localized) error indicators.

Our space-time adaptivity and mesh refinement strategy uses the following
algorithm.

\noindent\rule{\textwidth}{1pt}
  \begin{center}
   \textbf{Algorithm: goal-oriented space-time adaptivity}
  \end{center}
\vspace{-0.3cm}
\noindent\rule{\textwidth}{0.5pt}
\textbf{Initialization:}
Generate the initial space-time slabs
$Q^{1}_n=\Omega_h^{n,1}\times I_{\tau}^{n,1}$, $n=1,\dots,N^1$,
with $\bar{I}=\cup_n\, \bar{I}_{\tau}^{n,1}$,
for the goal-oriented adaptive transport problem.

\noindent\rule{\textwidth}{0.5pt}
\textbf{DWR-loop $\ell=1,\dots$:}
\begin{enumerate}
\item[0.] %
  Find the solutions $\{\convection_{h},\pressure_{h}\} \in Y_{h}^{p}$
  of the Stokes flow problem (\ref{eq:stokes_problem_0}), if the corresponding
  mesh has changed.

\item \textbf{Find the primal solution}
  $\concentration_{\tau h} \in X_{\tau h}^{\textnormal{primal}}$
  of problem (\ref{eq:transport_problem_0}).

\item \textbf{Break if the goal yields convergence}.

\item \textbf{Find the dual solution}
  $\dualz_{\tau h} \in X_{\tau h}^{\textnormal{dual}}$
  of problem (\ref{eq:transport_problem_0}).

\item \textbf{Evaluate the a posteriori space-time error indicators}
  $\eta_{h}$ and $\eta_{\tau}$ given by Eq. \eqref{eq:4:2:eta_h} and
   \eqref{eq:4:1:eta_tau}, respectively.

\item \textbf{Mark the slabs} $Q^{\ell}_{\tilde{n}}$, $\tilde{n}\in\{1,\dots,N^\ell\}$,
  \textbf{for temporal refinement} if the corresponding $\eta_{\tau}^{\tilde{n}}$
  is in the set of $\theta_\tau^\textnormal{top}$ percent of the worst indicators.

\item \textbf{Mark the cells} $\tilde{K} \in \Omega_h^{n,\ell}$
  \textbf{for spatial refinement} if the corresponding $\eta_h^{n}|_{\tilde{K}}$
  is in the set of $\theta_h^\textnormal{top}$ percent of the worst indicators,
  \textbf{or}, respectively,
  mark \textbf{for spatial coarsening} if $\eta_h^{n}|_{\tilde{K}}$ is in the set of
  $\theta_h^\textnormal{bottom}$ percent of the best indicators.

\item \textbf{Execute spatial adaptations} on all slabs under the use of mesh
  smoothing operators.

\item \textbf{Execute temporal refinements of slabs}.

\item Increase $\ell$ to $\ell+1$ and return to Step~0.
\end{enumerate}
\vspace{-0.3cm}
\noindent\rule{\textwidth}{0.5pt}

Regarding this algorithm, we note the following issues.
\begin{remark}
\label{rem:4:1:primal_and_dual_spaces}
~\\
\vspace{-0.6cm}
\begin{itemize}
\item The primal and dual spaces in the Steps 1 and 3 of the algorithm,
$X_{\tau h}^{\textnormal{primal}} =
X_{\tau h}^{\textnormal{dG}(r),\, \textnormal{cG}(p)}$ and
$X_{\tau h}^{\textnormal{dual}} =
X_{\tau h}^{\textnormal{dG}(s),\, \textnormal{cG}(q)}$,
must be chosen properly, i.e. $p < q$ and $r < s$.

\item Within the Steps 1, 3 and 4 of the algorithm, the computed convection field
$\convection_h$ of the Stokes problem is interpolated to the adaptively refined
spatial triangulation of the space-time slabs.

\item Technical details of the implementation are given in
\cite{BruchhaeuserKBB17}.
\end{itemize}
\end{remark}

A new software, the \texttt{dwr-stokes-condiffrea} module
of the \texttt{DTM++} project, was developed for the implementation of the
adaptive algorithm.
The new module is an extension of the published open-source module
\texttt{dwr-diffusion} of the \texttt{DTM++} project;
cf. \cite{BruchhaeuserKBB17}.
The extensions are the implementation of the space-time tensor-product
finite element spaces for the primal and dual problem on each slab and
the coupling of the transport solver to a separated flow module integrated
in the software platform.
In detail, a tensor product of the $d$-dimensional spatial finite
element space with an one-dimensional temporal finite element space is implemented.
The temporal finite element space is based on a discontinuous Galerkin
method of arbitrary order $r$ on a one-dimensional triangulation.
The temporal triangulation on a space-time slab is shared by the primal and
dual problem. The temporal polynomial degree of the primal and dual problem can be
chosen arbitrarily. These features enable the full space-time adaptivity our and
the flexible choice of polynomial degrees in space and time.
Further, the software platform provides an encapsulated flow module,
an implementation of a quasi-stationary Stokes problem, that is coupled with
a convection-diffusion transport solver. The coupling of the transport and flow
solvers is done via the convection tensor field. The latter needs to be
interpolated (or projected) from the Stokes solver to the adaptively refined
spatial meshes and for the respective degrees of freedom in time on a slab of
the transport problem.
For the future, this concept enables the realization of multirate,
fully space-time adaptive extensions of the current schemes and
the incorporation of more sophisticated flow problems.
Our simulation tools of the \texttt{DTM++} project are frontend solvers
for the \texttt{deal.II} library; cf. \cite{BruchhaeuserABCD19}.

Finally, we give some remarks regarding the localization of the error representations
that are derived in Thm.~\ref{Thm:1.1}. Their practical realization and the definition
of error indicators $\eta_{\tau}$ and $\eta_{h}$ is obtained by neglecting the
remainder terms $\mathcal{R}_{\tau}$ and $\mathcal{R}_{h}$ of the result given in
Thm.~\ref{Thm:1.1} and splitting the resulting quantities into elementwise contributions.

\begin{align}
\label{eq:4:1:eta_tau}
\nonumber
J(\concentration)-J(\concentration_{\tau}) & \doteq
\frac{1}{2}\rho_{\mathrm{t}}^{n}(\concentration_{\tau})(\dualz-\tilde{\dualz}_{\tau})
+ \frac{1}{2}\rho_{\mathrm{t}}^{\ast,n}(\concentration_{\tau},\dualz_{\tau})
(\concentration-\tilde{\concentration}_{\tau})
\\
& =: \eta_{\tau} = \displaystyle \sum_{n=1}^{N}
\eta_{\tau}^{n}\,,
\end{align}

\begin{align}
\label{eq:4:2:eta_h}
\nonumber
J(\concentration_{\tau})-J(\concentration_{\tau h}) & \doteq
\frac{1}{2}\rho_{\mathrm{t}}^{n}(\concentration_{\tau h})(\dualz_{\tau}-\tilde{\dualz}_{\tau h})
+ \frac{1}{2}
\rho_{\mathrm{t}}^{\ast,n}(\concentration_{\tau h},\dualz_{\tau h})
(\concentration_{\tau}-\tilde{\concentration}_{\tau h})
\\
\nonumber
&
+ \frac{1}{2} \mathcal{D}_{\tau h}^{\prime,n}(\concentration_{\tau h},\dualz_{\tau h})
(\tilde{\concentration}_{\tau h}-\concentration_{\tau h},\tilde{\dualz}_{\tau h}-\dualz_{\tau h})
+ \mathcal{D}_{\tau h}^{n}(\concentration_{\tau h},\dualz_{\tau h})
\\
& =: \eta_h = \displaystyle \sum_{n=0}^{N} \eta_{h}^{n}
= \displaystyle \sum_{n=0}^{N}
 \sum\limits_{K\in\mathcal{T}_h^n} \eta_{h,K}^{n}\,.
\end{align}

To compute the error indicators $\eta_{\tau}$ and $\eta_{h}$ we replace
all unknown solutions by the approximated fully discrete solutions
$\concentration_{\tau h} \in X_{\tau h}^{\textnormal{dG}(r),\, \textnormal{cG}(p)}$
and
$\dualz_{\tau h} \in X_{\tau h}^{\textnormal{dG}(s),\, \textnormal{cG}(q)}$,
with $r < s$ and $p < q$,
whereby the arising weights are approximated in the following way.

\begin{itemize}
\itemsep1.5ex

\item We put $\concentration-\tilde{\concentration}_{\tau} \approx
\operatorname{E}_{\tau}^{r+1}(\concentration_{\tau h}) - \concentration_{\tau h}$
with $\operatorname{E}_{\tau}^{r+1}(\cdot)$ denoting the extrapolation in time operator
thats acts on a time cell of length $\tau$ and lifts the solution to a piecewise polynomial
of degree ($r$+$1$) in time.

We note that the additional solution for the (local) extrapolation in time on a
specific time cell is here interpolated from the previous time cell or the
initial condition $\concentration_0$ in the left end of the time cell. For this,
the previous time cell is located on the same slab $Q_n^\ell$
or the previous one $Q_{n-1}^\ell$, where the latter case requires an
additional interpolation between two spatial triangulations.

In future works, this concept can be extended to a patchwise higher-order
extrapolation in time by using the solutions of two neighboring time cells
with an order in time of ($2r$+$1$) on the $2\tau$ patch time cell.

\item We put $\concentration_{\tau}-\tilde{\concentration}_{\tau h} \approx
\operatorname{E}_{2h}^{2p}(\concentration_{\tau h}) - \concentration_{\tau h}$,
with $\operatorname{E}_{2h}^{2p}(\cdot)$ denoting the extrapolation in space operator
that acts on a patched cell of size $2h$ and lifts the solution to a piecewise polynomial
of degree $2p$ on the reference cell corresponding to the patched cell of width $2h$.

We note that the extrapolation operator in space is implemented in the
\texttt{deal.II} library for quadrilateral and hexahedral finite elements and
continuous discrete functions of piecewise polynomials with degree $p$ in each
variable. The application of this operator requires the spatial triangulation
on each slab being at least once globally refined to construct the patched cells
of size $2h$.
Thereby, an extrapolation degree of at most $2p$ (and not $2p+1$) is due to the
shared degrees of freedom on the edges or faces of the continuous FE solution.

\item We put $\dualz-\tilde{\dualz}_{\tau} \approx
\dualz_{\tau h} - \operatorname{R}_{\tau}^{r}(\dualz_{\tau h})$
with $\operatorname{R}_{\tau}^{r}(\cdot)$ denoting the restriction in time operator
on a time cell that restricts the solution to a polynomial of degree $r < s$.

We note that the restriction operator in time is implemented in our software since
deal.II is currently not able to operate on ($d$+$1$)-dimensional tensor-product
solutions. This is done by a Lagrangian interpolation in time to the primal
space of the dual solution and an additional interpolation back to the dual space.

\item We put $\dualz_{\tau}-\tilde{\dualz}_{\tau h} \approx
\dualz_{\tau h} - \operatorname{R}_{h}^{p}(\dualz_{\tau h})$ with
$\operatorname{R}_{h}^{p}(\cdot)$ denoting the restriction in space operator
that acts on a spatial cell and restricts the solution to a polynomial
of degree $p < q$ on the corresponding reference cell.

We note that the restriction operator in space is implemented in the
\texttt{deal.II} library for dimension $d=2,3$ as back-interpolation operator
between two finite element spaces that are here the dual finite element space
and the intermediate primal finite element space.
\end{itemize}

We recall that within the DWR framework the respective weights have to be
approximated using a suitable technique. The most common way for this
approximation is the application of a patch-wise higher-order extrapolation;
cf.~\cite{BruchhaeuserBR12,BruchhaeuserBGR10,BruchhaeuserSV08}.
In this work, the respective operators are chosen in the described manner
due to the specific character of linear convection-dominated problems.
Our motivation for this results from a comparative study between
higher-order extrapolation and higher-order finite element approximations that 
has been done for a steady-state variant of a convection-dominated problem in
\cite{BruchhaeuserBSB19}.

Since we do not consider an analytical solution for the Stokes flow problem
\eqref{eq:stokes_problem_0} in some of our numerical examples in
Sec.~\ref{sec:5:examples},
we neglect the coupling terms appearing in (\ref{eq:4:2:eta_h}).
Instead, we approximate $\{\convection_h,\pressure_h\}$ on a sufficiently
refined mesh with the Stokes stable finite element pairs
$Q_2-Q_1$ or even $Q_4-Q_2$ to avoid spatial approximation errors in the transport
problem.

For measuring the accuracy of the error estimator, we will study in our numerical
convergence experiments (cf.\ Sec.~\ref{sec:5:1}) the effectivity index
\begin{equation}
\label{eq:4:3:Ieff}
\mathcal{I}_{\textnormal{eff}} = \left|\frac{\eta_{\tau}+\eta_{h}}{J(\concentration)
-J(\concentration_{\tau h})}\right|
\end{equation}
as the ratio of the estimated error over the exact error.
Desirably, the index $\mathcal{I}_{\textnormal{eff}}$ should be close to one.

\section{Numerical examples}
\label{sec:5:examples}

In the following section we study the convergence, computational efficiency
and stability of the introduced goal-oriented adaptivity approach for
the coupled transport and flow problem.
The first example in Sec.~\ref{sec:5:1} is an academic problem with a given
analytical solution to study the space-time higher-order convergence behavior
with a constant convection field for a non-stabilized convection-diffusion
transport and a stabilized convection-dominated transport.
The lowest-order results can be compared with our preceding published works
\cite{BruchhaeuserBSB18,BruchhaeuserKBB17}.
The second example in Sec.~\ref{sec:5:2} is motivated by problem of physical
relevance in which we simulate a convection-dominated transport with goal-oriented
adaptivity of a species through a channel with a constraint.

\subsection{Example 1 (Stabilized higher-order space-time convergence studies)}
\label{sec:5:1}

This first example is an academic test problem with the given solution
\begin{equation}
\label{eq:5:1:cone}
\begin{array}{l@{\,}c@{\,}l}
u(\boldsymbol x, t) &:=&
u_1 \cdot u_2\,,\,\,
\boldsymbol x = (x_1, x_2)^\top \in \mathbb{R}^2 \text{ and }
t \in \mathbb{R}\,,\\[.5ex]
u_1(\boldsymbol x, t) &:=&
  (1 + a \cdot ( (x_1 - m_1(t))^2 + (x_2 - m_2(t))^2  ) )^{-1}\,,\\[.5ex]
u_2(t) &:=& \nu_1(t) \cdot s \cdot \arctan( \nu_2(t) )\,,
\end{array}
\end{equation}
with $m_1(t) := \frac{1}{2} + \frac{1}{4} \cos(2 \pi t)$ and
$m_2(t) := \frac{1}{2} + \frac{1}{4} \sin(2 \pi t)$, and,
$\nu_1(\hat t) := -1$,
$\nu_2(\hat t) := 5 \pi \cdot (4 \hat t - 1)$,
for $\hat t \in [0, 0.5)$ and
$\nu_1(\hat t) := 1$,
$\nu_2(\hat t) := 5 \pi \cdot (4 (\hat t-0.5) - 1)$,
for $\hat t \in [0.5, 1)$, $\hat t = t - k$,
$k \in \mathbb{N}_0$, and,
scalars $a = 50$ and $s=-\frac{1}{3}$.
The (analytic) solution \eqref{eq:5:1:cone} mimics a counterclockwise rotating
cone which additionally changes its height and orientation over the period
$T=1$. Precisely, the orientation of the cone switches from negative to positive
while passing $t=0.25$ and from positive to negative while passing $t=0.75$.
Exemplary solution profiles at $t=0$, $t=0.33$, $t=0.50$ and $t=0.85$
are illustrated in Fig.~\ref{fig:1:cone:profiles}.

\begin{figure}[ht]
\centering

\includegraphics[width=.24\linewidth]{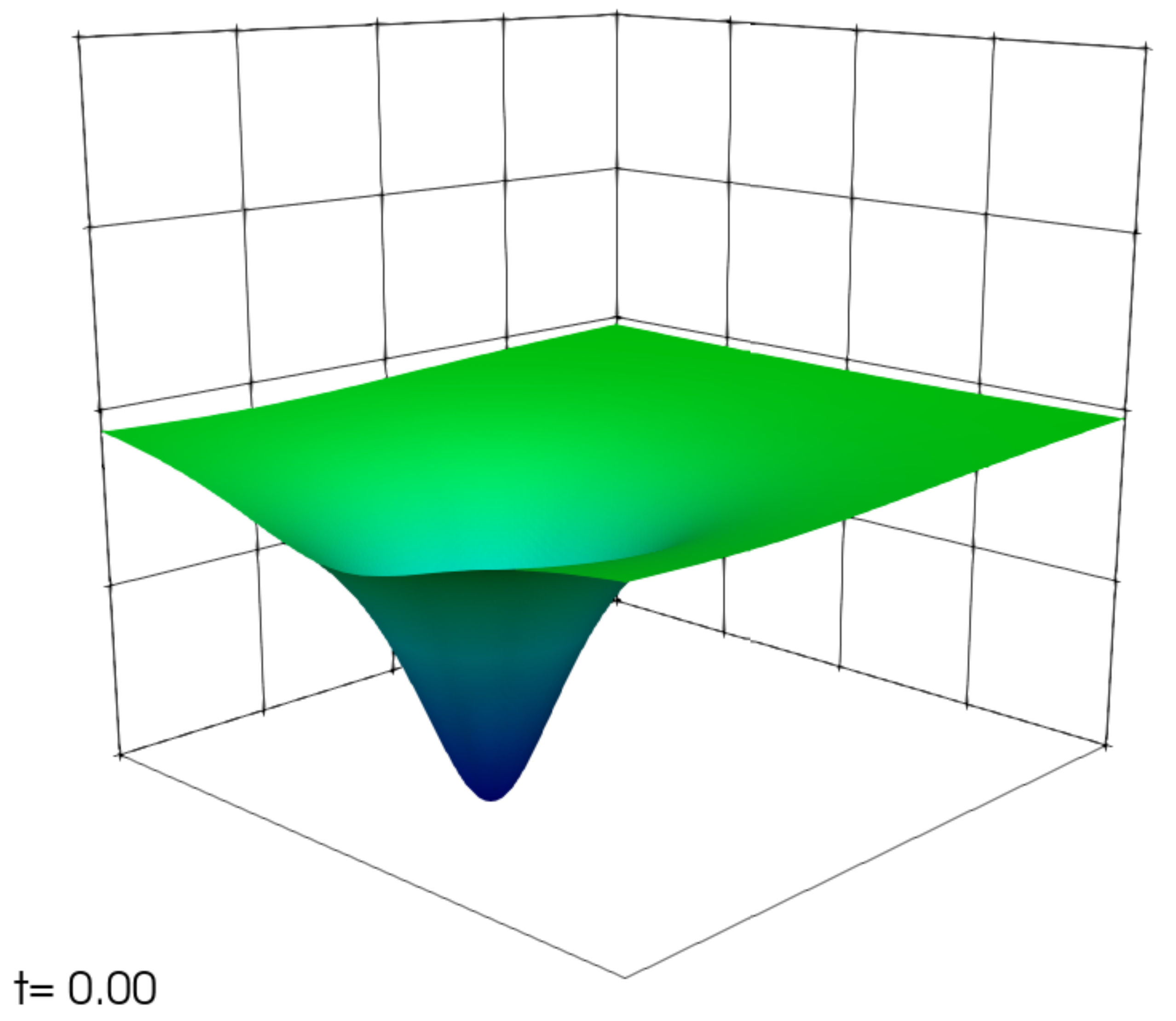}
\includegraphics[width=.24\linewidth]{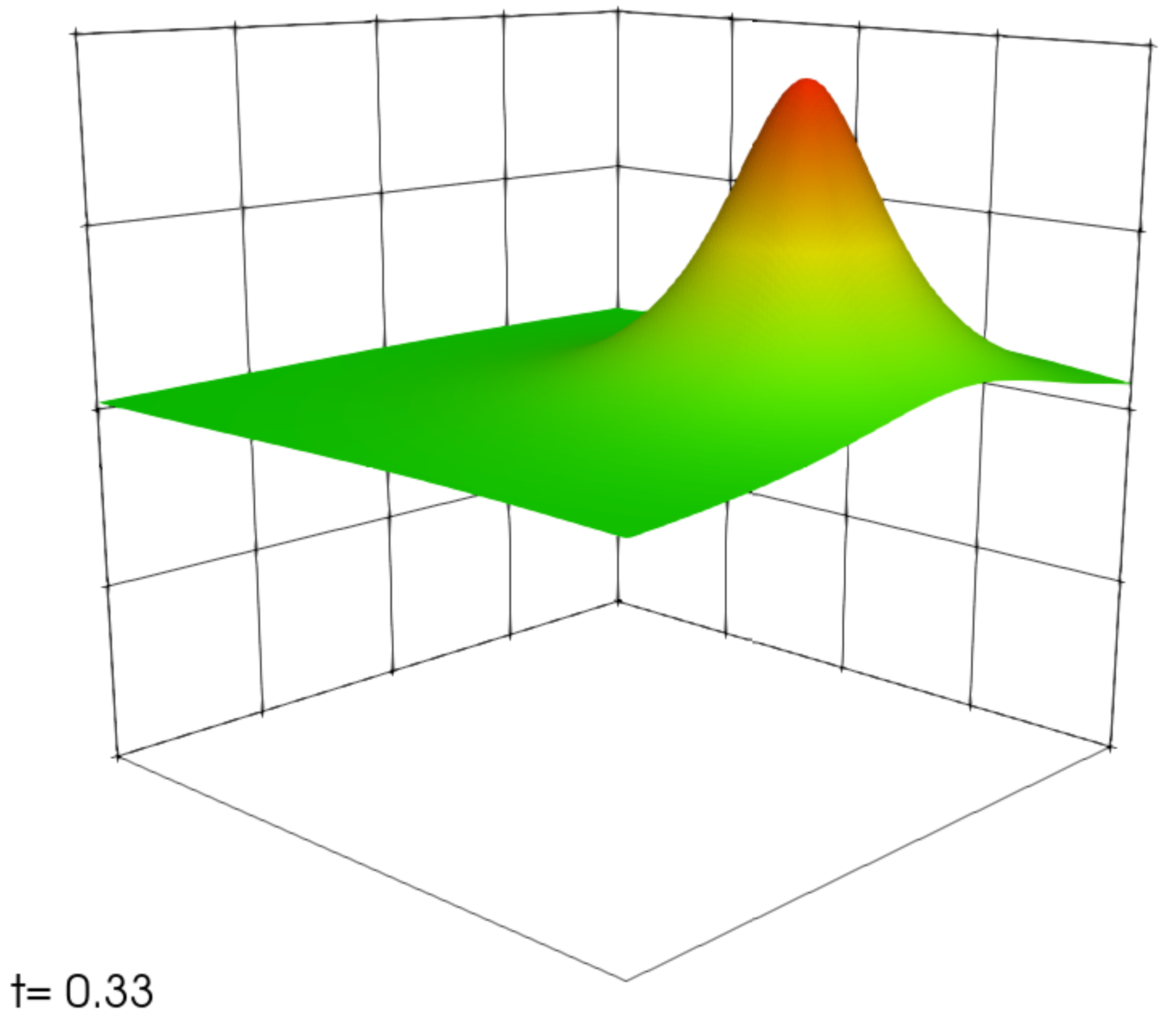}
\includegraphics[width=.24\linewidth]{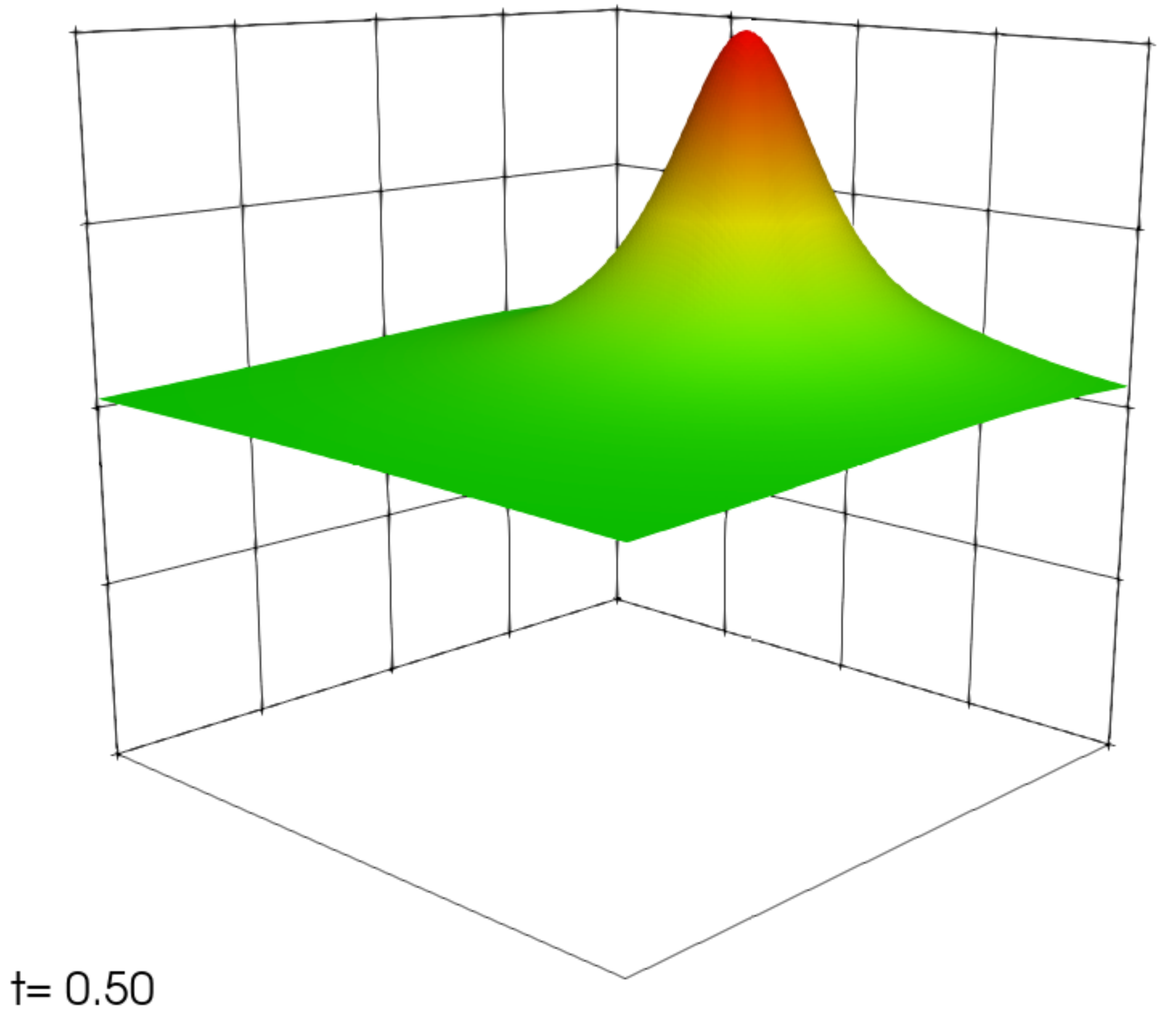}
\includegraphics[width=.24\linewidth]{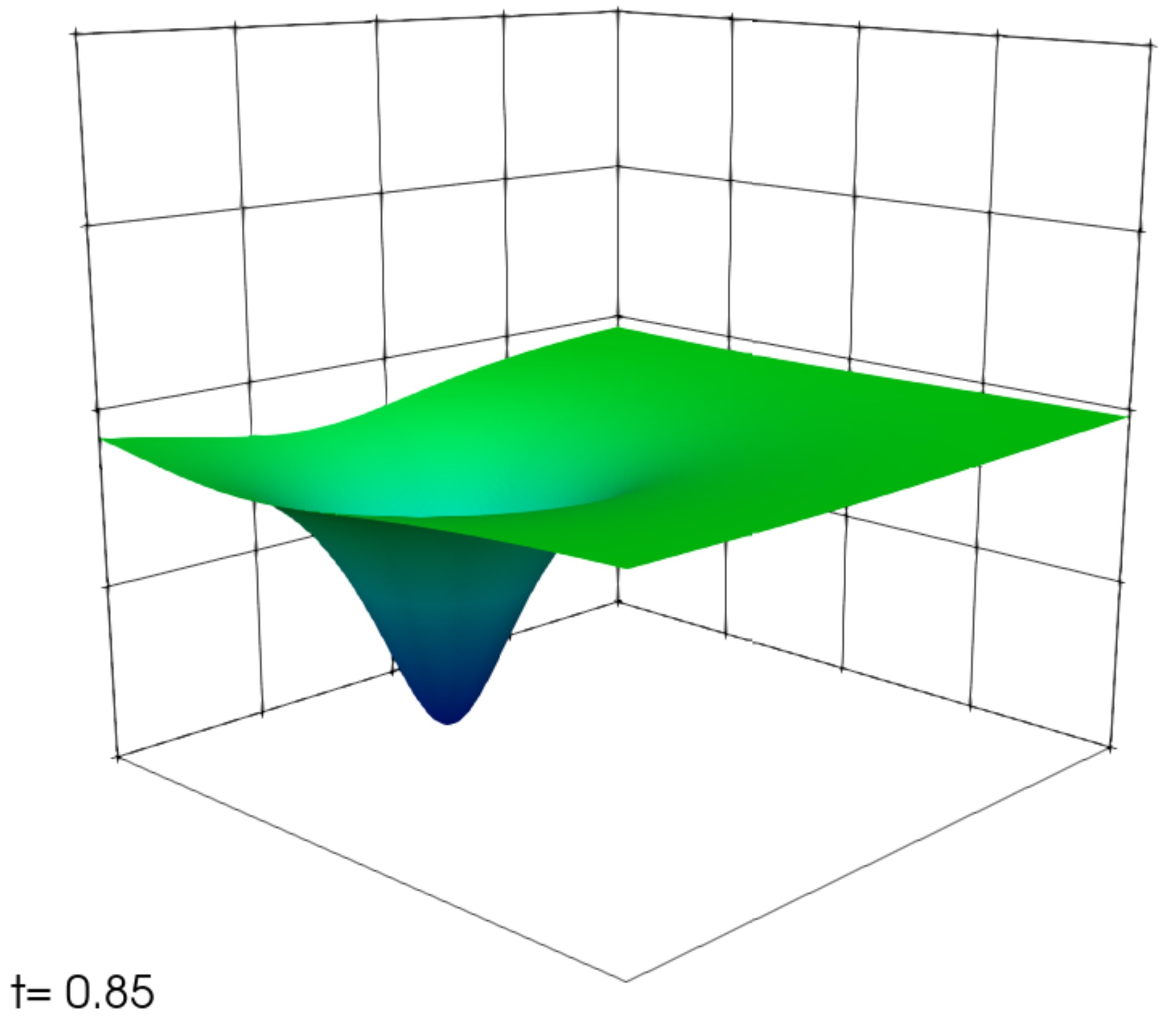}

\caption{Solution profiles $u_{\tau h}$ at $t=0$, $t=0.33$, $t=0.50$ and $t=0.85$
for Sec.~\ref{sec:5:1}.}
\label{fig:1:cone:profiles}
\end{figure}

The right-hand side forcing term $\transportforce$, the inhomogeneous Dirichlet
boundary condition and the inhomogeneous initial condition are calculated from
the given analytic solution \eqref{eq:5:1:cone} and
Eqs. \eqref{eq:transport_problem_0} (a)-(c).
Our target quantity is chosen to control the global $L^2(L^2)$-error of $e$,
$e = \concentration - \concentration_{\tau h}$, in space and time, given by
\begin{equation}
\label{eq:5:2:L2Goal}
J(\varphi)= \frac{1}{\|e\|_{(0,T)\times\Omega}}
\displaystyle\int_I(\varphi,e)\mathrm{d}t\,,
\quad \mathrm{with} \;\; \|\cdot\|_{(0,T)\times\Omega}
=\left(\int_I(\cdot,\cdot)\;\mathrm{d}t\right)^{\frac{1}{2}}\,.
\end{equation}
The tuning parameters of the goal-oriented adaptive Algorithm given in
Sec.~\ref{sec:4:practical_aspects} are chosen here in a way to balance
automatically the potential misfit of the spatial and temporal errors as
\begin{displaymath}
\theta_h^\textnormal{top} = 0.5 \cdot
\left| \frac{\eta_h}{\eta_h + \eta_\tau} \right|\,,\quad
\theta_h^\textnormal{bottom} = 0\quad
\textnormal{and}\quad
\theta_\tau^\textnormal{top} = 0.5 \cdot
\left| \frac{\eta_\tau}{\eta_h + \eta_\tau} \right|\,.
\end{displaymath}

\begin{table}[ht]
\centering

\resizebox{0.90\linewidth}{!}{%
\begin{tabular}{c rrr | cc}
\hline
\hline
$\ell$ & $N$ & $N_K$ & $N_{\text{DoF}}^{\text{tot}}$ &
$||u-u_{\tau h}^{1,1}||$ & EOC\\
\hline
1 & 4   & 4     & 72       & 8.7470e-02 & ---\\
2 & 8   & 16    & 400      & 2.7935e-02 & 1.65\\
3 & 16  & 64    & 2592     & 9.0984e-03 & 1.62\\
4 & 32  & 256   & 18496    & 3.0196e-03 & 1.59\\
5 & 64  & 1024  & 139392   & 7.6942e-04 & 1.97\\
6 & 128 & 4096  & 1081600  & 1.9290e-04 & 2.00\\
7 & 256 & 16384 & 8520192  & 4.9303e-05 & 1.97\\
8 & 512 & 65536 & 67634176 & 1.2490e-05 & 1.98\\
\hline
\end{tabular}
~~~~
\begin{tabular}{c rrr | cc}
\hline
\hline
$\ell$ & $N$ & $N_K$ & $N_{\text{DoF}}^{\text{tot}}$ &
$||u-u_{\tau h}^{2,2}||$ & EOC\\
\hline
1 & 5   & 4     & 375      & 5.3652e-02 & ---\\
2 & 10  & 16    & 2430     & 1.3714e-02 & 1.97\\
3 & 20  & 64    & 17340    & 1.5497e-03 & 3.15\\
4 & 40  & 256   & 130680   & 1.9911e-04 & 2.96\\
5 & 80  & 1024  & 1014000  & 3.5649e-05 & 2.48\\
6 & 160 & 4096  & 7987680  & 5.3813e-06 & 2.73\\
7 & 320 & 16384 & 63407040 & 6.8615e-07 & 2.97\\
~ & \\
\hline
\end{tabular}
}

\caption{Global convergence for $u_{\tau h}^{1,1}$ in a cG(1)-dG(1)
and $u_{\tau h}^{2,2}$ in a cG(2)-dG(2) primal approximation for a
convection-diffusion transport problem with $\varepsilon=1$ and $\delta_0=0$
for Sec.~\ref{sec:5:1}.
$\ell$ denotes the refinement level, $N$ the total cells in time,
$N_K$ the cells in space on a slab,
$N_{\text{DoF}}^{\text{tot}}$ the total space-time degrees of freedom,
$||\cdot||$ the global $L^2(L^2)$-norm error and
EOC the experimental order of convergence.
}
\label{tab:1:cone:global}
\end{table}

\begin{table}[ht]
\centering

\resizebox{0.90\linewidth}{!}{%
\begin{tabular}{c rrr | cc}
\hline
\hline
$\ell$ & $N$ & $N_K$ & $N_{\text{DoF}}^{\text{tot}}$ &
$||u-u_{\tau h}^{1,1}||$ & EOC\\
\hline
1 &   4 &     4 &       72 & 5.8984e-02 & ---\\
2 &   8 &    16 &      400 & 4.0001e-02 & 0.56\\
3 &  16 &    64 &     2592 & 1.5534e-02 & 1.36\\
4 &  32 &   256 &    18496 & 6.0496e-03 & 1.36\\
5 &  64 &  1024 &   139392 & 2.2615e-03 & 1.42\\
6 & 128 &  4096 &  1081600 & 1.0633e-03 & 1.09\\
7 & 256 & 16384 &  8520192 & 5.2811e-04 & 1.01\\
8 & 512 & 65536 & 67634176 & 2.6493e-04 & 1.00\\
\hline
\end{tabular}
~~~~
\begin{tabular}{c rrr | cc}
\hline
\hline
$\ell$ & $N$ & $N_K$ & $N_{\text{DoF}}^{\text{tot}}$ &
$||u-u_{\tau h}^{2,2}||$ & EOC\\
\hline
1 &   4 &     4 &      300 & 4.3801e-02 & ---\\
2 &   8 &    16 &     1944 & 1.7138e-02 & 1.35\\
3 &  16 &    64 &    13872 & 7.9707e-03 & 1.10\\
4 &  32 &   256 &   104544 & 3.4451e-03 & 1.21\\
5 &  64 &  1024 &   811200 & 1.7150e-03 & 1.01\\
6 & 128 &  4096 &  6390144 & 8.6229e-04 & 0.99\\
7 & 256 & 16384 & 50725632 & 4.3314e-04 & 0.99\\
~ & \\
\hline
\end{tabular}
}

\caption{Global convergence for $u_{\tau h}^{1,1}$ in a cG(1)-dG(1)
and $u_{\tau h}^{2,2}$ in a cG(2)-dG(2) primal approximation for a
convection-dominated transport problem with $\varepsilon=10^{-6}$ and
$\delta_0=10^{-1}$ for Sec.~\ref{sec:5:1}.
$\ell$ denotes the refinement level, $N$ the total cells in time,
$N_K$ the cells in space on a slab,
$N_{\text{DoF}}^{\text{tot}}$ the total space-time degrees of freedom,
$||\cdot||$ the global $L^2(L^2)$-norm error and
EOC the experimental order of convergence.
}
\label{tab:2:cone:global}
\end{table}

\begin{figure}[h]
\centering

\begin{tikzpicture}
\begin{footnotesize}
\begin{axis}[%
width=4.50in,
height=1.70in,
scale only axis,
xlabel={$N_{\text{DoF}}^{\text{tot}}$: accumulated total primal space-time degrees of freedom},
ylabel={$|| u - u_{\tau h} ||$},
xmode=log,
xmin=3e1,
xmax=67634176,
xminorticks=true,
ymode=log,
ymin=1e-4,
ymax=1e-1,
]


\addplot[
color=black,
dotted,
line width=1.0pt
]
table[row sep=crcr]{
72       8.7470e-02\\
400      2.7935e-02\\
2592     9.0984e-03\\
18496    3.0196e-03\\
139392   7.6942e-04\\
1081600  1.9290e-04\\
8520192  4.9303e-05\\
67634176 1.2490e-05\\
};

\addlegendentry{glob. ref. $\concentration_{\tau h}^{1,1}$};

\addplot[
color=black,
solid,
line width=1.0pt
]
table[row sep=crcr]{
375      5.3652e-02\\
2430     1.3714e-02\\
17340    1.5497e-03\\
130680   1.9911e-04\\
1014000  3.5649e-05\\
7987680  5.3813e-06\\
63407040 6.8615e-07\\
};

\addlegendentry{glob. ref. $\concentration_{\tau h}^{2,2}$};

\addplot[
color=red,
densely dashdotted,
line width=1.0pt
]
table[row sep=crcr]{
36    7.1588e-02\\
125   3.6641e-02\\
246   2.3021e-02\\
682   2.0517e-02\\
1181  1.6786e-02\\
1959  1.5438e-02\\
3085  9.3565e-03\\
4878  5.9652e-03\\
7610  3.9673e-03\\
11786 3.2686e-03\\
};

\addlegendentry{adap. $\{\concentration_{\tau h}^{1,0}, \dualz_{\tau h}^{2,1}\}$};

\addplot[
color=green,
densely dashed,
line width=1.0pt
]
table[row sep=crcr]{
72    8.7470e-02\\
250   3.1113e-02\\
492   2.1156e-02\\
1468  1.4856e-02\\
2666  8.0165e-03\\
4456  3.9318e-03\\
6960  3.1346e-03\\
10610 1.8780e-03\\
16522 1.4014e-03\\
26252 1.0214e-03\\
};

\addlegendentry{adap. $\{\concentration_{\tau h}^{1,1}, \dualz_{\tau h}^{2,2}\}$};

\addplot[
color=blue,
densely dotted,
line width=1.0pt
]
table[row sep=crcr]{
300    5.7763e-02\\
1215   1.4581e-02\\
2538   1.0721e-02\\
5112   5.4041e-03\\
11325  2.1878e-03\\
18489  9.1402e-04\\
29895  4.4147e-04\\
43026  2.6759e-04\\
66270  1.6078e-04\\
101178 1.1594e-04\\
};

\addlegendentry{adap. $\{\concentration_{\tau h}^{2,2}, \dualz_{\tau h}^{3,3}\}$};

\addplot[
color=black,
solid,
line width=.1pt
]
table[row sep=crcr]{
1 3.2686e-03\\
11786 3.2686e-03\\
};

\addplot[
color=black,
solid,
line width=.1pt
]
table[row sep=crcr]{
12000 1e-01\\
12000 1e-08\\
};

\addplot[
color=black,
solid,
line width=.1pt
]
table[row sep=crcr]{
26000 1e-01\\
26000 1e-08\\
};

\end{axis}

\end{footnotesize}
\end{tikzpicture}

\caption{$L^2(L^2)$-error reduction in the convection-diffusion transport case
with $\varepsilon = 1$ without stabilization $\delta_0=0$ for Sec.~\ref{sec:5:1}.
The solution approximations are
$\concentration_{\tau h}^{1,0}$ in cG(1)-dG(0),
$\concentration_{\tau h}^{1,1}$ in cG(1)-dG(1),
$\concentration_{\tau h}^{2,2}$ in cG(2)-dG(2)
and the dual solution approximations are
$\dualz_{\tau h}^{2,1}$ in cG(2)-dG(1),
$\dualz_{\tau h}^{2,2}$ in cG(2)-dG(2),
$\dualz_{\tau h}^{3,3}$ in cG(3)-dG(3).}
\label{fig:2:cone:plot}
\end{figure}
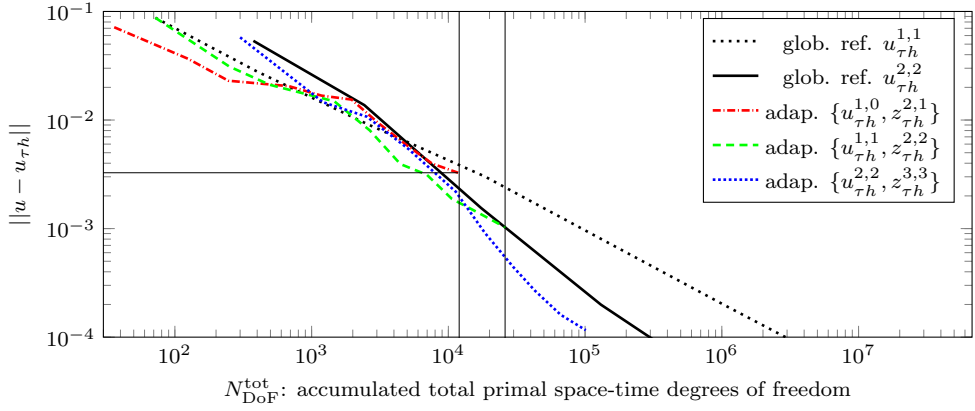

\begin{figure}[h]
\centering

\begin{tikzpicture}
\begin{footnotesize}
\begin{axis}[%
width=4.50in,
height=1.70in,
scale only axis,
xlabel={$N_{\text{DoF}}^{\text{tot}}$: accumulated total primal space-time degrees of freedom},
ylabel={$|| u - u_{\tau h} ||$},
xmode=log,
xmin=3e1,
xmax=67634176,
xminorticks=true,
ymode=log,
ymin=1e-4,
ymax=1e-1,
legend style={legend pos=south west},
]


\addplot[
color=black,
dotted,
line width=1.0pt
]
table[row sep=crcr]{
72       5.8984e-02\\
400      4.0001e-02\\
2592     1.5534e-02\\
18496    6.0496e-03\\
139392   2.2615e-03\\
1081600  1.0633e-03\\
8520192  5.2811e-04\\
67634176 2.6493e-04\\
};

\addlegendentry{glob. ref. $\concentration_{\tau h}^{1,1}$};

\addplot[
color=black,
solid,
line width=1.0pt
]
table[row sep=crcr]{
300      4.3801e-02\\
1944     1.7138e-02\\
13872    7.9707e-03\\
104544   3.4451e-03\\
811200   1.7150e-03\\
6390144  8.6229e-04\\
50725632 4.3314e-04\\
};

\addlegendentry{glob. ref. $\concentration_{\tau h}^{2,2}$};

\addplot[
color=red,
densely dashdotted,
line width=1.0pt
]
table[row sep=crcr]{
36    9.1527e-02\\
125   6.6916e-02\\
330   6.2426e-02\\
804   5.5849e-02\\
1265  4.3011e-02\\
2037  3.0684e-02\\
3283  2.1294e-02\\
5262  1.5710e-02\\
8357  1.2567e-02\\
13121 9.9855e-03\\
21851 8.2167e-03\\
35812 6.9471e-03\\
56496 5.4105e-03\\
92001 4.5907e-03\\
149673 3.7872e-03\\
245837 3.0636e-03\\
388841 2.5930e-03\\
611872 2.1692e-03\\
975317 1.7862e-03\\
1571582 1.5070e-03\\
2490706 1.2315e-03\\
3847118 1.0341e-03\\
5856334 8.7489e-04\\
9306237 7.5828e-04\\
14602984 6.1936e-04\\
22227774 5.3874e-04\\
33920429 4.6419e-04\\
52807201 4.0509e-04\\
};

\addlegendentry{adap. $\{\concentration_{\tau h}^{1,0}, \dualz_{\tau h}^{2,1}\}$};

\addplot[
color=green,
densely dashed,
line width=1.0pt
]
table[row sep=crcr]{
72    5.8984e-02\\
250   5.3935e-02\\
660   4.2387e-02\\
1634  2.4328e-02\\
2668  1.5031e-02\\
4316  8.9709e-03\\
7032  5.9470e-03\\
10668 4.2404e-03\\
17270 3.2138e-03\\
41550 2.9555e-03\\
70600 2.8066e-03\\
116914 2.2909e-03\\
200612 2.1275e-03\\
331022 1.8869e-03\\
540428 1.5635e-03\\
901690 1.4340e-03\\
1444372 1.2802e-03\\
2252002 1.0818e-03\\
3958710 9.7734e-04\\
6700404 8.5644e-04\\
10351646 7.4666e-04\\
16647900 6.3712e-04\\
26042484 5.6586e-04\\
39272854 4.8490e-04\\
63216858 4.0937e-04\\
};

\addlegendentry{adap. $\{\concentration_{\tau h}^{1,1}, \dualz_{\tau h}^{2,2}\}$};

\addplot[
color=blue,
densely dotted,
line width=1.0pt
]
table[row sep=crcr]{
300    4.3801e-02\\
1215   2.8964e-02\\
2538   2.6314e-02\\
4719   1.6058e-02\\
8175   9.9382e-03\\
14442  7.1355e-03\\
31548  5.2461e-03\\
66042  4.4188e-03\\
105825 3.9399e-03\\
164886 3.2102e-03\\
243975 2.8520e-03\\
357708 2.6024e-03\\
489993 2.3287e-03\\
717525 2.2987e-03\\
991206 2.2536e-03\\
1342650 2.2246e-03\\
1781844 2.0076e-03\\
2409129 1.8148e-03\\
3234252 1.6807e-03\\
4296147 1.5911e-03\\
5806080 1.5730e-03\\
7777026 1.5502e-03\\
11325789 1.4811e-03\\
14650563 1.4064e-03\\
20420922 1.3293e-03\\
30198813 1.2394e-03\\
40048821 1.1975e-03\\
};

\addlegendentry{adap. $\{\concentration_{\tau h}^{2,2}, \dualz_{\tau h}^{3,3}\}$};

\addplot[
color=black,
solid,
line width=.1pt
]
table[row sep=crcr]{
1 9.9855e-03\\
13121 9.9855e-03\\
};

\addplot[
color=black,
solid,
line width=.1pt
]
table[row sep=crcr]{
13000 1e-01\\
13000 1e-08\\
};

\addplot[
color=black,
solid,
line width=.1pt
]
table[row sep=crcr]{
42000 1e-01\\
42000 1e-08\\
};

\end{axis}

\end{footnotesize}
\end{tikzpicture}

\caption{$L^2(L^2)$-error reduction in the stabilized convection-dominated transport case
with $\varepsilon = 10^{-6}$ and $\delta_0=10^{-1}$ for Sec.~\ref{sec:5:1}.
The solution approximations are
$\concentration_{\tau h}^{1,0}$ in cG(1)-dG(0),
$\concentration_{\tau h}^{1,1}$ in cG(1)-dG(1),
$\concentration_{\tau h}^{2,2}$ in cG(2)-dG(2)
and the dual solution approximations are
$\dualz_{\tau h}^{2,1}$ in cG(2)-dG(1),
$\dualz_{\tau h}^{2,2}$ in cG(2)-dG(2),
$\dualz_{\tau h}^{3,3}$ in cG(3)-dG(3).}
\label{fig:3:cone:plot}
\end{figure}
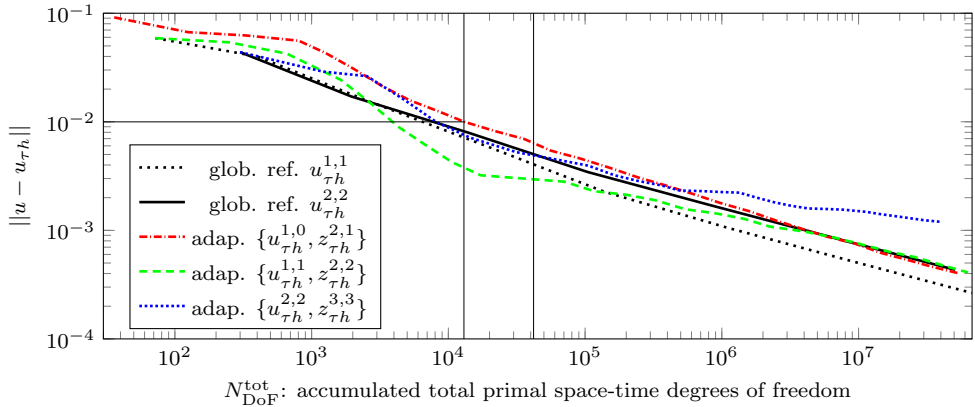

The convergence behavior, computationally efficiency and stability
for higher-order space-time discretizations with and without stabilization
is studied in the following of this numerical experiment. Additionally,
we compare global against goal-oriented space-time adaptivity.
We use here a constant convection field $\convection = (2,3)^\top$
for a non-stabilized convection-diffusion with $\varepsilon=1$ and $\delta_0=0$
transport and a stabilized convection-dominated transport with $\varepsilon=10^{-6}$
and $\delta_0=10^{-1}$
together with the constant reaction coefficient $\alpha=1$ and
the density $\density=1$.
The local SUPG stabilization coefficient is here
$\delta_K = \delta_0 \cdot h_K$ where $h_K$ denotes the cell diameter of the
spatial mesh cell $K$.

Initially, we study the global space-time refinement behavior for the
test case with $\varepsilon=1$ and vanishing stabilization
to show the correctness of the higher-order implementation. Therefore, the solution
$\concentration$ is approximated with the higher-order in time method
cG(1)-dG(1) and with the space-time higher-order method cG(2)-dG(2).
Due to the same polynomial orders of the spatial and temporal discretizations,
we expect experimental orders of convergence
(EOC $:= -\log_2(|| e ||_\ell / || e ||_{\ell-1})$) of
$\textnormal{EOC}^{1,1} \approx 2$ for the cG(1)-dG(1) method and
$\textnormal{EOC}^{2,2} \approx 3$ for the cG(2)-dG(2) method
for a global refinement convergence test.
The results are given by Tab.~\ref{tab:1:cone:global} and nicely confirm
our expected results for the cG(1)-dG(1) method and roughly confirm the expected
results for the cG(2)-dG(2) method. The EOC of the cG(2)-dG(2) is not perfect
since the initial space-time errors are not well balanced for this example.

The error reduction of the global space-time refinement results
for the convection-dominated case with $\varepsilon=10^{-6}$ and
stabilization with $\delta_0=10^{-1}$
are given by Tab.~\ref{tab:2:cone:global}.
Both methods, that are the cG(1)-dG(1) and the cG(2)-dG(2) approximations
of the primal solution $\concentration_{\tau h}$, are limited in their
experimental order of convergence of approximately $1$.
This is not a suprising result since the regularity of the convection-dominated
test case is typically of low order.
Therefore, we expect further the lower-order or maybe the lowest-order
goal-oriented adaptivity methods to perform better then higher-order methods for
the convection-dominated test case.

Secondly, we study the goal-oriented space-time adaptivity behavior.
Precisely, we compare the solution and dual solution approximation pairings
$\{\concentration_{\tau h}, \dualz_{\tau h}\}$:
cG(1)-dG(0)/cG(2)-dG(1), cG(1)-dG(1)/cG(2)-dG(2) and cG(2)-dG(2)/cG(3)-dG(3).
The lowest-order results can be compared with the results of our preceding
published work \cite{BruchhaeuserBSB18},
while remarking that a cG(1)-dG(0)/cG(2)-cG(1) discretization was used there
combined with a different choice for the tuning parameters.
The results are given by Fig.~\ref{fig:2:cone:plot} and
by Tab.~\ref{tab:3:cone:ada}-\ref{tab:5:cone:ada}.
Here, the adaptive method of highest order, i.e. the approximation by
cG(2)-dG(2)/cG(3)-dG(3), outperforms all other methods.

Finally, we study the goal-oriented space-time adaptivity behavior for the
convection-dominated case with $\varepsilon=10^{-6}$.
Precisely, we compare the solution and dual solution approximation pairings
$\{\concentration_{\tau h}, \dualz_{\tau h}\}$:
cG(1)-dG(0)/cG(2)-dG(1), cG(1)-dG(1)/cG(2)-dG(2) and cG(2)-dG(2)/cG(3)-dG(3).
The lowest-order results can be compared with the results of our preceding
published work \cite{BruchhaeuserBSB18},
while remarking that a cG(1)-dG(0)/cG(2)-cG(1) discretiatzion was used there
combined with a different choice for the tuning parameters.
The results are given by Fig.~\ref{fig:3:cone:plot} and
by Tab.~\ref{tab:6:cone:ada}-\ref{tab:8:cone:ada}.
Here, the low order method, but not lowest order, i.e. the approximation by
cG(1)-dG(1)/cG(2)-dG(2), outperforms all other methods.

\clearpage

\begin{table}[h!]
\centering

\resizebox{0.90\linewidth}{!}{%
\begin{tabular}{c | rrr | cc | ccc | c}
\hline
\hline
$\ell$ & $N$ & $N_K^{\text{max}}$ & $N_{\text{DoF}}^{\text{tot}}$ &
$\|e^{1,0,2,1}\|$ & EOC &
${\eta}_h$ & ${\eta}_\tau$ & ${\eta}_{\tau h}$ &
$\mathcal{I}_{\textnormal{eff}}$
\\
\hline
1  & 4  & 4   & 36    & 7.1588e-02 & ---  & \phantom{-}5.6668e-02 & \phantom{-}1.9391e-02 & \phantom{-}7.6060e-02 & 1.062\\
2  & 5  & 16  & 125   & 3.6641e-02 & 0.97 & \phantom{-}7.4575e-03 & \phantom{-}2.8043e-02 & \phantom{-}3.5501e-02 & 0.969\\
3  & 6  & 28  & 246   & 2.3021e-02 & 0.67 & \phantom{-}5.4712e-03 & \phantom{-}3.2327e-02 & \phantom{-}3.7798e-02 & 1.642\\
4  & 8  & 76  & 682   & 2.0517e-02 & 0.17 & \phantom{-}2.1734e-03 & \phantom{-}2.0763e-02 & \phantom{-}2.2936e-02 & 1.118\\
5  & 11 & 88  & 1181  & 1.6786e-02 & 0.29 & \phantom{-}2.3597e-03 & \phantom{-}2.1533e-02 & \phantom{-}2.3892e-02 & 1.423\\
6  & 15 & 124 & 1959  & 1.5438e-02 & 0.12 & \phantom{-}1.4977e-03 & \phantom{-}1.4946e-02 & \phantom{-}1.6443e-02 & 1.065\\
7  & 21 & 136 & 3085  & 9.3565e-03 & 0.72 & \phantom{-}2.3671e-03 & \phantom{-}1.5271e-02 & \phantom{-}1.7638e-02 & 1.885\\
8  & 30 & 160 & 4878  & 5.9652e-03 & 0.65 & \phantom{-}2.5744e-03 & \phantom{-}1.2993e-02 & \phantom{-}1.5567e-02 & 2.610\\
9  & 42 & 172 & 7610  & 3.9673e-03 & 0.59 & \phantom{-}2.1052e-03 & \phantom{-}7.8182e-03 & \phantom{-}9.9234e-03 & 2.501\\
10 & 58 & 208 & {11786} & {3.2686e-03} & 0.28 & \phantom{-}2.4333e-03 & \phantom{-}6.3257e-03 & \phantom{-}8.7590e-03 & 2.680\\
\hline
   &    &     &       & avg. EOC = & 0.50 &\\
\hline
\end{tabular}
}

\vskip-1ex
\caption{$L^2(L^2)$-error reduction in the convection-diffusion transport case
with $\varepsilon = 1$ without stabilization $\delta_0=0$ for Sec.~\ref{sec:5:1}.
$e^{1,0,2,1}$ corresponds to
the adaptive solution approximation
$\concentration_{\tau h}^{1,0}$ in cG(1)-dG(0)
and dual solution approximation
$\dualz_{\tau h}^{2,1}$ in cG(2)-dG(1).}

\label{tab:3:cone:ada}
\end{table}

\begin{table}[h!]
\centering

\resizebox{0.90\linewidth}{!}{%
\begin{tabular}{c | rrr | cc | ccc | c}
\hline
\hline
$\ell$ & $N$ & $N_K^{\text{max}}$ & $N_{\text{DoF}}^{\text{tot}}$ &
$\|e^{1,1,2,2}\|$ & EOC &
${\eta}_h$ & ${\eta}_\tau$ & ${\eta}_{\tau h}$ &
$\mathcal{I}_{\textnormal{eff}}$
\\
\hline
1  & 4  & 4   & 72    & 8.7470e-02 & ---  & \phantom{-}4.7959e-02 &           -4.5912e-03 & \phantom{-}4.3367e-02 & 0.496\\
2  & 5  & 16  & 250   & 3.1113e-02 & 1.49 & \phantom{-}4.1953e-03 & \phantom{-}7.3543e-03 & \phantom{-}1.1550e-02 & 0.371\\
3  & 6  & 28  & 492   & 2.1156e-02 & 0.56 & \phantom{-}3.0706e-03 & \phantom{-}9.2354e-03 & \phantom{-}1.2306e-02 & 0.582\\
4  & 8  & 76  & 1468  & 1.4856e-02 & 0.51 & \phantom{-}2.1015e-03 & \phantom{-}6.4001e-03 & \phantom{-}8.5016e-03 & 0.572\\
5  & 11 & 124 & 2666  & 8.0165e-03 & 0.89 & \phantom{-}2.6080e-03 & \phantom{-}6.5137e-03 & \phantom{-}9.1217e-03 & 1.138\\
6  & 14 & 160 & 4456  & 3.9318e-03 & 1.03 & \phantom{-}3.1048e-03 & \phantom{-}7.1463e-03 & \phantom{-}1.0251e-02 & 2.607\\
7  & 18 & 196 & {6960}  & {3.1346e-03} & 0.33 & \phantom{-}2.2002e-03 & \phantom{-}4.3348e-03 & \phantom{-}6.5350e-03 & 2.085\\
8  & 23 & 232 & 10610 & 1.8780e-03 & 0.74 & \phantom{-}5.4403e-04 & \phantom{-}1.7686e-03 & \phantom{-}2.3126e-03 & 1.231\\
9  & 31 & 280 & 16522 & 1.4014e-03 & 0.42 & \phantom{-}3.5661e-04 & \phantom{-}9.3709e-04 & \phantom{-}1.2937e-03 & 0.923\\
10 & 42 & 340 & 26252 & 1.0214e-03 & 0.46 & \phantom{-}5.1205e-05 & \phantom{-}2.8010e-04 & \phantom{-}3.3130e-04 & 0.324\\
\hline
   &    &     &       & avg. EOC = & 0.71 &\\
\hline
\end{tabular}
}

\vskip-1ex
\caption{$L^2(L^2)$-error reduction in the convection-diffusion transport case
with $\varepsilon = 1$ without stabilization $\delta_0=0$ for Sec.~\ref{sec:5:1}.
$e^{1,1,2,2}$ corresponds to
the adaptive solution approximation
$\concentration_{\tau h}^{1,1}$ in cG(1)-dG(1)
and dual solution approximation
$\dualz_{\tau h}^{2,2}$ in cG(2)-dG(2).}
\label{tab:4:cone:ada}
\end{table}

\begin{table}[h!]
\centering

\resizebox{0.90\linewidth}{!}{%
\begin{tabular}{c | rrr | cc | ccc | c}
\hline
\hline
$\ell$ & $N$ & $N_K^{\text{max}}$ & $N_{\text{DoF}}^{\text{tot}}$ &
$\|e^{2,2,3,3}\|$ & EOC &
${\eta}_h$ & ${\eta}_\tau$ & ${\eta}_{\tau h}$ &
$\mathcal{I}_{\textnormal{eff}}$
\\
\hline
1  & 4  & 4   & 300    & 5.7763e-02 & ---  & \phantom{-}4.1837e-03 & \phantom{-}3.6501e-03 & \phantom{-}7.8339e-03 & 0.136\\
2  & 5  & 16  & 1215   & 1.4581e-02 & 1.99 & \phantom{-}3.5634e-03 & \phantom{-}4.5982e-03 & \phantom{-}8.1616e-03 & 0.560\\
3  & 6  & 28  & 2538   & 1.0721e-02 & 0.44 & \phantom{-}5.9011e-04 & \phantom{-}4.9905e-03 & \phantom{-}5.5806e-03 & 0.521\\
4  & 8  & 76  & {5112}   & {5.4041e-03} & 0.99 & \phantom{-}3.0895e-04 & \phantom{-}3.3977e-03 & \phantom{-}3.7067e-03 & 0.686\\
5  & 11 & 88  & {11325}  & {2.1878e-03} & 1.30 & \phantom{-}6.4201e-04 & \phantom{-}2.7536e-03 & \phantom{-}3.3956e-03 & 1.552\\
6  & 15 & 100 & 18489  & 9.1402e-04 & 1.26 & \phantom{-}4.6743e-04 & \phantom{-}1.8813e-03 & \phantom{-}2.3487e-03 & 2.570\\
7  & 21 & 124 & 29895  & 4.4147e-04 & 1.05 & \phantom{-}3.3121e-04 & \phantom{-}1.9212e-04 & \phantom{-}5.2333e-04 & 1.185\\
8  & 24 & 184 & 43026  & 2.6759e-04 & 0.72 & \phantom{-}1.5650e-04 & \phantom{-}1.7675e-04 & \phantom{-}3.3325e-04 & 1.245\\
9  & 30 & 208 & 66270  & 1.6078e-04 & 0.73 & \phantom{-}7.4964e-05 & \phantom{-}1.5182e-04 & \phantom{-}2.2678e-04 & 1.410\\
10 & 40 & 256 & 101178 & 1.1594e-04 & 0.47 & \phantom{-}6.7057e-05 & \phantom{-}1.9555e-05 & \phantom{-}8.6611e-05 & 0.747\\
\hline
   &    &     &       & avg. EOC = & 0.99 &\\
\hline
\end{tabular}
}

\vskip-1ex
\caption{$L^2(L^2)$-error reduction in the convection-diffusion transport case
with $\varepsilon = 1$ without stabilization $\delta_0=0$ for Sec.~\ref{sec:5:1}.
$e^{2,2,3,3}$ corresponds to
the adaptive solution approximation
$\concentration_{\tau h}^{2,2}$ in cG(2)-dG(2)
and dual solution approximation
$\dualz_{\tau h}^{3,3}$ in cG(3)-dG(3).}
\label{tab:5:cone:ada}
\end{table}

\begin{table}[h!]
\centering

\resizebox{0.90\linewidth}{!}{%
\begin{tabular}{c | rrr | cc | ccc | c}
\hline
\hline
$\ell$ & $N$ & $N_K^{\text{max}}$ & $N_{\text{DoF}}^{\text{tot}}$ &
$\|e^{1,0,2,1}\|$ & EOC &
${\eta}_h$ & ${\eta}_\tau$ & ${\eta}_{\tau h}$ &
$\mathcal{I}_{\textnormal{eff}}$
\\
\hline
1  & 4  & 4   & 36    & 9.1527e-02 & ---  & \phantom{-}3.6697e-02 & \phantom{-}1.5147e-02 & \phantom{-}5.1844e-02 & 0.566\\
2  & 5  & 16  & 125   & 6.6916e-02 & 0.45 & \phantom{-}6.4515e-03 &           -6.0024e-04 & \phantom{-}5.8513e-03 & 0.087\\
3  & 6  & 52  & 330   & 6.2426e-02 & 0.10 & \phantom{-}2.6882e-03 & \phantom{-}5.9916e-03 & \phantom{-}8.6798e-03 & 0.139\\
4  & 8  & 88  & 804   & 5.5849e-02 & 0.16 & \phantom{-}2.1670e-03 & \phantom{-}6.6884e-03 & \phantom{-}8.8554e-03 & 0.159\\
5  & 11 & 100 & 1265  & 4.3011e-02 & 0.38 & \phantom{-}3.2302e-03 & \phantom{-}1.1793e-02 & \phantom{-}1.5023e-02 & 0.349\\
6  & 15 & 124 & 2037  & 3.0684e-02 & 0.49 & \phantom{-}2.3297e-03 & \phantom{-}1.3276e-02 & \phantom{-}1.5605e-02 & 0.509\\
7  & 21 & 160 & 3283  & 2.1294e-02 & 0.53 & \phantom{-}1.8278e-03 & \phantom{-}1.1939e-02 & \phantom{-}1.3767e-02 & 0.647\\
8  & 30 & 184 & 5262  & 1.5710e-02 & 0.44 & \phantom{-}2.7619e-03 & \phantom{-}1.0873e-02 & \phantom{-}1.3635e-02 & 0.868\\
9  & 41 & 220 & 8357  & 1.2567e-02 & 0.32 & \phantom{-}3.3126e-03 & \phantom{-}7.9147e-03 & \phantom{-}1.1227e-02 & 0.893\\
10 & 55 & 280 & {13121} & {9.9855e-03} & 0.33 & \phantom{-}5.5998e-03 & \phantom{-}6.5796e-03 & \phantom{-}1.2179e-02 & 1.220\\
\hline
   &    &     &       & avg. EOC = & 0.36 &\\
\hline
\end{tabular}
}

\vskip-1ex
\caption{$L^2(L^2)$-error reduction in the stabilized convection-dominated
transport case with $\varepsilon=10^{-6}$ and $\delta_0=10^{-1}$ for Sec.~\ref{sec:5:1}.
$e^{1,0,2,1}$ corresponds to
the adaptive solution approximation
$\concentration_{\tau h}^{1,0}$ in cG(1)-dG(0)
and dual solution approximation
$\dualz_{\tau h}^{2,1}$ in cG(2)-dG(1).}

\label{tab:6:cone:ada}
\end{table}

\begin{table}[h!]
\centering

\resizebox{0.90\linewidth}{!}{%
\begin{tabular}{c | rrr | cc | ccc | c}
\hline
\hline
$\ell$ & $N$ & $N_K^{\text{max}}$ & $N_{\text{DoF}}^{\text{tot}}$ &
$\|e^{1,1,2,2}\|$ & EOC &
${\eta}_h$ & ${\eta}_\tau$ & ${\eta}_{\tau h}$ &
$\mathcal{I}_{\textnormal{eff}}$
\\
\hline
1  & 4  & 4   & 72    & 5.8984e-02 & ---  & \phantom{-}2.6394e-02 & \phantom{-}1.3546e-03 & \phantom{-}2.7749e-02 & 0.470\\
2  & 5  & 16  & 250   & 5.3935e-02 & 0.13 &           -1.4730e-02 & \phantom{-}3.8407e-03 &           -1.0889e-02 & 0.202\\
3  & 6  & 52  & 660   & 4.2387e-02 & 0.35 &           -5.4370e-04 & \phantom{-}4.7724e-03 & \phantom{-}4.2287e-03 & 0.100\\
4  & 9  & 76  & 1634  & 2.4328e-02 & 0.80 & \phantom{-}1.2300e-03 & \phantom{-}4.5411e-03 & \phantom{-}5.7711e-03 & 0.237\\
5  & 12 & 124 & 2668  & 1.5031e-02 & 0.69 & \phantom{-}7.6984e-03 & \phantom{-}4.4589e-03 & \phantom{-}1.2157e-02 & 0.809\\
6  & 14 & 172 & {4316}  & {8.9709e-03} & 0.74 & \phantom{-}2.7195e-03 & \phantom{-}5.1674e-03 & \phantom{-}7.8869e-03 & 0.879\\
7  & 18 & 208 & 7032  & 5.9470e-03 & 0.59 & \phantom{-}3.2797e-03 & \phantom{-}3.1406e-03 & \phantom{-}6.4202e-03 & 1.080\\
8  & 22 & 268 & 10668 & 4.2404e-03 & 0.49 &           -1.1081e-04 & \phantom{-}1.7962e-03 & \phantom{-}1.6854e-03 & 0.397\\
9  & 33 & 292 & 17270 & 3.2138e-03 & 0.40 &           -1.8817e-03 & \phantom{-}7.9462e-04 &           -1.0871e-03 & 0.338\\
10 & 45 & 568 & 41550 & 2.9555e-03 & 0.12 &           -8.4016e-03 & \phantom{-}3.2177e-04 &           -8.0798e-03 & 2.734\\
\hline
   &    &     &       & avg. EOC = & 0.48 &\\
\hline
\end{tabular}
}

\vskip-1ex
\caption{$L^2(L^2)$-error reduction in the stabilized convection-dominated
transport case with $\varepsilon=10^{-6}$ and $\delta_0=10^{-1}$ for Sec.~\ref{sec:5:1}.
$e^{1,1,2,2}$ corresponds to
the adaptive solution approximation
$\concentration_{\tau h}^{1,1}$ in cG(1)-dG(1)
and dual solution approximation
$\dualz_{\tau h}^{2,2}$ in cG(2)-dG(2).}
\label{tab:7:cone:ada}
\end{table}

\begin{table}[h!]
\centering

\resizebox{0.90\linewidth}{!}{%
\begin{tabular}{c | rrr | cc | ccc | c}
\hline
\hline
$\ell$ & $N$ & $N_K^{\text{max}}$ & $N_{\text{DoF}}^{\text{tot}}$ &
$\|e^{2,2,3,3}\|$ & EOC &
${\eta}_h$ & ${\eta}_\tau$ & ${\eta}_{\tau h}$ &
$\mathcal{I}_{\textnormal{eff}}$
\\
\hline
1  & 4  & 4   & 300    & 4.3801e-02 & ---  &           -2.6854e-03 & \phantom{-}1.0697e-03 &           -1.6156e-03 & 0.037\\
2  & 5  & 16  & 1215   & 2.8964e-02 & 0.60 & \phantom{-}2.2071e-03 & \phantom{-}2.3582e-03 & \phantom{-}4.5653e-03 & 0.158\\
3  & 6  & 28  & 2538   & 2.6314e-02 & 0.14 & \phantom{-}1.0415e-03 & \phantom{-}2.0664e-03 & \phantom{-}3.1078e-03 & 0.118\\
4  & 7  & 76  & 4719   & 1.6058e-02 & 0.71 & \phantom{-}6.6330e-04 & \phantom{-}2.2853e-03 & \phantom{-}2.9486e-03 & 0.184\\
5  & 9  & 88  & {8175}   & {9.9382e-03} & 0.69 & \phantom{-}4.8613e-04 & \phantom{-}1.9507e-03 & \phantom{-}2.4369e-03 & 0.245\\
6  & 12 & 124 & 14442  & 7.1355e-03 & 0.48 &           -5.9480e-04 & \phantom{-}1.0807e-03 & \phantom{-}4.8594e-04 & 0.068\\
7  & 18 & 244 & 31548  & 5.2461e-03 & 0.44 &           -1.7721e-03 & \phantom{-}6.0417e-04 &           -1.1680e-03 & 0.223\\
8  & 22 & 376 & 66042  & 4.4188e-03 & 0.25 &           -1.1307e-03 & \phantom{-}1.0326e-04 &           -1.0275e-03 & 0.233\\
9  & 23 & 556 & 105825 & 3.9399e-03 & 0.17 &           -1.3787e-03 &           -8.2448e-05 &           -1.4611e-03 & 0.371\\
10 & 24 & 868 & 164886 & 3.2102e-03 & 0.30 &           -1.0300e-03 &           -1.4457e-04 &           -1.1746e-03 & 0.366\\
\hline
   &    &     &        & avg. EOC = & 0.42 &\\
\hline
\end{tabular}
}

\vskip-1ex
\caption{$L^2(L^2)$-error reduction in the stabilized convection-dominated
transport case with $\varepsilon=10^{-6}$ and $\delta_0=10^{-1}$ for Sec.~\ref{sec:5:1}.
$e^{2,2,3,3}$ corresponds to
the adaptive solution approximation
$\concentration_{\tau h}^{2,2}$ in cG(2)-dG(2)
and dual solution approximation
$\dualz_{\tau h}^{3,3}$ in cG(3)-dG(3).}
\label{tab:8:cone:ada}
\end{table}

\clearpage
\subsection{Example 2 (Transport in a channel)}
\label{sec:5:2}

\begin{figure}
\centering

\includegraphics[width=.44\linewidth]{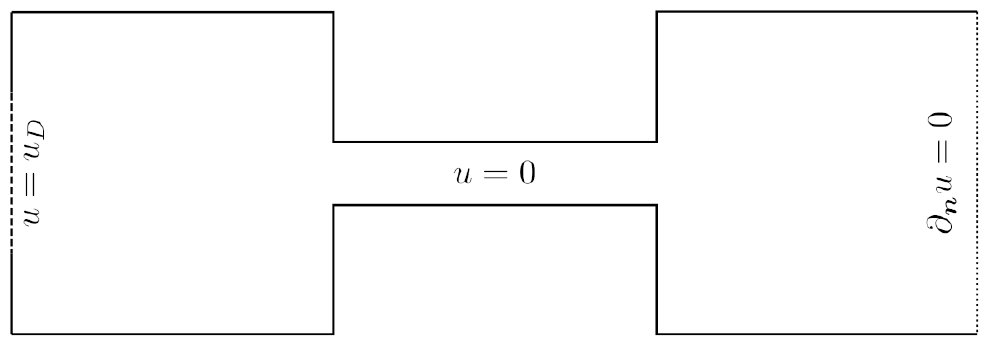}
~~
\includegraphics[width=.44\linewidth]{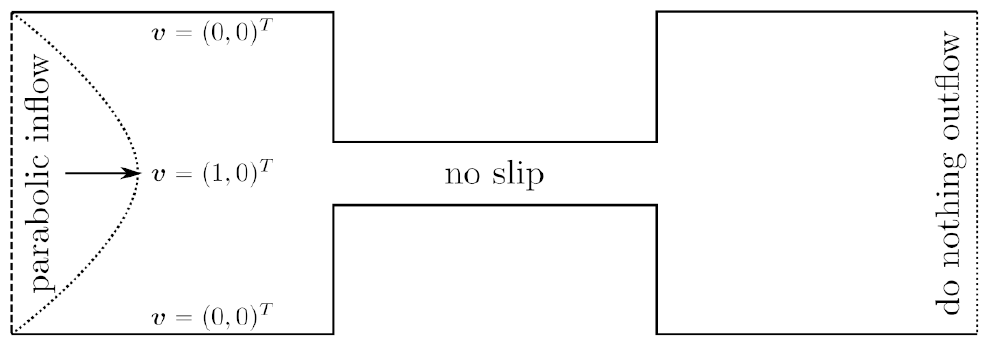}

\caption{Boundary colorization for the convection-diffusion problem (left)
and the coupled Stokes problem (right) for Sec.~\ref{sec:5:2}.}

\label{fig:4:boundary}
\end{figure}

\begin{figure}
\centering

\includegraphics[width=.6\linewidth]{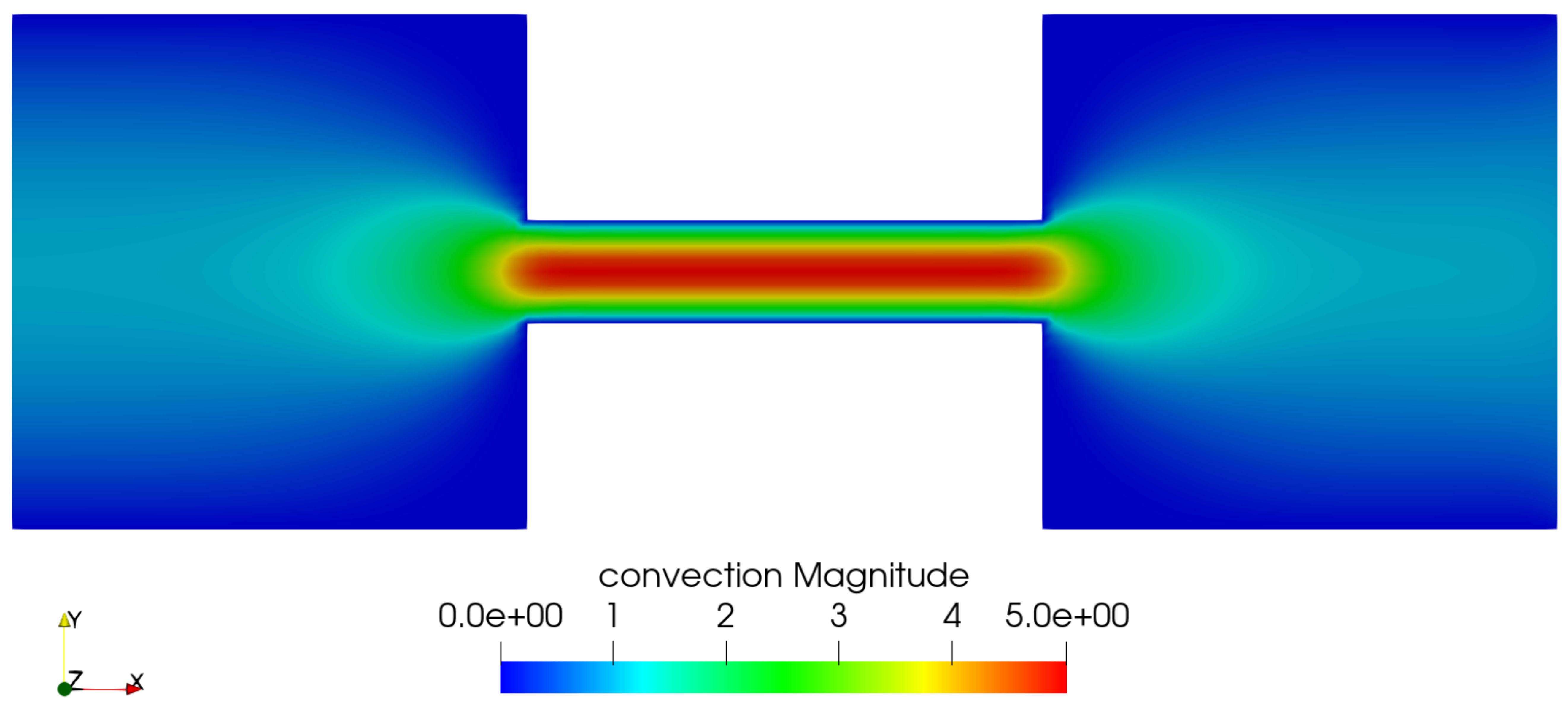}

\caption{Convection $\convection_h$ solution of the Stokes problem on a
sufficiently globally refined mesh with $Q_2$-$Q_1$ finite elements
for Sec.~\ref{sec:5:2}.
On the left boundary a parabolic inflow profile in the positive x-direction
with unit magnitude is prescribed for the convection $\convection$.}

\label{fig:5:convection}
\end{figure}

In this example we simulate a convection-dominated transport with goal-oriented
adaptivity of a species through a channel with a constraint. The domain and its
boundary colorization are presented by Fig. \ref{fig:4:boundary}. Precisely,
the spatial domain is composed of two unit squares and a constraint in the middle
which restricts the channel height by a factor of 5.
Precisely, $\Omega = (-1,0)\times(-0.5,0.5) \cup (0,1)\times(-0.1,0.1) \cup
(1,2)\times(-0.5,0.5)$ with an initial cell diameter of $h=\sqrt{2 \cdot 0.025^2}$.
The time domain is set to $I=(0,2.5)$ with an initial $\tau=0.1$ for
the initialization of the slabs for the first loop $\ell=1$.
We approximate the primal solution $\concentration_{\tau h}^{1,1}$ with the
cG(1)-dG(1) method and the dual solution $\dualz_{\tau h}^{2,2}$ with the
cG(2)-dG(2) method.
The target quantity is
\begin{displaymath}
J(\varphi)= \frac{1}{\| \concentration_{\tau h} \|_{(0,T)\times\Omega}}
\displaystyle\int_I (\varphi, \concentration_{\tau h})\, \mathrm{d}t\,.
\end{displaymath}
The transport of the species, which enters the domain on the left with an
inhomogeneous and time-dependent Dirichlet boundary condition and
leaves the domain on the right through a homogeneous Neumann boundary condition,
is driven by the convection with magnitudes between 0 and 5 as displayed in
Fig.~\ref{fig:5:convection}. The diffusion coefficient has the constant and small
value of $\varepsilon=10^{-4}$, the reaction coefficient $\alpha=0$ is vanishing
and the density has the value $\rho=1$.
The local SUPG stabilization coefficient is here set to
$\delta_K = \delta_0 \cdot h_K$, $\delta_0=0$,
i.e. a vanishing stabilization here.
The initial value function $\concentration_0=0$ as well as the forcing term
$\transportforce=0$ are homogeneous.
The Dirichlet boundary function value is homogeneous on $\Gamma_D$ except for the
line $(-1,-1) \times (-0.25,0.25)$ where the value
\begin{displaymath}
\concentration(y,t)=16 \cdot (0.25-y) \cdot (y+0.25) \cdot
\min\{100t, 0.1\}
\end{displaymath}
is prescribed on the solution.
The viscosity is set to $\tilde{\viscosity}=1$.
The tuning parameters of the goal-oriented adaptive Algorithm given in
Sec.~\ref{sec:4:practical_aspects} are chosen here in a way to balance
automatically the potential misfit of the spatial and temporal errors as
$\theta_h^\textnormal{bottom} = 0$,
\begin{equation}
\label{eq:5:3:balancing}
\theta_h^\textnormal{top} = \frac{1}{2} \cdot
\min\left\{ \left| \frac{\eta_h}{\eta_h + \eta_\tau} \right|\,, 1\right\}\quad
\textnormal{and}\quad
\theta_\tau^\textnormal{top} = \frac{1}{2} \cdot
\min\left\{ \left| \frac{\eta_\tau}{\eta_h + \eta_\tau} \right|\,, 1\right\}\,.
\end{equation}

\begin{figure}
\centering

\includegraphics[width=.44\linewidth]{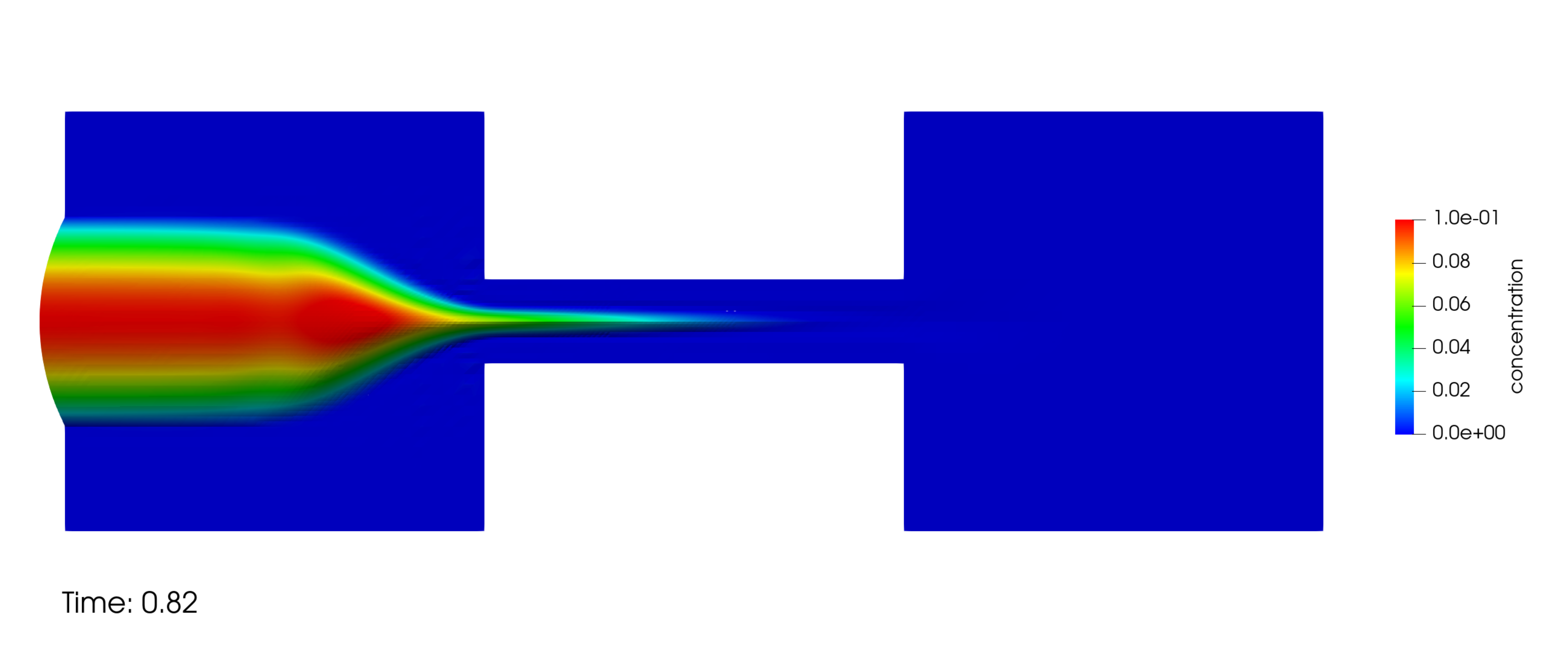}
~~
\includegraphics[width=.44\linewidth]{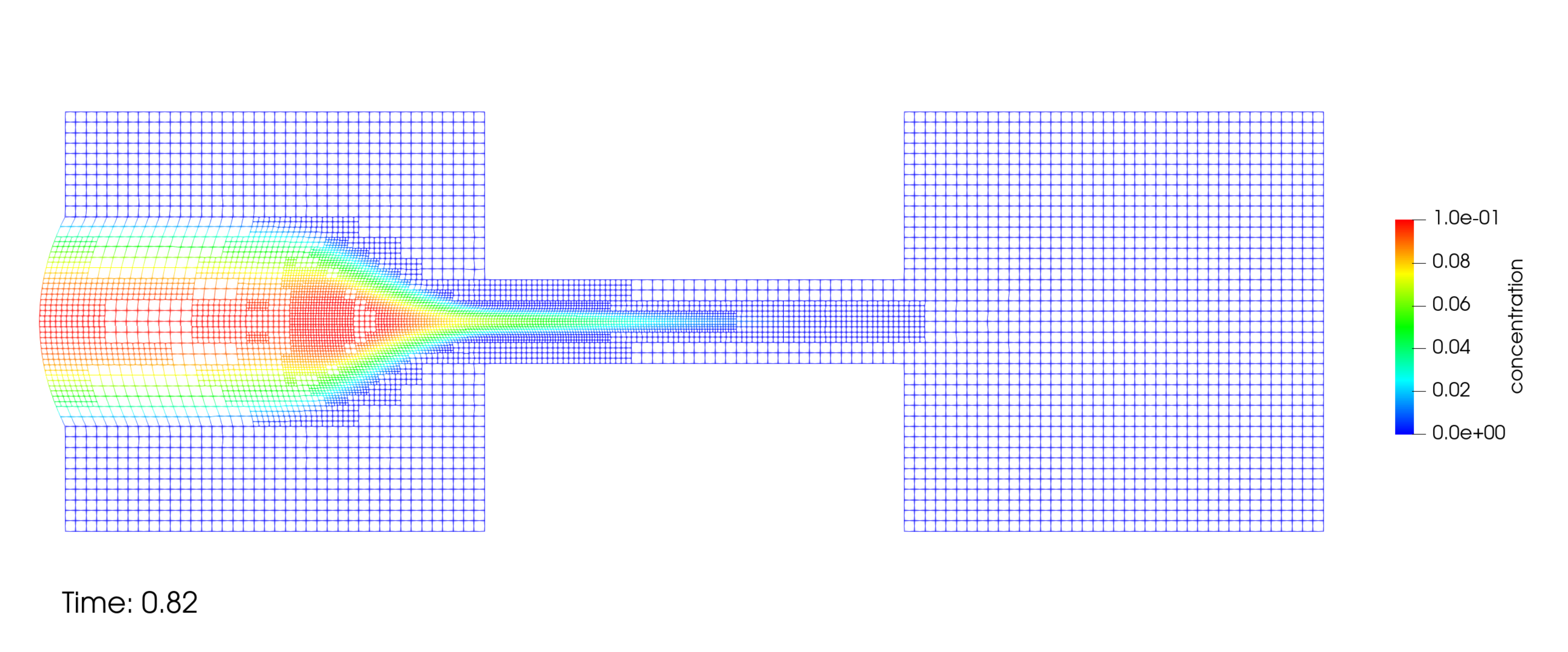}

\includegraphics[width=.44\linewidth]{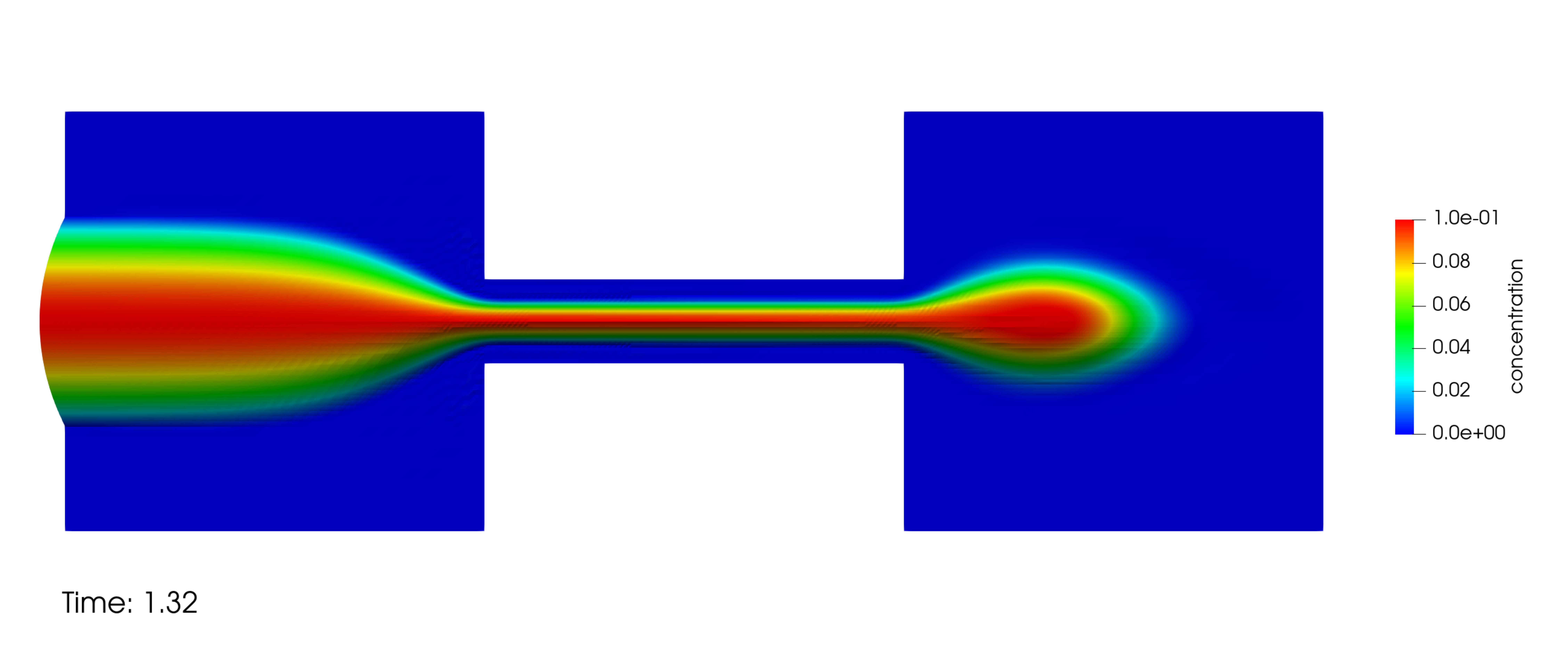}
~~
\includegraphics[width=.44\linewidth]{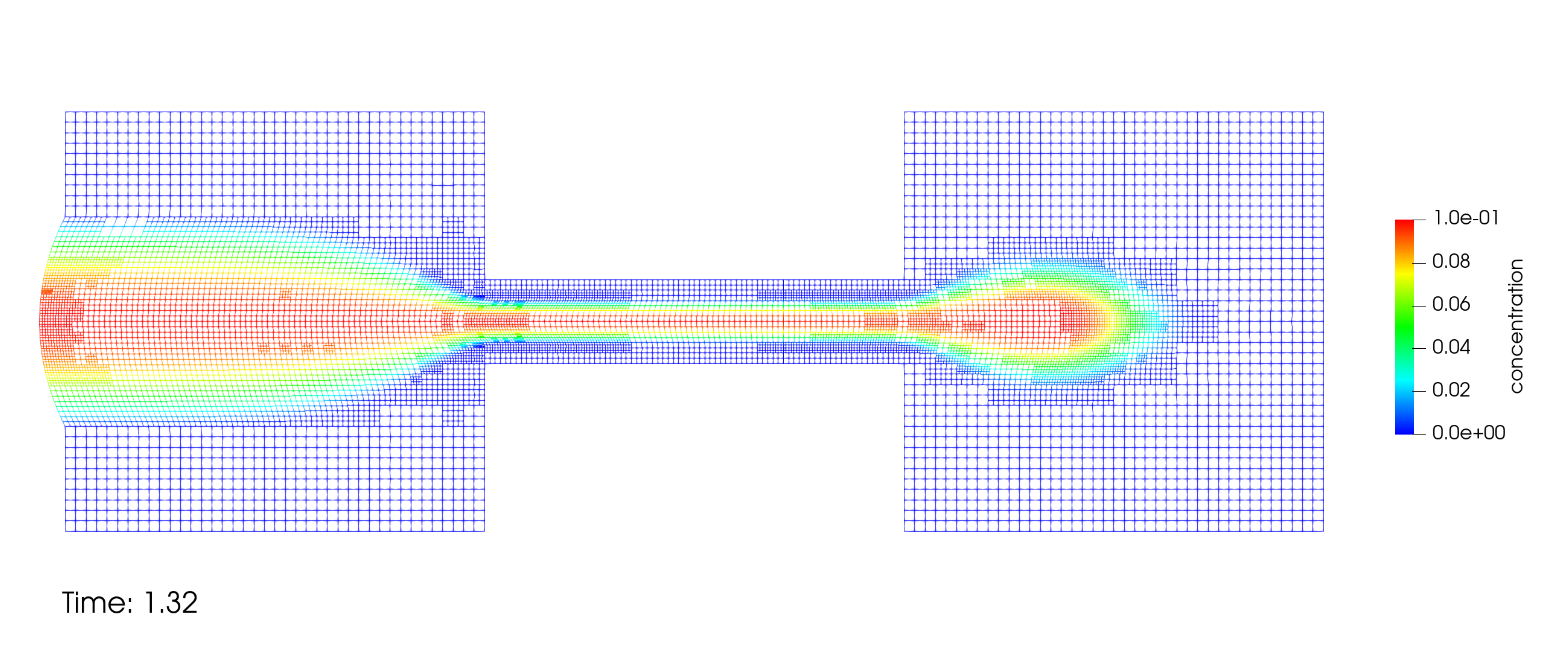}

\includegraphics[width=.44\linewidth]{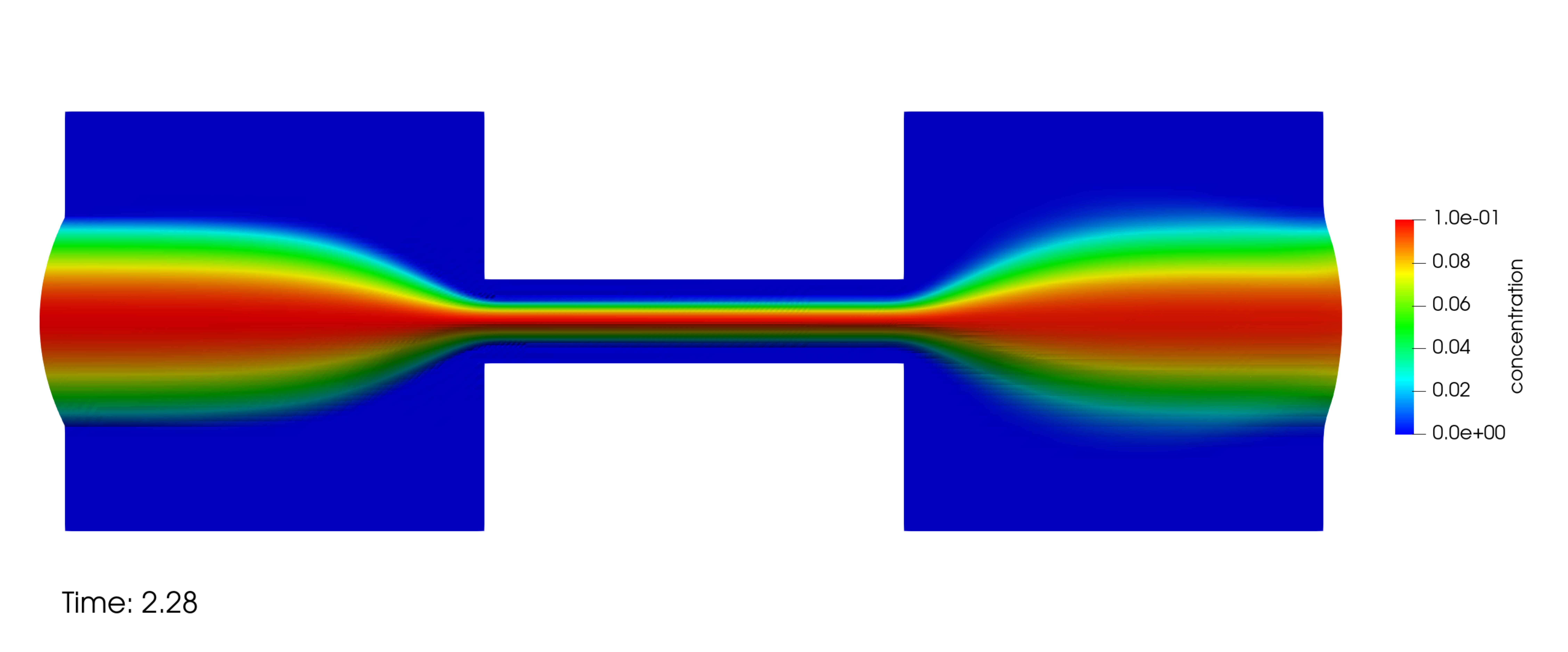}
~~
\includegraphics[width=.44\linewidth]{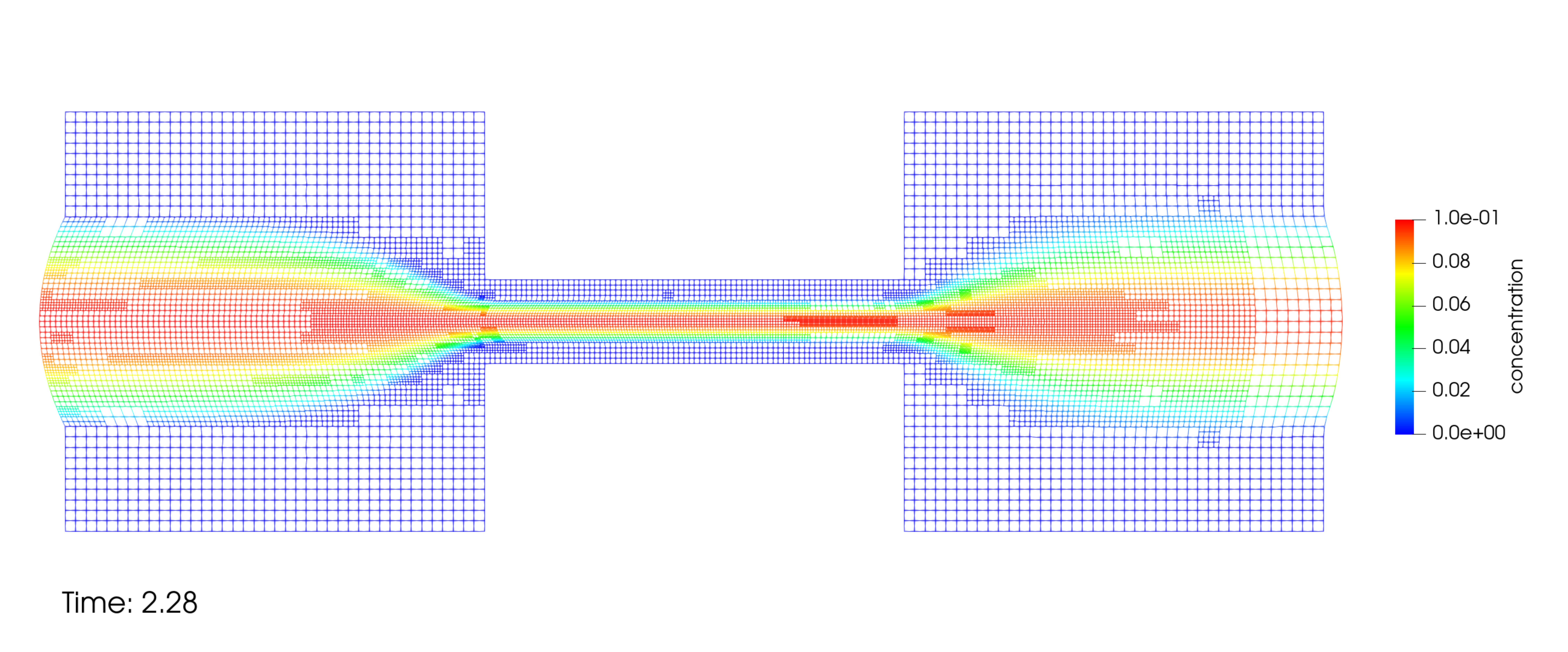}

\caption{Solution profiles and corresponding meshes of loop $\ell=5$
for Sec.~\ref{sec:5:2}.}
\label{fig:6:ex2:loop5}
\end{figure}

\begin{figure}
\centering

\includegraphics[width=.44\linewidth]{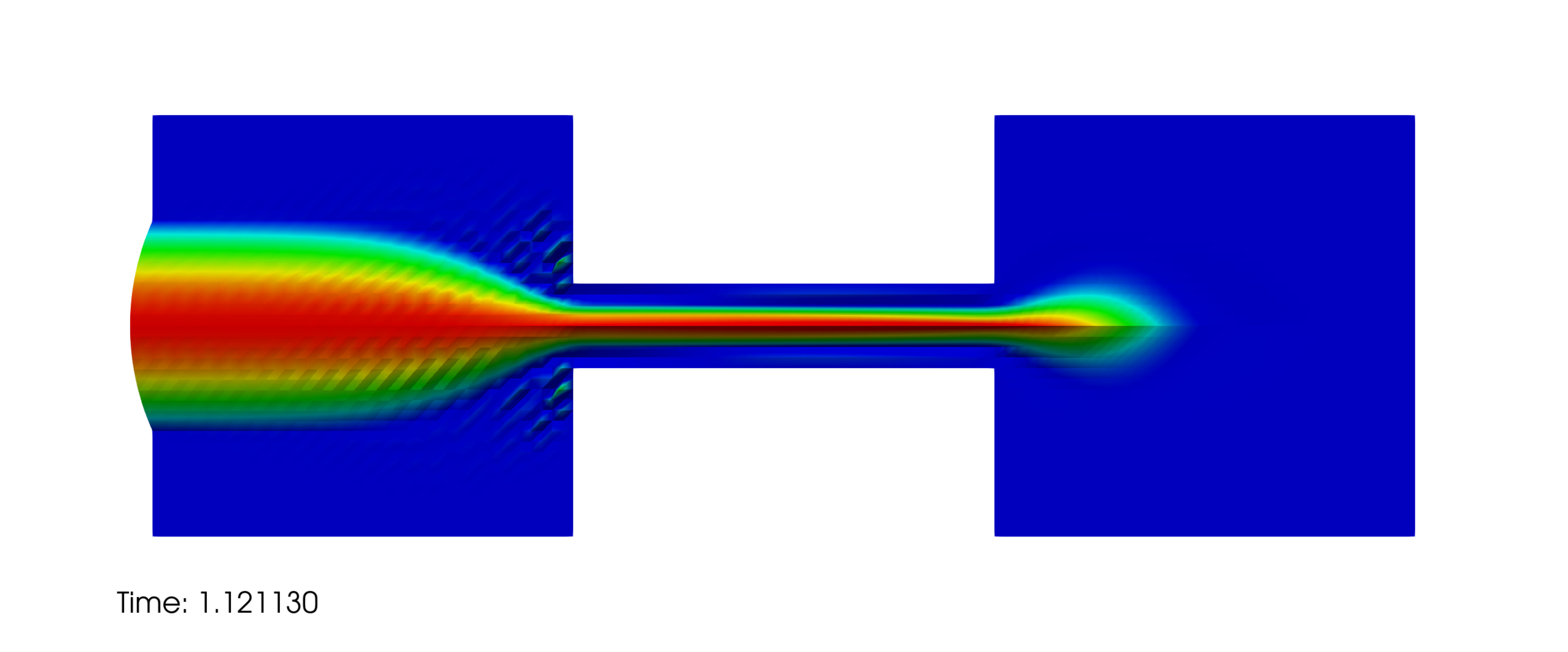}
~~
\includegraphics[width=.44\linewidth]{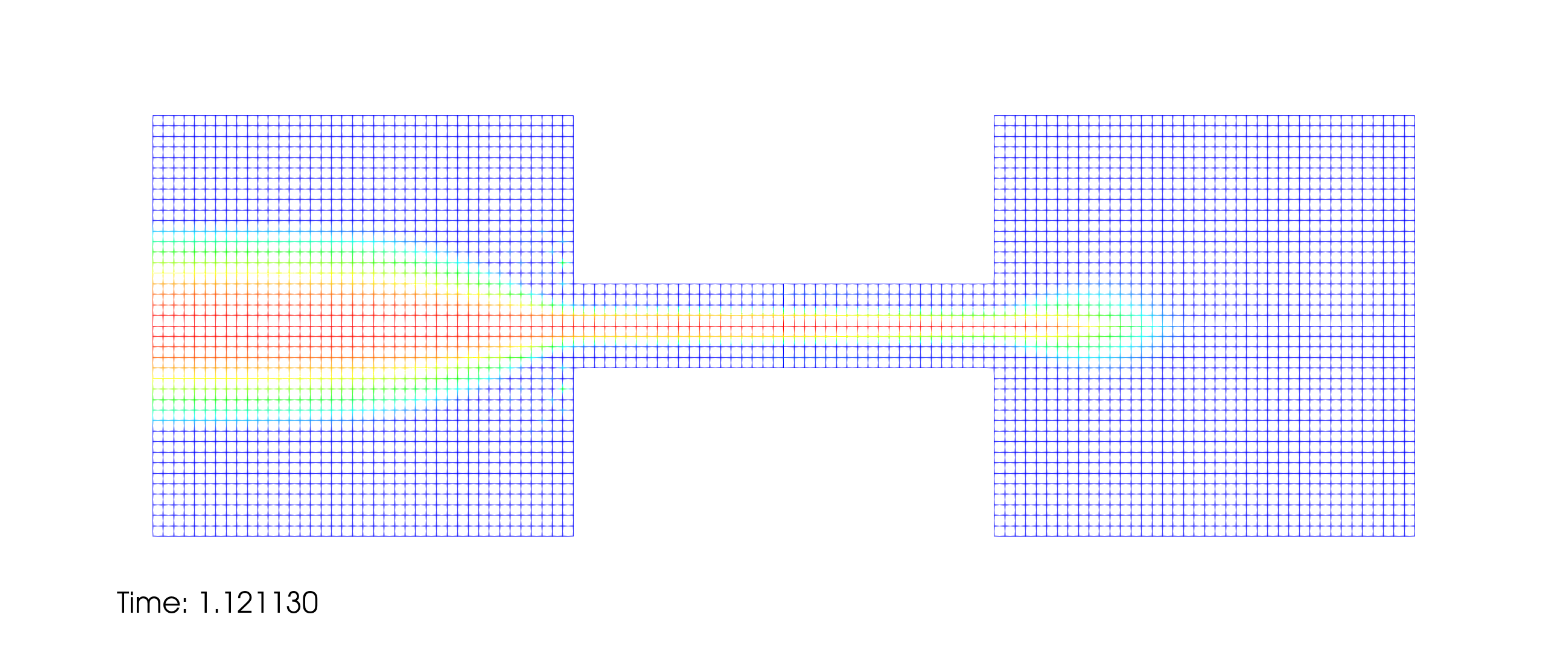}

\includegraphics[width=.44\linewidth]{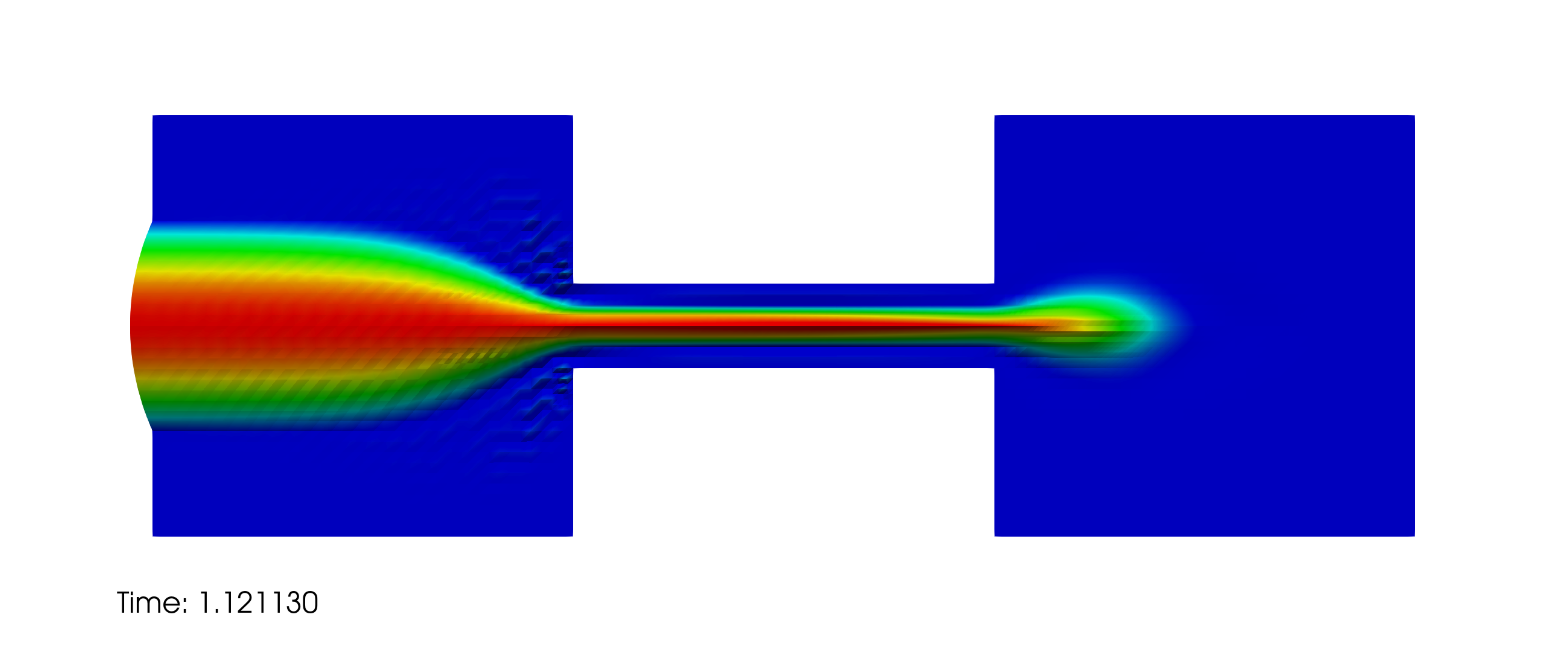}
~~
\includegraphics[width=.44\linewidth]{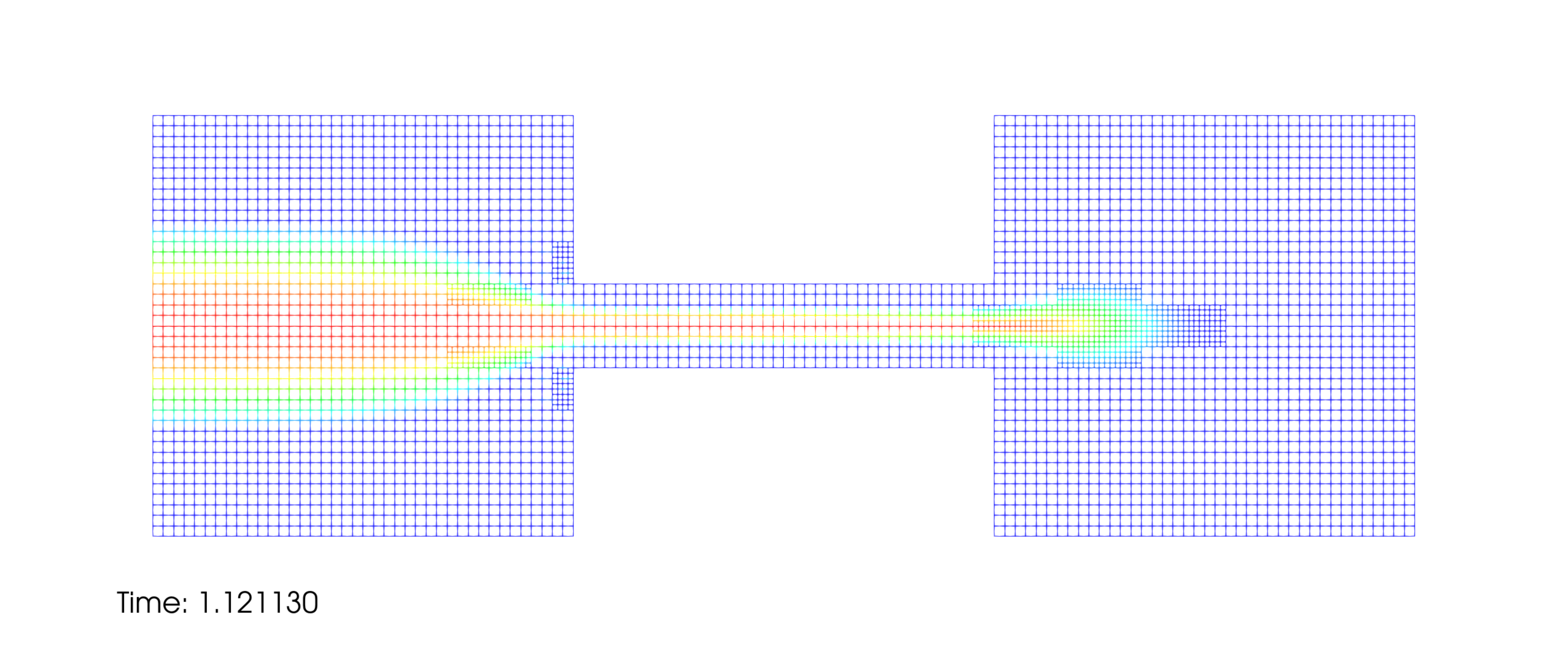}

\includegraphics[width=.44\linewidth]{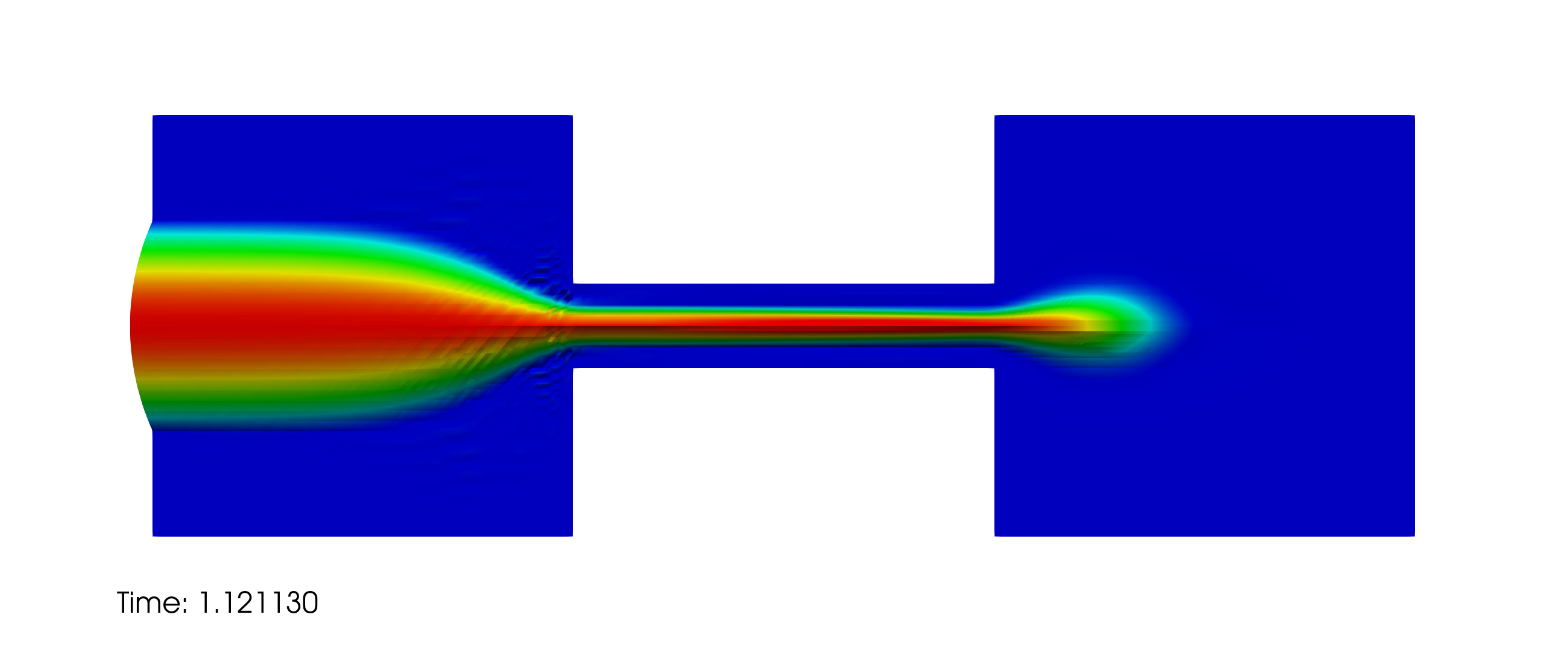}
~~
\includegraphics[width=.44\linewidth]{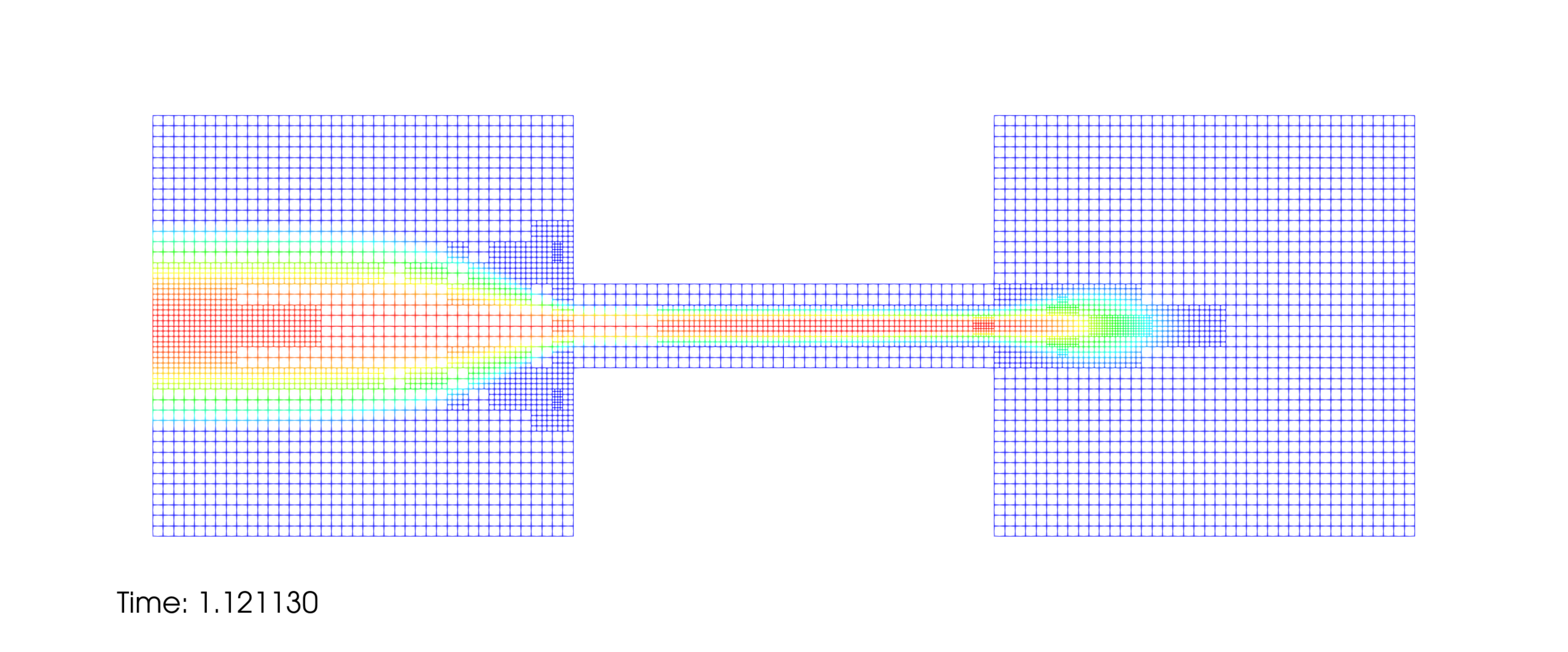}

\includegraphics[width=.44\linewidth]{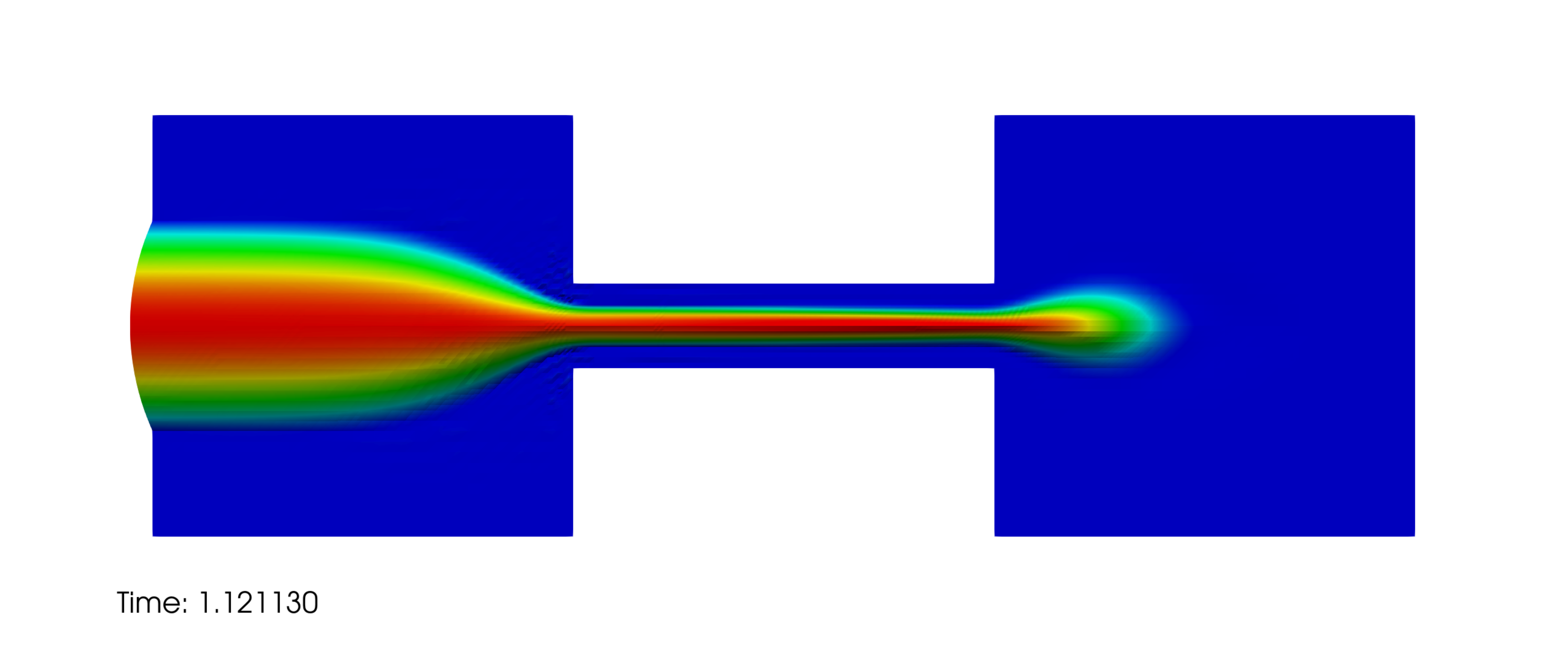}
~~
\includegraphics[width=.44\linewidth]{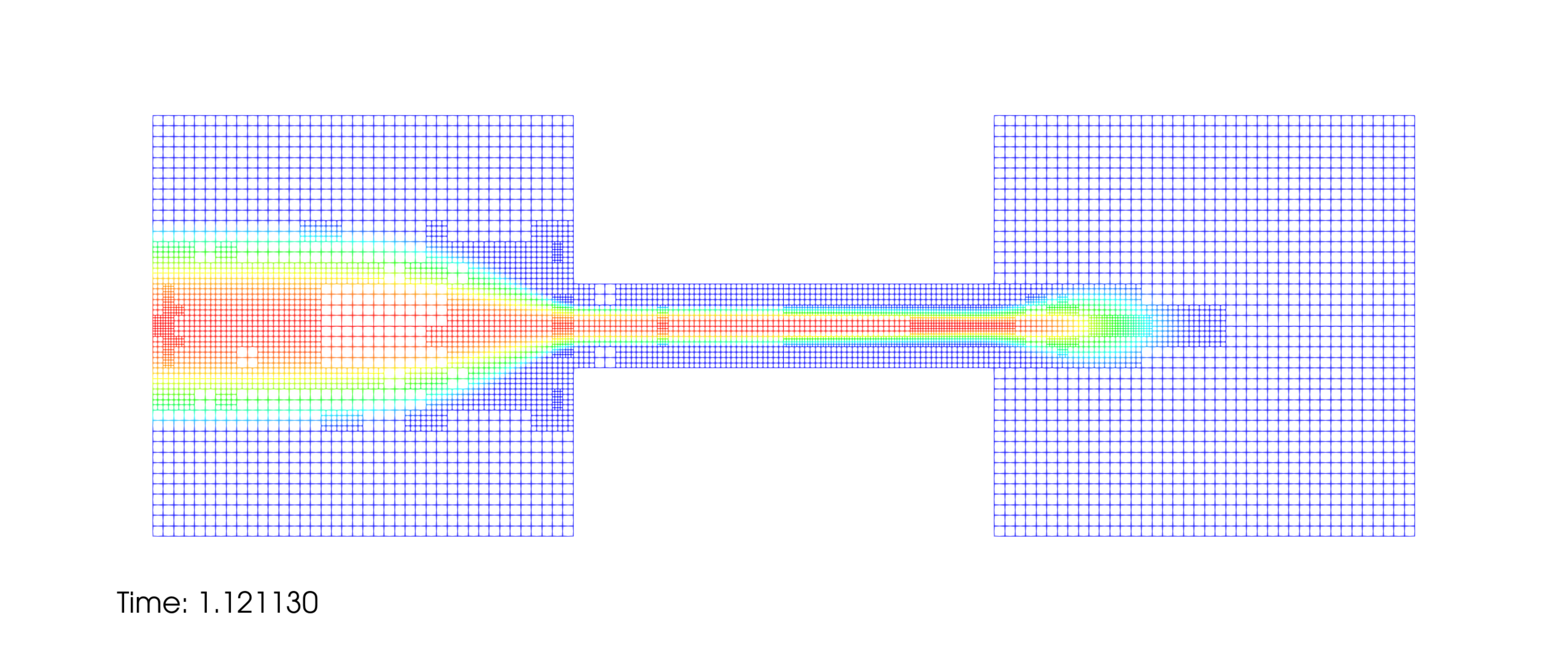}

\includegraphics[width=.44\linewidth]{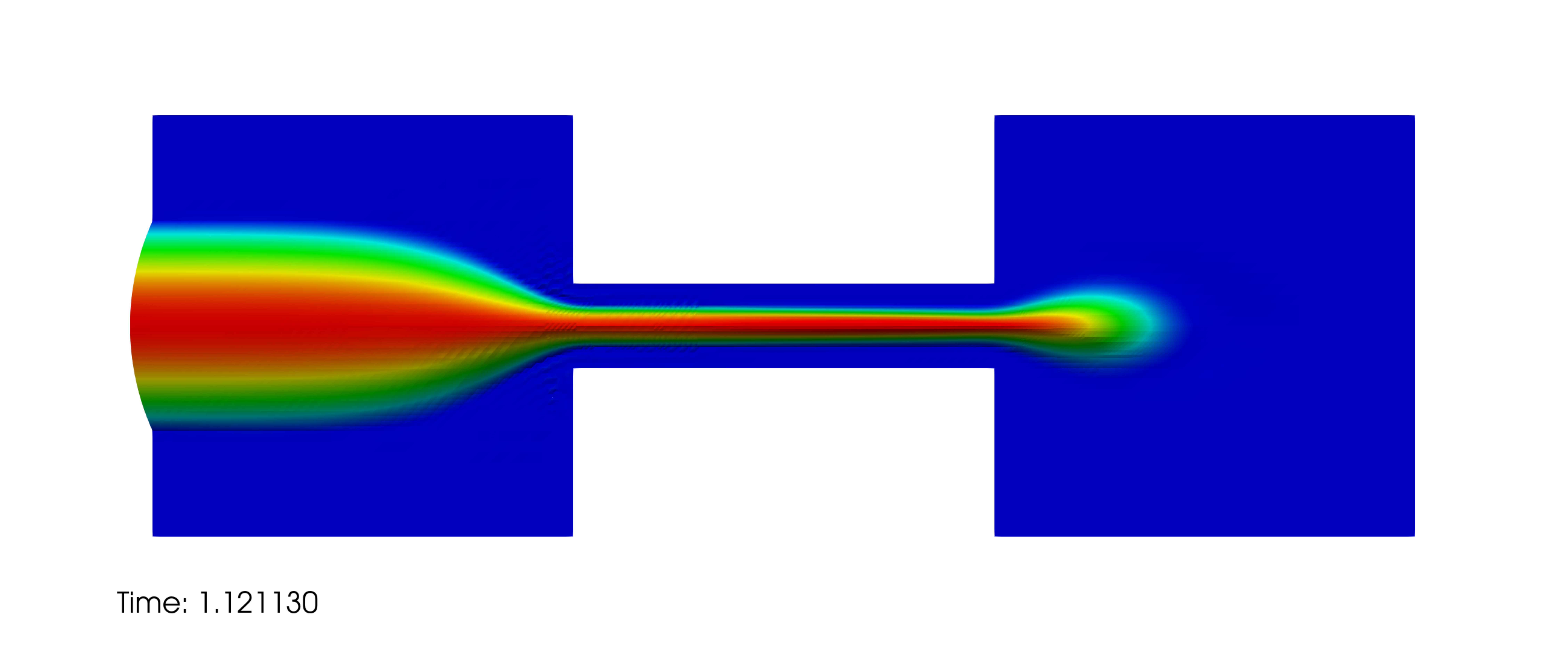}
~~
\includegraphics[width=.44\linewidth]{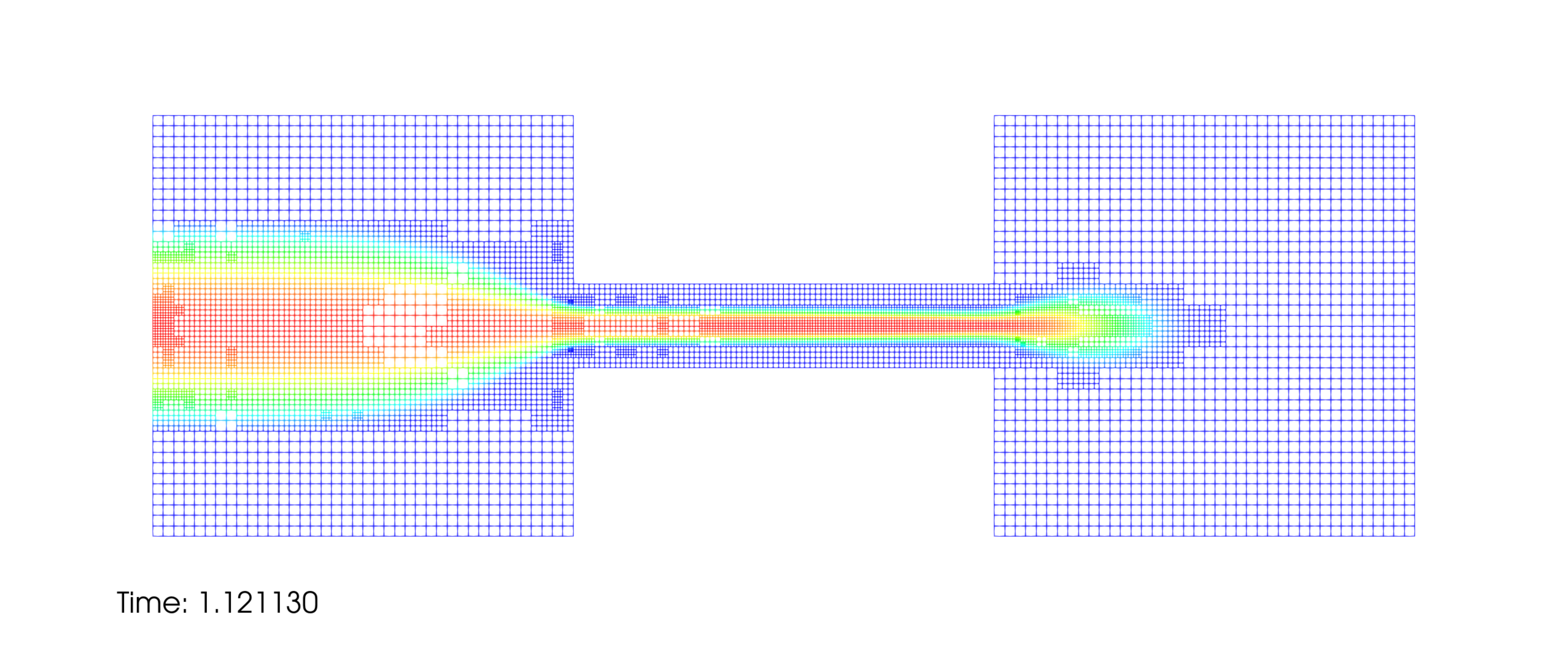}

\includegraphics[width=.44\linewidth]{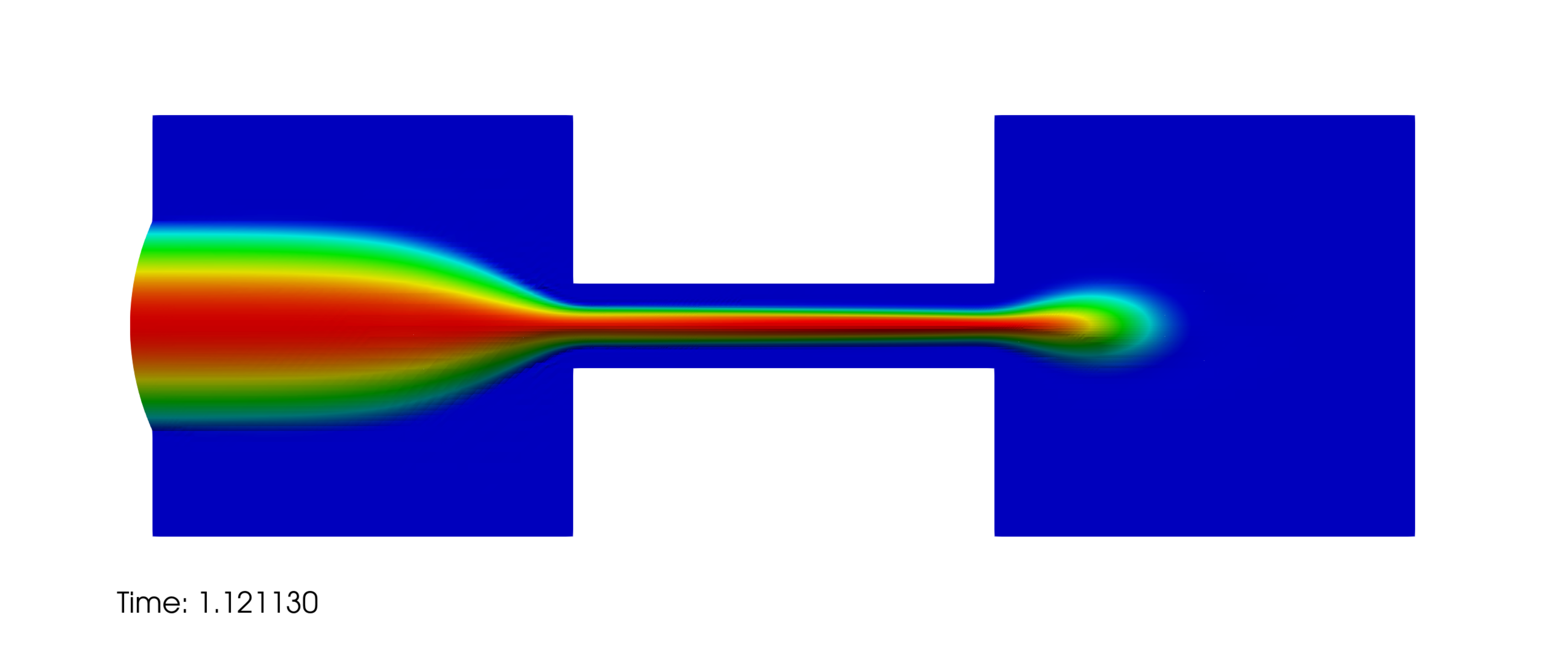}
~~
\includegraphics[width=.44\linewidth]{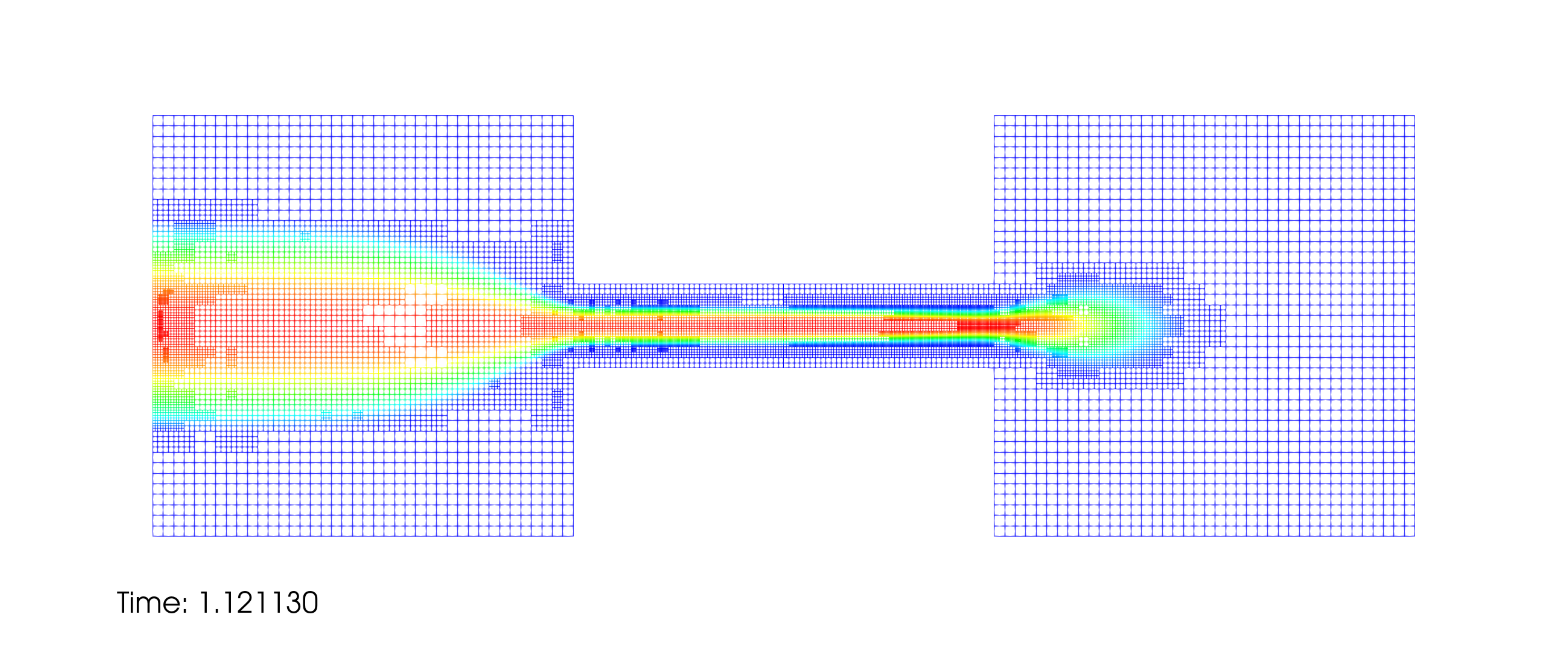}

\includegraphics[width=.44\linewidth]{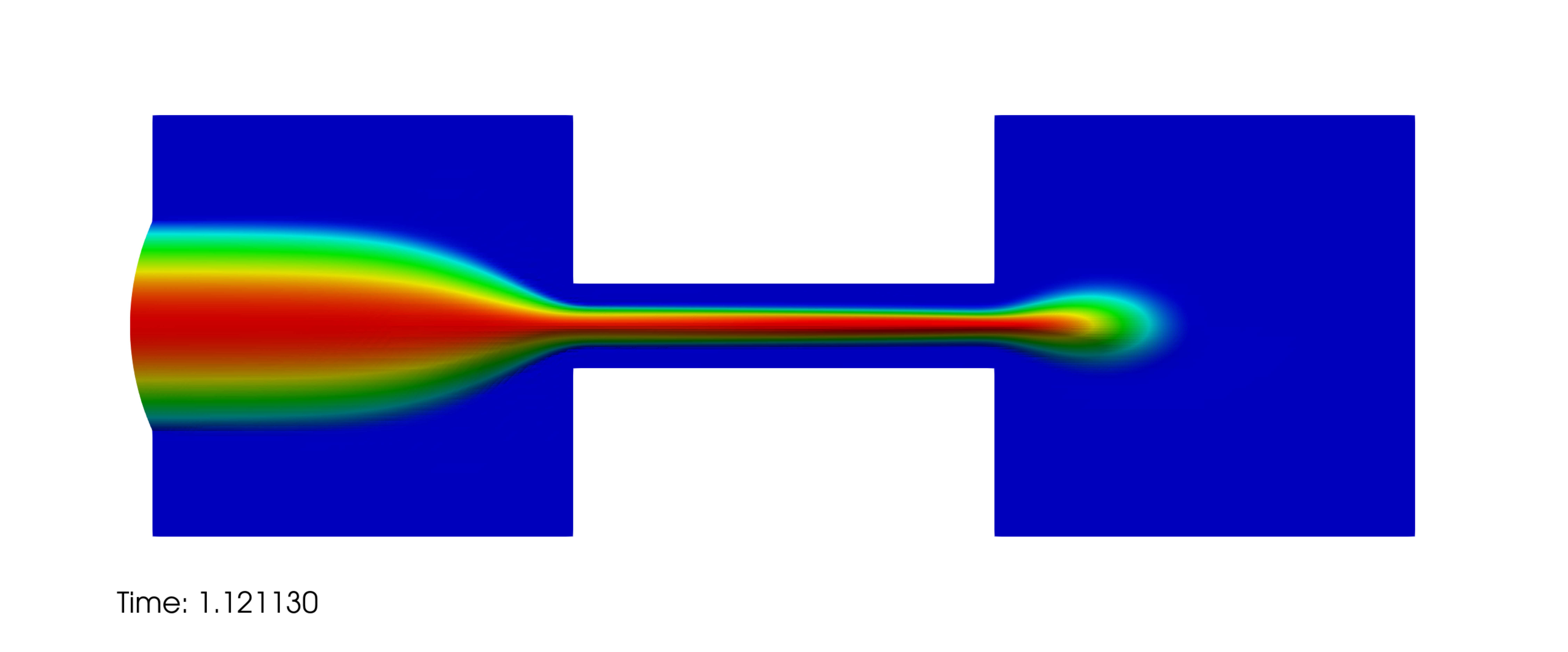}
~~
\includegraphics[width=.44\linewidth]{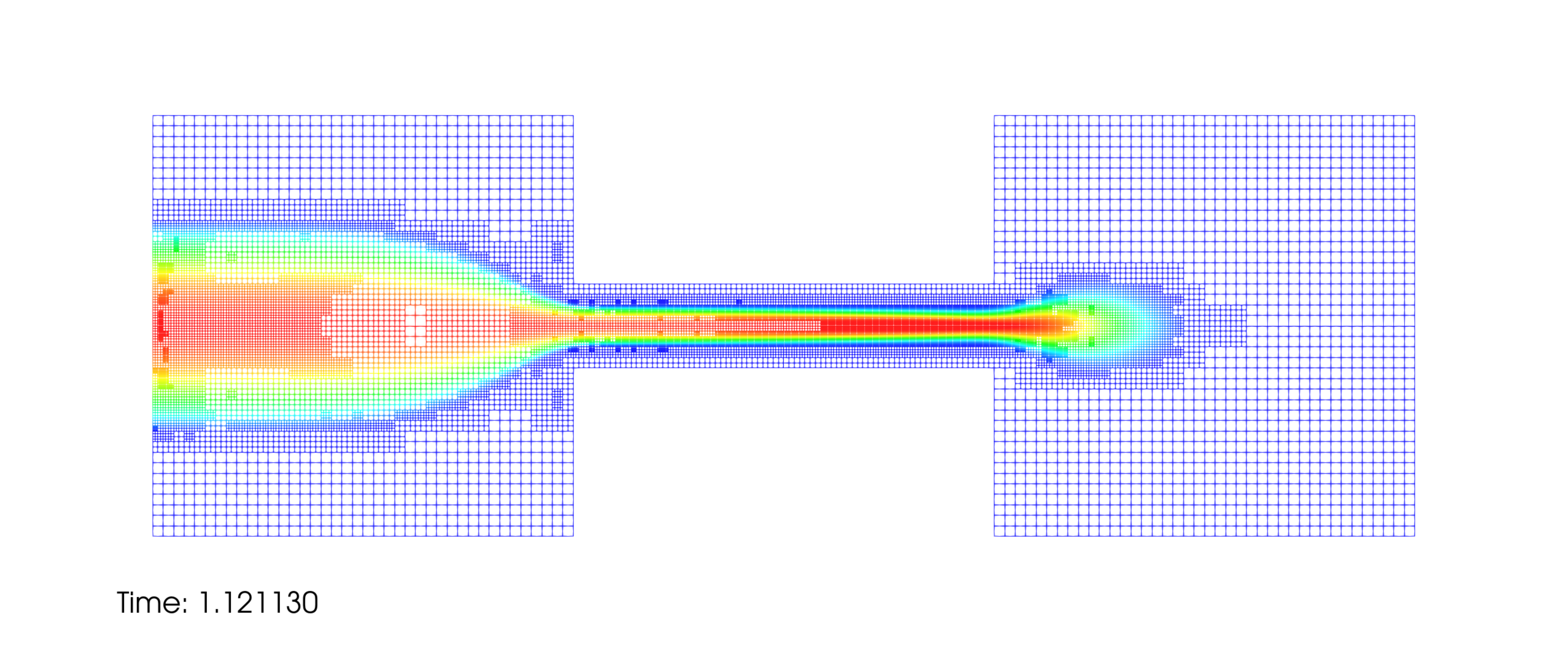}

\caption{Capturing of spurious oscillations with goal-oriented adaptivity
illustrated by comparative solution profiles and corresponding meshes of
the loops $\ell=1-7$ for Sec.~\ref{sec:5:2}.}
\label{fig:7:ex2:osc}
\end{figure}

\begin{table}[ht]
\centering

\begin{tabular}{c | rr|r | c}
\hline
\hline
$\ell$ & $N$ & $N_K^{\text{tot}}$ & $N_K^{\text{max}}$ &
$||u_{\tau h}^{1,1}||$\\
\hline
1  & 25 & 88000   & 3520 & 0.0845694\\
2  & 26 & 108380  & 5008 & 0.0843620\\
3  & 27 & 153324  & 8464 & 0.0844429\\
4  & 28 & 203524  & 10792 & 0.0845339\\
5  & 29 & 298952  & 17680 & 0.0846325\\
6  & 30 & 459696  & 25996 & 0.0847412\\
7  & 31 & 649852  & 38212 & 0.0848211\\
8  & 32 & 1058456 & 68344 & 0.0848755\\
9  & 41 & 1881548 & 100744 & 0.0849878\\
\hline
\end{tabular}

\caption{Goal-oriented temporal and spatial refinements for Sec.~\ref{sec:5:2}.
$\ell$ denotes the refinement level loop,
$N$ the accumulated total cells in time,
$N_K^{\text{tot}}$ the accumulated total cells in space,
$N_K^{\text{max}}$ the maximal number of cells on a slab and
$||u_{\tau h}^{1,1}||$ the value of the goal-functional.
}
\label{tab:9:ex2}
\end{table}

The solution profiles and corresponding adaptive meshes of the primal solution
$\concentration_{\tau h}^{1,1}$ of the loop $\ell=5$
for $t=0.82$, $t=1.32$ and $t=2.28$ are given by Fig.~\ref{fig:6:ex2:loop5}.
In Fig.~\ref{fig:7:ex2:osc} we present a comparative study of the solution
profile and corresponding meshes for $t=1.21$ over the adaptivity loops.
For $\ell=1,2,3$ obvious spurious oscillations in the left square are
existing, which are captured and resolved by the goal-oriented adaptivity
by taking spatial mesh refinements next to the right boundary of the left square.
For $\ell>3$ the spatial refinements capture especially the solution profile
fronts with strong gradients with a focus on the high-convective middle of
the spatial domain.
The refinement in space and time is automatically balanced due to the dynamic
choice of $\theta_h^\textnormal{top}$ and $\theta_\tau^\textnormal{top}$
given by \eqref{eq:5:3:balancing} and is illustrated by Tab.~\ref{tab:9:ex2}.
Precisely, the space-time mesh updates for the loops 2-8 are dominated by
spatial refinements while a significant growth of the time elements can be
recognized for the 9th loop.

\section{Summary}
\label{sec:6:summary}

In this work we presented a space-time adaptive solution algorithm
for SUPG stabilized finite element approximations of a
convection-dominated transport problem that is coupled with a flow problem.
The convection-dominance puts further facets of complexity and sensitivity on
the a posteriori error control, but also illustrates the efficiency and potential
of automatic mesh adaptation. The underlying approach is based on the Dual
Weighted Residual method for goal-oriented error control. A splitting of
the discretization errors in space and time is used for the transport
problem which is then used for the respective mesh adaptation process in
the form of underlying error indicators $\eta_\tau$ and $\eta_h$, respectively.
A discontinuous Galerkin method dG($r$) with an arbitrary polynomial degree
$r \geq 0$ is applied for the discretization in time of the transport problem.
The weights of the DWR adaptivity process are approximated by higher-order 
finite elements instead of using patch-wise higher-order extrapolation.
In numerical experiments we could prove that spurious oscillations that
typically arise in numerical approximations of convection-dominated problems
could be reduced significantly. Effectivity indices close to one were
obtained for small diffusion coefficients corresponding to high P\'eclet numbers.
Moreover, the potential of the approach was illustrated for a problem of practical
interest. Along with the underlying software platform, even more sophisticated
techniques and settings including multirate approximations, varying meshes for
flow and transport or coupling with time-dependent flow problems become feasible.

\section*{Acknowledgements}
U. K\"ocher was partially supported by the Oden Institute for Computational
Engineering and Sciences, University of Texas at Austin, Texas, USA
as long-term guest visitor of M.F.~Wheeler for the implementation of the used
space-time-slab finite element handler.




%




\end{document}